%% file: tunel_number_degeneration_prime_knots.tex
\documentclass{amsart}

\usepackage{amsmath,amsthm,amsfonts, eucal, amssymb}
\usepackage[english]{babel}
\usepackage{graphicx,pdflscape}
\usepackage{layout}
\usepackage[all]{xy}
\usepackage{hyperref}
\usepackage[font=small,labelsep=none]{caption}

\newtheorem{lem}{Lemma}[section]
\newtheorem{claim}{Claim}[lem]
\newtheorem{thm}{Theorem}
\newtheorem{cor}{Corollary}[thm]
\newtheorem{prop}{Proposition}

\theoremstyle{definition}
\newtheorem{defn}{Definition}

\theoremstyle{remark}
\newtheorem{rem}{Remark}

\newcommand{\mc}{\mathcal}

\title{Tunnel number degeneration\\ under the connected sum of prime knots}
\author{Jo\~{a}o Miguel Nogueira}
\address{ CMUC, Department of Mathematics\\ 
University of Coimbra\\
Apartado 3008 EC Santa Cruz\\
3001-501 Coimbra\\
Portugal}
\email{nogueira@mat.uc.pt}

\date{}

\advance\textwidth 2cm
\advance\oddsidemargin  -1cm
\advance\evensidemargin  -1cm

\begin{document}

\maketitle

\begin{abstract}
We study 2-string free tangle decompositions of knots with tunnel number two.  As an application, we construct infinitely many counter-examples to a conjecture in the literature stating that the tunnel number of the connected sum of prime knots doesn't degenerate by more than one: $t(K_1\#K_2)\geq t(K_1)+t(K_2)-1$, for $K_1$ and $K_2$ prime knots.
\end{abstract}


\section{Introduction}
\input{Introduction}

\section{Preliminaries}\label{preliminairies}
\input{Preliminaries}

\section{Outermost disks over ball components of $V-V\cap S$}\label{over balls}
\input{Outermost_disks_over_ball_components}

\section{Outermost disks over torus components of $V-V\cap S$}\label{over torus}
\input{Outermost_disks_over_torus_components}

\section{Outermost disks over components of $V-V\cap S$ when $n_1=3$}\label{n1=3}
\input{Outermost_disks_when_n1_is_3}

\section{Outermost disks over components of $V-V\cap S$ when $n_1=4$}\label{n1=4}
\input{Outermost_disks_when_n1_is_4}

\section{Proof of Theorem \ref{2-string tangle}}\label{proof of theorem 1}
\input{2-string_tangle}

\section{On the tunnel number degeneration under the connected sum of prime knots}\label{counter-example}
\input{tunnel_number_degeneration}

\end{document}

%% file: Introduction.tex
Given a knot $K$ in $S^3$, an \textit{unknotting tunnel system} for $K$ is a collection of arcs $t_1, t_2, \ldots, t_n$, properly embedded in the exterior of $K$, with the complement of a regular neighborhood of $K\cup t_1\cup\cdots\cup t_n$ being a handlebody\footnote{Every knot has an unknotting tunnel system obtained from the knot exterior triangulation.}. The minimum cardinality of an unknotting tunnel system for a knot $K$ is a knot invariant, referred to as the \textit{tunnel number} of $K$ and is denoted by $t(K)$.\\

A natural question of study on knot invariants is their behavior under the connected sum of knots. In the particular case of the tunnel number, it is known, by Norwood \cite{Norwood}, that tunnel number one knots are prime. This result is now consequence of more general work. For instance, Scharlemann and Schultens prove in \cite{Scharlemann-Schultens2} that the tunnel number of the connected sum of knots is bigger than or equal to the number of summands: $$t(K_1\# \cdots \#K_n )\geq n,$$ where $K_1\# \cdots \#K_n $ represents the connected sum of the knots $K_1,\ldots , K_n$. Also, in \cite{Gordon-Reid} Gordon and Reid prove that tunnel number one knots are, in fact, $n$-string prime\footnote{A knot is $n$-string prime if it has no $n$-string essential tangle decomposition. For definitions on tangle decompositions see Definitions \ref{n-string tangle} and \ref{n-string tangle decomposition} in section \ref{preliminairies} (Preliminaries).} for any positive integer $n$.

On the tunnel number behavior under connected sum, it is a consequence from the definition of connected sum of knots that for two knots $K_1$ and $K_2$ in $S^3$ we have: $$t(K_1\#K_2)\leq t(K_1)+t(K_2)+1.$$ For some time the only examples known had an additive behavior: $$t(K_1\#K_2)=t(K_1)+t(K_2).$$
However, in the early nineties, Morimoto \cite{Morimoto1} constructed connected sum examples of prime knots $K_1$ with $2$-bridge knots $K_2$ whose tunnel number degenerates by one\footnote{In \cite{Morimoto3}, without mentioning it, Morimoto gives also the first examples of knots that when connected sum with themselves the tunnel number degenerates (by one): all tunnel number two $3$-bridge knots with a $2$-string free tangle decomposition (as the knot $10_{149}$ from Rolfsen's list in \cite{Rolfsen}).}: $$t(K_1\#K_2)=t(K_1)+t(K_2)-1.$$
Shortly afterwards,  Moriah and Rubinstein in \cite{Moriah-Rubinstein}, and also independently Morimoto, Sakuma and Yokota in \cite{Morimoto-Sakuma-Yokota}, gave examples of knots with supper-additive behavior: $$t(K_1\#K_2)= t(K_1)+t(K_2)+1.$$
Furthermore, about the same time, Kobayashi in \cite{Kobayashi} constructed examples of knots that degenerate arbitrarily under connected sum: for any positive integer $n$, there are knots $K_1$ and $K_2$ where $$t(K_1\#K_2)\leq t(K_1)+t(K_2)-n.$$ However, Kobayashi's examples to show arbitrarily hight degeneration of the tunnel number under connected sum are composite knots.\\

In this paper we study further the tunnel number degeneration under connected sum of prime knots. For this study we use the work of Morimoto in \cite{Morimoto3} that relates $n$-string free tangle decompositions of knots and high tunnel number degeneracy under the connected sum of prime knots. Within this setting, we study $2$-string free tangle decompositions of knots with tunnel number two and we obtain Theorem \ref{2-string tangle}, and its Corollary \ref{t(K)=3}, for which statement we need the following definition.


\begin{defn}
Let $s$ be a properly embedded arc in a ball $B$. Suppose the knot obtained by capping off $s$ along $\partial B$ has tunnel number at most one. We say that $s$ is \textit{$\mu$-primitive} if there is an unknotted arc $t$ properly embedded in $B$, disjoint from $s$, such that the tangle $(B, s\cup t)$ is free.
\end{defn}

\begin{rem}\label{rem 1}
Note that a string $s$ is $\mu$-primitive if, and only if, the knot obtained by capping off $s$ along $\partial B$ is the unknot or a $\mu$-primitive tunnel number one knot\footnote{For a definition of \textit{$\mu$-primitive} knot see Definition 5.13 of the survey paper \cite{Moriah} by Moriah.}.
\end{rem}

\begin{thm}\label{2-string tangle}
Let $K$ be a tunnel number two knot with a $2$-string free tangle decomposition.
Then both strings of some tangle are $\mu$-primitive.\footnote{The correspondent result to Theorem \ref{2-string tangle} for links exists and is also proved by the author in \cite{Nogueira}.}
\end{thm}

\begin{cor}\label{t(K)=3}
Let $K$ be a knot with a $2$-string free tangle decomposition where no tangle has both strings being $\mu$-primitive. Then $t(K)=3$.
\end{cor}

The only examples of prime knots whose tunnel number degenerates under connected sum are the ones given by Morimoto, and in this case the tunnel number only degenerates by one. Also, Kobayashi and Rieck in \cite{Kobayashi-Rieck}, and also Morimoto in \cite{Morimoto5}, proved that the tunnel number of the connected sum of $m$-small\footnote{A knot is said $m$-small if there is no incompressible surface with meridional boundary components in its complement.} knots doesn't degenerate.  With this and other results in perspective, Moriah conjectured in \cite{Moriah} that the tunnel number of the connected sum of prime knots doesn't degenerate by more than one: $t(K_1\#K_2)\geq t(K_1)+t(K_2)-1$, for $K_1$ and $K_2$ prime knots.\\
In this paper, we construct infinitely many counter-examples to this conjecture as in Theorem \ref{counter-example Moriah conjecture} and its Corollary \ref{tunnel number degeneration}.

\begin{thm}\label{counter-example Moriah conjecture}
There are infinitely many tunnel number three prime knots $K_1$ such that, for any $3$-bridge knot $K_2$, $t(K_1\# K_2)\leq 3$.  
\end{thm}

\begin{cor}\label{tunnel number degeneration}
There are infinitely many prime knots $K_1$ and $K_2$ where $$t(K_1\#K_2)=t(K_1)+t(K_2)-2.$$
\end{cor}

In \cite{Scharlemann-Schultens}, Scharlemann and Schultens introduced the concept of degeneration ratio for the connected sum of two prime knots, $K_1$ and $K_2$: $$d(K_1, K_2)=\frac{t(K_1)+t(K_2)-t(K_1\#K_2)}{t(K_1)+t(K_2)}.$$
If the knots $K_1$ and $K_2$ behave additively we have $d(K_1, K_2)=0$.\\
In case the knots $K_1$ and $K_2$ have supper-additive behavior then $-\frac{1}{2}\leq d(K_1, K_2)<0$. The minimum is achieved by the examples of Morimoto, Sakuma and Yokota in \cite{Morimoto-Sakuma-Yokota}. From the examples of Moriah and Rubinstein in \cite{Moriah-Rubinstein} we can choose a sequence of pairs of prime knots $(K_1, K_2)$, with super-additive behavior, where $d(K_1, K_2)$ converges to zero.\\
For the sub-additive behavior of the tunnel number, the degeneration ratio is not so well understood. Naturally $d(K_1, K_2)>0$, and from Corollary 9.2 in \cite{Scharlemann-Schultens}\footnote{In \cite{Schirmer}, Schirmer proves that if $K_1$ is a $m$-small knot then $t(K_1 \# K_2)\geq \max{t(K_1), t(K_2)}$, which implies that for these knots $d(K_1, K_2)\leq \frac{1}{2}$.}, $d(K_1, K_2)\leq\frac{3}{5}$. The examples of Morimoto in \cite{Morimoto1} have degeneration ratio $\frac{1}{3}$. The examples from Corollary \ref{tunnel number degeneration} have degeneration ratio $\frac{2}{5}$. If $K_1$ is a knot as in the statement of the Theorem \ref{counter-example Moriah conjecture} and $K_2$ is any $3$-bridge knot with tunnel number one, from the main theorem of Morimoto in \cite{Morimoto2}, $t(K_1\#K_2)=3$. Hence, the degeneration ratio for these knots is $\frac{1}{4}$. So, for sub-additive behavior, from the results in this paper we have the lowest known degeneration ratio for the connected sum of prime knots\footnote{From work of Morimoto in \cite{Morimoto3}, there are pairs of prime knots with tunnel number two that also have degeneration ratio of $\frac{1}{4}$. In this paper, the same degeneration ratio is obtained with a tunnel number three knot and a tunnel number one knot.\\}, $\frac{1}{4}$, and also the highest, $\frac{2}{5}$.\\

The proof of Theorem \ref{counter-example Moriah conjecture} is a consequence of Morimoto's work in \cite{Morimoto3} and Theorem \ref{2-string tangle}, and is explained in Section \ref{counter-example}. For the proof of Theorem \ref{2-string tangle}, we present the setting in Section \ref{preliminairies}. In Sections \ref{over balls} and \ref{over torus} we prove some auxiliary technical lemmas that are used along the paper. In Sections \ref{n1=3} and \ref{n1=4} we present the main lemmas that together give an outline of the proof. And finally in Section \ref{proof of theorem 1} we organize all the information to prove Theorem \ref{2-string tangle}. For this proof, new and deeper arguments of innermost-arc type are developed to study the $2$-string free tangle decomposition of $K$ with respect to a minimal unknotting tunnel system of $K$.

\begin{center}Acknowledgment\end{center}
I would like to express deep gratitude to my thesis advisor, Cameron Gordon, for the excellent support and guidance. I also thank John Luecke for several useful discussions on the subject of this paper, and The University of Texas at Austin for the outstanding institutional support.

%% file: Preliminaries.tex
The subject of this paper is mainly on tunnel number two knots with $2$-string free tangle decompositions. So, we introduce the following definitions.

\begin{defn}\label{n-string tangle}
A \textit{$n$-string tangle}, $(B, \mathcal{T})$, is collection $\mathcal{T}$ of $n$ mutually disjoint  arcs properly embedded in a ball $B$.\\ 
We say that a tangle $(B, \mathcal{T})$ is: \textit{essential} (resp., \textit{inessential}), if the planar surface $\partial B-\partial \mathcal{T}$ is essential (resp., inessential) in $B-\mathcal{T}$; \textit{free}, if $\pi_1(B-\mathcal{T})$ is a free group or, equivalently, if the complement of a regular neighborhood of $\mathcal{T}$ in $B$ is a handlebody; \textit{trivial}, if $\mathcal{T}$ can be ambient isotoped into $\partial B$.\\
A string $s$ of $(B, \mathcal{T})$ is said to be: \textit{trivial} if there is an embedded disk in $B$, disjoint from $\mathcal{T}-s$, with boundary $s$ and an arc in $\partial B$; \textit{unknotted} if it is trivial in the tangle $(B, s)$.\\
The tangle $(B, \mathcal{T})$ is said to be the \textit{product tangle} with respect to a disk $D$ in $\partial B$, if $\mathcal{T}$ can be isotoped to a tangle $(B, p_1\times I, \ldots, p_n\times I)$, with $B=D\times I$, for a collection of $n$ points, $p_1, \ldots, p_n$, in $int D$.
\end{defn}

\begin{defn}\label{n-string tangle decomposition}
Consider a $2$-sphere $S$ in $S^3$ in general position with a knot $K$, bounding the balls $B_1$ and $B_2$. Let $\mathcal{T}_i=B_i\cap K$, for $i=1,2$. If $\mathcal{T}_i$ is a collection of $n$ arcs we say that $S$ defines a \textit{$n$-string tangle decomposition} of $K$: $(S^3, K)=(B_1, \mathcal{T}_1)\cup_S(B_2, \mathcal{T}_2)$.\\
If both tangles $(B_i, \mathcal{T}_i)$ are free (resp., essential), we say that the tangle decomposition of $K$ defined by $S$ is \textit{free} (resp., \textit{essential}); if the tangle decomposition of $K$ defined by $S$ is not essential, we say that it is \textit{inessential}.\\
If the tangles $(B_i, \mathcal{T}_i)$, for $i=1, 2$, are trivial then we say that $K$ has an $n$-bridge decomposition, defined by $S$.\\
Two tangle decompositions of a knot $K$ defined by the $2$-spheres $S$, $S'$ are said to be \textit{isotopic}, if there is an ambient isotopy of $S \cup K$ to $S'\cup K$.
\end{defn}

Let $K$ be a tunnel number two knot in $S^3$ with a $2$-string essential\footnote{Along the following sections we are assuming the tangle decomposition is essential. The inessential case is observed in the proof of Theorem \ref{2-string tangle}.} free tangle decomposition defined by the $2$-sphere $S$. We represent this tangle decomposition by $(S^3,K)=(B_1, \mc{T}_1)\cup_S (B_2, \mc{T}_2)$. As the tangles are free, their strings have no local knots\footnote{We say that a tangle $(B, \mathcal{T})$ contains a \textit{local knot}, if there is a ball in $B$ intersecting a single string of $\mathcal{T}$ at a knotted arc.}. This property and the next lemma will be frequently used along this paper.

\begin{lem}\label{no parallel strings}
The two strings of a $2$-string essential free tangle are not parallell\footnote{We say that the strings of a tangle $(B; s_1, s_2)$ are \textit{parallel} if there is an embedded disk $D$ in $B$ with boundary the strings $s_1\cup s_2$ and two arcs in $\partial B$ connecting the ends of these strings.}.
\end{lem}
\begin{proof}
Let $(B, s_1\cup s_2)$ be a $2$-string essential free tangle. Suppose that $s_1$ and $s_2$ are parallel, and let $D$ be a disk in $B$ with boundary the strings $s_1\cup s_2$ and two arcs in $\partial B$ connecting the ends of these strings. As $s_1$ and $s_2$ are parallel, from Theorem 1' of \cite{Gordon}, the strings are knotted in $B$.\\
Let $N$ be a regular neighborhood of $D$ in $B$. Hence, $N$ is a regular neighborhood of  $s_1$ and of $s_2$. We have that $B-int\,N$ is embedded in $B-s_1\cup s_2$ and $\partial N$ is a proper essential surface in $B-s_1\cup s_2$. So, $\pi_1(B-int N)$, that is not a free group, injects into $\pi_1(B-(s_1\cup s_2))$, that is free, which is a contradiction. So, the strings $s_1\cup s_2$ are not parallel.
\end{proof}

As in the statement of Theorem \ref{2-string tangle}, we want to prove that the two strings of $(B_1, \mathcal{T}_1)$ or $(B_2, \mathcal{T}_2)$ are $\mu$-primitive. With this purpose, it is useful to consider the following characterization of $\mu$-primitive string.

\begin{lem}\label{mu-primitive characterization}
Let $s$ be a string properly embedded in a ball $B$. Then $s$ is $\mu$-primitive if and only if $s$ is trivial in a solid torus $T$ in $B$ intersecting $\partial B$ in a single disk and whose complement in $B$ is also a solid torus.
\end{lem}
\begin{proof}
Assume $s$ is $\mu$-primitive in $B$. Then there is a trivial string $t$ in $B$, disjoint from $s$ and where $(B, s\cup t)$ is a free tangle. Let $T'=B-int N(t)$\footnote{For a manifold $X$ smoothly embedded in the manifold $Y$, we denote by $N(X)$ the regular neighborhood of $X$ in $Y$.}. As $t$ is trivial in $B$ we have that $T'$ is a solid torus and, from Theorem 1' in \cite{Gordon}, $s$ is trivial in $T'$. Consider the annulus $A=\partial B\cap \partial T'$. Let $D'$ be a disk in $A$ that is a regular neighborhood of an arc in $A$ connecting the two boundary components of $A$. We have that $A-int D'$ is also a disk $D$. Consider a regular neighborhood of $D'$ in $T'$ and isotope $\partial T'$, along the neighborhood of $D'$, away from $D'$. We are left with a solid torus $T$ in $B$, intersecting $\partial B$ at the disk $D$. Furthermore, the complement of $T$ in $B$ is also a solid torus, it is a $1$-handle attached to a ball, and $s$ is trivial in $T$.\\
Assume now that $s$ is a trivial string in a solid torus $T$ in $B$ intersecting $\partial B$ in a single disk and whose complement in $B$ is also a solid torus. Take a meridian disk $L$ of the complement of $T$ in $B$ not intersecting $S$. Add the $2$-handle with core $L$ to $T$. We have that $R=N(L)\cup T$ is a ball intersecting $\partial B$ in a single disk. So, the complement of $R$ in $B$ is a ball. We isotope $\partial R$ to $\partial B$ along this ball, and from $T$ we obtain the solid torus $T'$, and from the disk $L$ we obtain the disk $L'$. We have that $\partial T'\cap \partial B$ is an annulus and the complement of $T'$ in $B$ is the cylinder $N(L')$, where $N(L')$ intersects $\partial B$ in two disks. Let $t$ be the co-core arc of $N(L')$. Hence, as $T'$ is a solid torus, $t$ is a trivial string in $B$. Also, $N(t)=N(L')$ and $s$ is trivial in the complement of $N(t)$. Therefore, $(B, t\cup s)$ is a free tangle, and $s$ is $\mu$-primitive.
\end{proof}

Consider an unknotting tunnel system of $K$, $\{t_1,t_2\}$, and the respective union of regular neighborhoods to be $V=N(K)\cup(N( t_1)\cup N( t_2))$. So, $W=S^3-int V$ is a handlebody and $S^3=V\cup W$ is a genus three Heegaard decomposition of $S^3$. Taking $K\cup t_1\cup t_2$ in general position with respect to $S$, we can assume that $S\cap V$ is a collection of essential disks: $S\cap V=D_1^*\cup \cdots \cup D_{n_1}^*\cup D_1\cup\cdots\cup D_{n_2}$, where $D_i^*$ , $i=1,\ldots,n_1$, are the disks of $S\cap V$ intersecting $K$. Let $\mathcal{D}^*=D_1^*\cup \cdots\cup D_{n_1}^*$ and $\mathcal{D}=D_1\cup \cdots\cup D_{n_2}$.

\begin{lem}\label{strings parallel to boundary}
\begin{itemize}
\item[]
\item[(a)] There is no $2$-sphere in $V$ defining a tangle decomposition of $K$ isotopic to the one defined by $S$.
\item[(b)] Let $C$ be a component of $V-V\cap S$ that intersects $K$. Then $C\cap K$ is parallel to the boundary of $C$.
\end{itemize}
\end{lem}
\begin{proof}
As $V=N(K)\cup N(t_1)\cup N(t_2)$ there is an annulus $A$ in $V$ with $\partial A=K\cup b$, where $b$ is a simple closed curve in $\partial V$ in general position with $S\cap V$.
As $K\cap \mathcal{D}^*$ is non-empty, $A\cap S$ is also non-empty. Assume that $|A\cap S|$ is minimal.\\ 
First assume that some arc $\gamma$ of $A\cap S$ has both ends in a string $s$ from the tangle decomposition, and also that $\gamma$ co-bounds a disk in $A$ together with the string $s$, that intersects $S\cap V$ only at $\gamma$. As $\gamma$ is in $S$, we have that $s$ is trivial in the respective tangle decomposition, which contradicts the tangle decomposition being essential.\\ 
 Suppose $A\cap S$ contains a simple closed curve $c$ essential in $A$. Then $K$ is isotopic to $c$. As $c$ is a simple closed curve in $S$ it bounds a disk in $S$. Therefore, in this case, $K$ would be unknotted, which is a contradiction. Therefore, if $c$ is a simple closed curve of $A\cap S$ then $c$ bounds a disk in $A$.\\

(a) Suppose there is a $2$-sphere in $V$ defining a tangle decomposition isotopic to the one defined by $S$, and, abusing notation, denote it also by $S$. Hence, $S\subset V$. So, there cannot be arcs of $A\cap S$ between $K$ and $b$. Let $B$ be the ball bounded by $S$ in $V$. Suppose $A\cap S$ contains some simple closed curve $c$. As observed before, $c$ bounds a disk $D$ in $A$. Suppose that $c$ is an innermost simple closed curve of $A\cap S$ in $A$. Then, $D$ intersects $S$ only at $c$. As $S\subset V$, if $c$ bounds a disk $S$ disjoint from $K$ we can reduce $|A\cap S|$, which is a contradiction to the minimality of $|A\cap S|$. Otherwise, both disks bounded by $c$ in $S$ intersect $K$, which contradicts the surface $S-int N(K)$ being essential in $S^3-N(K)$. Then, $A\cap S$ contains no simple closed curves. Then, from the previous observations, the components of $A\cap B$ are two disks co-bounded by the strings of the tangle in $B$ and two arcs of $A\cap S$. As each disk of $A\cap B$ intersects $S$ only at a single arc in its boundary, we have that both strings of the tangle $(B, B\cap K)$ are trivial, which is a contradiction to the tangle decomposition defined by $S$ being essential.\\
\begin{figure}[htbp]
\centering
\includegraphics[width=0.24 \textwidth]{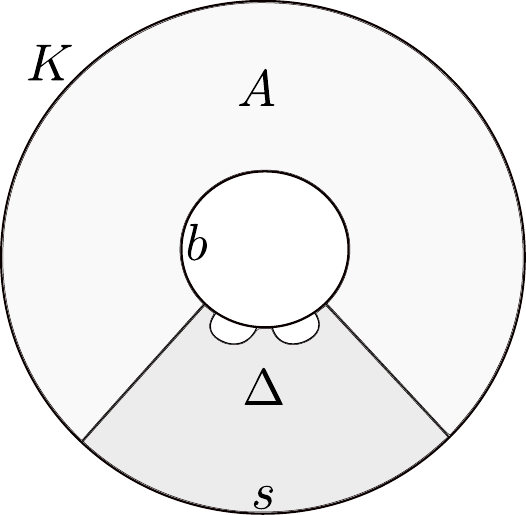}
\caption{: An annulus $A$ in $V$ with boundary being $K$ and a curve $b$ in the boundary of $V$. The disk $\Delta$ is a disk in the component $C$ of $V-V\cap S$, with boundary being the string $s$ and a curve in the boundary of $C$.}
\label{paralleltoboundary.pdf}
\end{figure}

(b) To prove part (b) of this lemma we just need to prove that $A\cap S$ contains no simple closed curves, and that no arc of $A\cap S$ has both ends in ends of strings from the tangle decomposition.\\
Assume now $A\cap S$ contains a simple closed curve, $c$. As observed before, $c$ bounds a disk $D$ in $A$; suppose that it is an innermost curve with this property. Let $L$ be the disk bounded by $c$ in $S\cap V$. If $L$  doesn't intersect $K$ then we can reduce $|A\cap S|$, which contradicts the minimality of $|A\cap S|$. If $L$ intersects $K$ in less than four points then $D$ contradicts the tangle decomposition defined by $S$ being essential. If $L$ intersects $K$ in four points then the tangles decomposition defined by $D\cup L\subset V$ and $S$ are isotopic, which is a contradiction to (a). Then $A\cap (S\cap V)$ contains no simple closed curve.\\
From the previous arguments all arcs of $A\cap S$ either have both ends in $b$ or one end in $b$ and the other at an end of a string in $C$. Also, as $A$ is in general position with $S\cap V$, each string end is attached to a single arc of $A\cap S$. Let $C$ be a component of $V-V\cap S$ that intersects $K$. Then each string $s$ of $C$ belongs to the boundary of a properly embedded disk component of $A-A\cap S$ in $C$, disjoint from the other string components in $C$, as in Figure \ref{paralleltoboundary.pdf}. Therefore, all components of $C\cap K$ are independently parallel to the boundary of $C$, which gives us the statement (b) of the lemma.
\end{proof}

Considering the previous lemma and that all $2$-spheres in $S^3$ intersect $K$ an even number of times, no disk of $\mathcal{D}$ is parallel to a disk of $\mathcal{D}^*$ in $V$.\\

From the work of Ozawa \cite{Oz}, we know that if a knot has an essential $2$-string free tangle decomposition then this decomposition is unique up to \textit{isotopy}, and, furthermore, $K$ is $n$-string prime for $n\neq 2$. (In particular, $K$ is prime.) This is a result frequently throghout this paper, and we refer to it as Ozawa's unicity theorem.\\
We assume the unknotting tunnel system and the tangle decomposition defined by $S$ up to isotopy are such that $S\cap V$ is a collection of disks with minimum cardinality $|S\cap V|$\footnote{For a topological space $X$, $|X|$ denotes the number of connected components of $X$.}. From Lemma \ref{strings parallel to boundary} and the minimality of $|S\cap V|$, we can assume that all disks $S\cap V$ are essential in $V$. As $S$ decomposes $K$ in two $2$-string tangles we have $n_1\leq 4$. If $n_1\geq 3$, we denote the string with one end in $D_i^*$ and the other end in $D_j^*$ by $s_{ij}$.\\

Let $P$ denote the planar surface $S\cap W$. By the minimality of $|S\cap V|$ and the incompressibility of $S-int(N(K))$ in $S^3-int(N(K))$, we have that $P$ is essential in $W$.\\
For a complete system of meridian disks\footnote{A \textit{complete system of meridian disks} of a handlebody $H$ is a collection of disks in $H$ whose complement is a ball.} of $W$, $\{E_1, E_2, E_3\}$, we write $E=E_1\cup E_2\cup E_3$. Considering $E$ and $P$ in general position, we choose $E$ such that $|P\cap E|$ is minimal between the complete systems of meridian disks of $W$.\\
By the incompressibility of $P$ and the minimality of $|P\cap E|$, no component of $P\cap E$ is a simple closed curve. Also, if an arc component of $P\cap E$ is a loop co-bounding a disk in $P$ disjoint from $P\cap E$, using this disk, we can change the complete system of meridian disks of $W$ to $E'$ with $|P\cap E|> |P\cap E'|$. This is a contradiction, and therefore $P\cap E$ is a collection of  essential arcs\footnote{An arc $\alpha$ in $P$ is \textit{essential} if the components closure of $P-\alpha$ doesn't contain any disk component.} in $P$.\\
With the arcs of $P\cap E$ we define a graph in $S$ that we denote by $G_P$: the vertices are the disks from $S\cap V$, each of which corresponds to a boundary component of $P$, and the edges are the arcs $P\cap E$. 
The graph $G_P$ is connected: in fact, if the graph $G_P$ is not connected then by cutting along a complete system of meridian disks of $W$ we can find a compressing disk for $P$ in $W$, which is a contradiction as $P$ is essential.
As $G_P$ is a connected graph in a $2$-sphere, from the arcs $P\cap E$ in $E$ we can create a sequence of isotopies of type A\footnote{See chapter 2 of \cite{Jaco} by Jaco for a definition of an \textit{isotopy of type A}, and section 2 of \cite{O} by Ochiai for a definition of the latter and also of an \textit{inverse isotopy of Type A}. \label{Ochiai}} over a sequence of arcs $\alpha_1, \alpha_2, \ldots, \alpha_m$ of $P\cap E$ such that the closure of the components of $P-\alpha_1\cup \cdots\cup\alpha_m$ is a collection of disks.\\ 

\begin{figure}[htbp]
\centering
\includegraphics[width=0.26\textwidth]{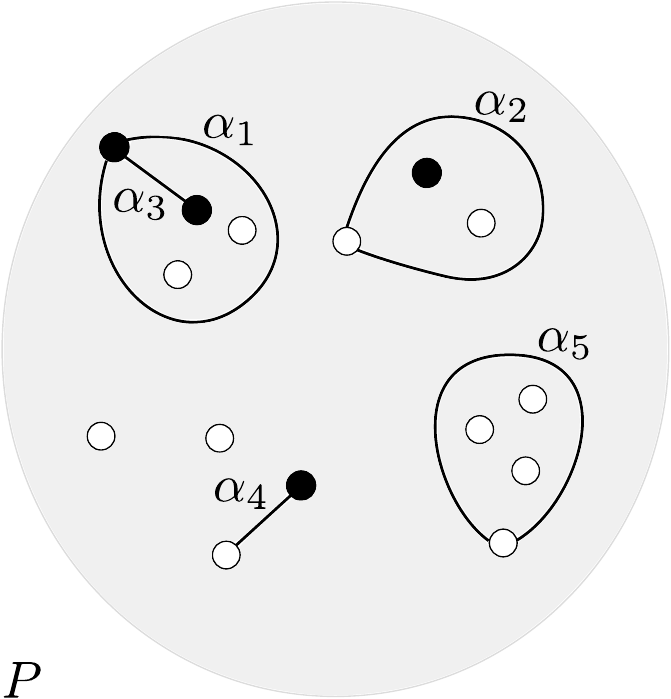}
\caption{: An illustration of some arc components of $E\cap P$ in $P$. The arc $\alpha_1$ (resp., $\alpha_2$) is a sk-arc (resp., st-arc). The arc $\alpha_3$ is a $\text{d}^*$-arc that is also a k-arc, and $\alpha_4$ is both a d-arc and a $\text{d}^*$-arc. Note also that the arc $\alpha_5$ is an example of a type I d-arc that is not a st-arc.}                
\label{typeofarcs.pdf}
\end{figure}

The vertices of $G_P$ associated to $\mathcal{D}$, resp. $\mathcal{D}^*$, are referred to as \textit{d-vertices}, resp. \textit{$\text{d}^*$-vertices}, and are illustrated as white disks, resp. dark disks.
Between the edges of $G_P$ it is useful to define some types of arcs as follows. (See Figure \ref{typeofarcs.pdf}.)
\begin{itemize}
\item[\textit{Type I}:] is an edge connected to a single vertex.
\item[\textit{Type II}:] is an edge connected to two distinct vertices.
\item[\textit{d-arc}:] is an edge with at least one end attached to a d-vertex.
\item[\textit{$\text{d}^*$-arc}:] is an edge with at least one end attached to a $\text{d}^*$-vertex.
\item[\textit{t-arc}:] is an edge of type II, $\alpha_i$, in a sequence of isotopies of type $A$ as above, connected to some d-vertex $D$ and where $\alpha_j$, $j<i$, is disjoint from $D$. (See Remark \ref{t-arc} and also Lemma 1 of \cite{O} by Ochiai.)
\item[\textit{k-arc}:] is a type II arc connecting two $\text{d}^*$-vertices.
\item[\textit{st-arc}:] is a type I d-arc separating $P$ into two planar surfaces, each with some boundary component of $\mathcal{D}^*$.
\item[\textit{sk-arc}:] is a type I $\text{d}^*$-arc separating $P$ into two planar surfaces, each with some boundary component of $\mathcal{D}^*$.
\end{itemize}

\begin{rem}\label{t-arc}
Suppose $\alpha$ is a type II d-arc with one end in the d-vertex associated with $D$. If one of the disks separated by $\alpha$ from $E$ intersects the disk $D$ only at the end of $\alpha$ in $D$, then all arcs of $E\cap P$ in this disk have no end in $D$. This implies that $\alpha$ is a t-arc. (See Figure \ref{typeofarcsinE.pdf}.) In Lemma \ref{property} we prove that such arcs cannot exist.
\end{rem}

\begin{figure}[htbp]
\centering
\includegraphics[width=0.26\textwidth]{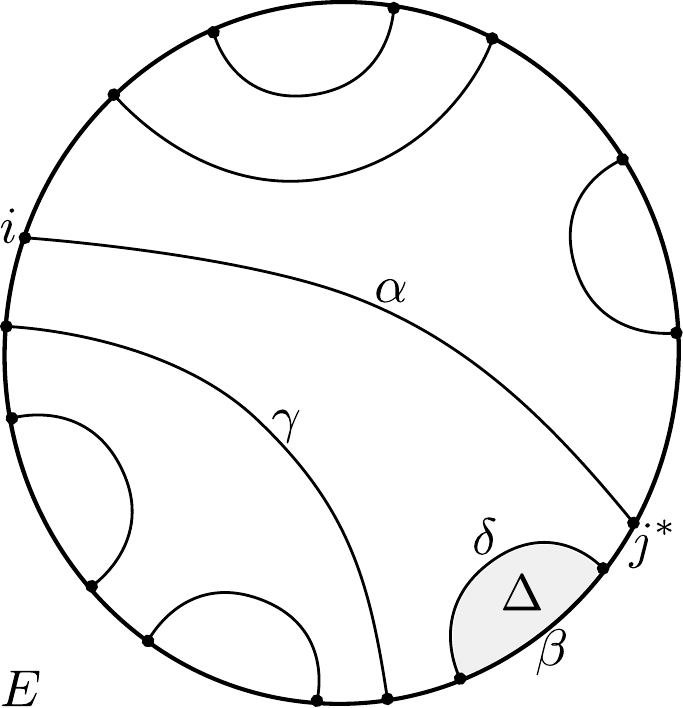}
\caption{: An illustration of arc components of $E\cap P$ in some component of $E$. If an arc of $E\cap P$ in $E$ has one end in the disk $D_i$, resp. $D_j^*$, then we label the end of the arc in $E$ by $i$, resp. by $j^*$. The ends of the arc $\alpha$ in the figure exemplify this notation. If all arcs in one of the disks separated by $\alpha$ from $E$ have no ends being $i$, then $\alpha$ is a t-arc.}
\label{typeofarcsinE.pdf}
\end{figure}

We say that an arc $\delta$ of $P\cap E$ is an \emph{outermost arc}, if $\delta$ separates a disk component $\Delta$ of $E-E\cap S$ from $E$. We have $\Delta \cap P=\delta$ and $\partial\Delta=\delta\cup\beta$ with $\beta\subset\partial E$. The disk $\Delta$ is said to be an \textit{outermost disk}. (See Figure \ref{typeofarcsinE.pdf}.) An outermost disk is said to be \textit{over} a component of $V-V\cap S$ if the correspondent arc $\beta$ is in the (boundary) of the component.\\

In the next lemma, we study the arcs $P\cap E$ in $P$ and in $E$ and obtain properties that give the base setting for these arcs along the work in this paper.

\begin{lem}\label{property}
\begin{itemize}
\item[]
\item[(a)] All outermost arcs are of type I.
\item[(b)] If $n$ is the number of vertices of $G_P$ then either $n=1$ and the graph $G_P$ has no edges, or $n\geq 3$ and at least two vertices of $G_P$ are not adjacent to edges of type I.
\item[(c)] No arc of $E\cap P$ is a t-arc.
\item[(d)] The outermost d-arcs of $E\cap P$ in $E$ are of type I.
\item[(e)] Each type I arc of $E\cap P$ is a st-arc or sk-arc.
\item[(f)] All d-vertices are adjacent to a st-arc.
\item[(g)] Each outermost arc of $E\cap P$ in $E$ is a st-arc or a sk-arc.
\end{itemize}
\end{lem}
\begin{proof}\text{}

(a) Suppose some outermost arc is of type II. Then, proceeding with an isotopy of type A along the respective outermost disk we can reduce $|S\cap V|$, which is a contradiction to the minimality of $|S\cap V|$.

(b) If $n=1, 2$ and there is some loop in $G_P$, the outermost loop co-bounds a disk in $P$. Furthermore, if $G_P$ has no loops and $n=2$ then the outermost arcs of $E\cap P$ in $E$ are of type II, which is a contradiction to (a).
If $n\geq 3$ and at most one vertex is not adjacent to a loop, then one outermost loop co-bounds a disk in $P$.
In both cases we contradict the fact that all edges of $G_P$ are essential in $P$.

(c) If there is a t-arc, then by a sequence of isotopies of type A followed by a sequence of inverse isotopies of type A, as in Lemma 1 of \cite{O} by Ochiai, we can ambient isotope $S$, in the exterior of $K$, to some $2$-sphere $S'$ where $|S\cap V|>|S'\cap V|$. This is a contradiction to the minimality of $|S\cap V|$.

(d) If an outermost d-arc of $E\cap P$ in $E$ is of type II then it is a t-arc, which is a contradiction to (c). Therefore, the outermost d-arcs of $E\cap P$ in $E$ are of type I.

(e) Let $\alpha$ be a type I arc of $E\cap P$. As $\alpha$ is essential in $P$ it separates $P$ into two components that are not disks. If one of these components, say $F$, only contains boundary components corresponding to d-vertices  there is some arc of $E\cap P$ in $F$ that is a t-arc, which contradicts (c).

(f) Assume there is a d-vertex $D$ that is only adjacent to edges of type II. Then there is a t-arc with respect to $D$ (choose an outermost arc, in $E$, between the edges of type II attached to $D$), which is a contradiction to (c). Hence, there is at least one edge of type I with ends in $D$, and from (e) it is a st-arc.

(g) From (a) the outermost arcs are of type I, and from (e) the type I arcs are st or sk-arcs.
\end{proof}

%% file: Outermost_disks_over_ball_components.tex
In this section we study the case when there is an outermost disk over some ball component of $V-V\cap S$, as in Lemma \ref{no beta in ball}. We also have presented other crucial lemmas relating ball components of $V-V\cap S$ and certain disks of $E-E\cap P$, together with the next lemma where we show several properties of tangles obtained from balls in $B_1$ or $B_2$. 

\begin{lem}\label{inner ball}
Suppose there is a ball $Q$ in one of the tangles defined by $S$ that intersects each string of the tangle in a single arc.
\begin{itemize}
\item[(a)] Let $Q^c$ denote the complement of $Q$ in $S^3$. The tangle $(Q^c, Q^c\cap K)$ is essential.
\item[(b)] If one of the strings of $Q\cap K$ is unknotted in $Q$ then either the tangle $(Q, Q\cap K)$ is trivial or some string of some tangle defined by $S$ is unknotted.
\item[(c)] Suppose both strings of the tangle are in $Q$ and have ends in one or two disk components of $Q\cap S$. Then the tangle $(Q, Q\cap K)$ is essential.
\item[(d)] If a ball component of $V-S\cap V$ contains a string with both ends in the same component of $\mathcal{D}^*$, then some string of some tangle is unknotted.
\end{itemize}
\end{lem}
\begin{proof}\text{}\\
Assume that the tangle containing $Q$ is $(B_1, \mathcal{T}_1)$.

(a) Suppose that $(Q^c, Q^c\cap K)$ is inessential. As this tangle contains only two strings, both strings are trivial in it. Let $s'$ be an arc component of $Q^c\cap K$, and $D'$ be a disk in $Q^c$ with interior disjoint from $K$ and with boundary being the union of $s'$ and an arc in $\partial Q^c$. Let $s$ be the string from the tangle decomposition defined by $S$ that is a subset of $s'$. So, $s$ is a string of the tangle $(B_2, \mathcal{T}_2)$. As $s'$ contains only the string $s$ of $K-S\cap K$, we have that $\partial D'$ intersects $S$ only at two points, which are the end points of $s$. Considering a minimal collection $D'\cap S$ and following an innermost curve or arc type of argument, we can prove that $D'\cap S$ is a single arc $a$ with ends being the ends of $s\subset \partial D'$. Let $D$ be the disk in $D'$ with boundary defined by the arcs $a\subset S$ and $s$. Then $D$ is in the tangle $(B_2, \mathcal{T}_2)$ and the interior of $D$ doesn't intersect $S$. Therefore, the string $s$ is trivial in $(B_2, \mathcal{T}_2)$, which is a contradiction to the tangle decomposition defined by $S$ being essential.

(b) Assume that one of the strings of $Q\cap K$ is unknotted in $B$. If the tangle $(Q, Q\cap K)$ is essential then, from (a), the $2$-sphere $\partial Q$ defines a $2$-string essential tangle decomposition of $K$. By Ozawa's unicity theorem, the tangle decompositions given by $S$ and $\partial Q$ are isotopic. Hence, as one string of $(Q, Q\cap K)$ is unknotted, some string of some tangle defined by $S$ is also unknotted. Otherwise, the tangle $(Q, Q\cap K)$ is trivial.

(c) Suppose the tangle $(Q, Q\cap K)$ is trivial. Let $Q'$ be obtained from $Q$ after we isotope away from $S$ in $B_1$ the components of $Q\cap S$ that don't contain any string ends. If $Q'$ intersects $S$ in a disk with all string ends in it, as the strings are trivial in $Q'$ they are both unknotted in $(B_1, \mathcal{T}_1)$. From Theorem 1' in \cite{Gordon}, this is a contradiction to the tangle $(B_1, \mathcal{T}_1)$ being essential. Otherwise, if $Q'$ intersects $S$ in two disks that also contain the strings ends in them. As the tangle $(B_1, \mathcal{T}_1)$ is free, following an argument as in Lemma \ref{no parallel strings}, the complement of $Q'$ in $B_1$ is a solid torus.  Then $\partial Q'$ is ambient isotopic to $S$ in $S^3-K$, which is also a contradiction to the tangle $(B_1, \mathcal{T}_1)$ being essential. So, the tangle $(Q, Q\cap K)$ is essential.

(d) Suppose there is a ball component $C$ of $V-V\cap S$ containing a string $s$ with both ends in the same component of $\mathcal{D}^*$. From Lemma \ref{strings parallel to boundary}, the tangle $(C, C\cap K)$ is trivial. Consequently, the string $s$ is trivial in $C$. As the ends of $s$ are in the same disk of $C\cap S$, it is also unknotted in the respective tangle defined by $S$.
\end{proof}




\begin{lem}\label{no beta in ball}
If there is an outermost disk over a ball component of $V-(S\cap V)$ then some string of some tangle is unknotted.
\end{lem}
\begin{proof}
Suppose there is an outermost disk $\Delta$ over a ball component $C$ of $V-(S\cap V)$, and let $\delta$ be the respective outermost arc. Without loss of generality assume that $C\subset B_1$. Let $A$ be the annulus in the intersection of $C\subset V$ with the $2$-sphere obtained after an isotopy, along a regular neighborhood of $\Delta$, of a regular neighborhood of $\delta$ in $S$ into $V$. The component $C$ is either disjoint from $K$ or contains one or both strings of $ \mathcal{T}_1$. Then, a core of $A$ bounds a disk $D$ in $\partial C$ that is either disjoint or intersects $K$ in one or two points. We isotope $int \,D$ into $C$, slightly, such that $D\cap S=\partial D$.\\

Assume $D$ is disjoint from $K$. The arc $\delta$ union an arc component of $\partial P-\partial \delta$ is a simple closed curve parallel, in $S-K$, to a core of $A$. Also, the arc $\delta$ separates $P$ into two planar surfaces containing boundary components of $\mathcal{D}^*$. Therefore, $D$ separates the strings of the tangle $(B_1, \mathcal{T}_1)$ and intersects $S$ only at $\partial D$, which is a contradiction to the tangle decomposition defined by $S$ being essential.\\ 

Assume that $|D\cap K|=1$. Let $D'$ be the disk in $S$ with $\partial D'=\partial D$ and $|D'\cap K|=1$, and $Q$ be the ball in $B_1$ bounded by $D\cup D'$. Then $Q\cap \mathcal{T}_1$ is a single trivial arc in $Q$. So, considering the $2$-sphere $S'=(S-D')\cup D$, the tangle decompositions defined by $S$ and $S'$ are isotopic with $|S'\cap V|<|S\cap V|$, which is a  contradiction to the minimality of $|S\cap V|$.\\

Assume now that $|D\cap K|=2$. Then $D$ splits the tangle $(B_1, \mathcal{T}_1)$ into two $2$-string tangles: $(B_1', \mathcal{T}_1')$ and $(B_1'', \mathcal{T}_1'')$. If $D$ intersects $K$ in the same string of $\mathcal{T}_1$ then one string of this tangle, say $s_1$, is either in $(B_1', \mathcal{T}_1')$ or in $(B_1'', \mathcal{T}_1'')$. Assume, without loss of generality, that $s_1$ is in $(B_1', \mathcal{T}_1')$. From Lemma \ref{inner ball} (a), if the tangle $(B_1', \mathcal{T}_1')$ is essential then $\partial B_1'$ defines an essential $2$-string tangle decomposition of $K$ with $|\partial B_1'\cap V|<|S\cap V|$, which contradicts the minimality of $|S\cap V|$. Hence, the tangle $(B_1', \mathcal{T}_1')$ is trivial. So, $s_1$ is trivial in $(B_1', \mathcal{T}_1')$ and therefore unknotted in $(B_1, \mathcal{T}_1)$. Otherwise, assume that $D$ intersects $K$ in different strings of $\mathcal{T}_1$. By a similar argument as when $D$ intersects $K$ in the same string we can prove that the tangles $(B_1', \mathcal{T}_1')$ and $(B_1'', \mathcal{T}_1'')$ are trivial. Then the string $s_1\cap B_1'$ is trivial in $B_1'$ and the string $s_1\cap B_1''$ is trivial in $B_1''$, which implies that $s_1$ is unknotted in $(B_1, \mathcal{T}_1)$.
\end{proof}

\begin{rem}\label{remark no beta in ball}
From Lemma \ref{no beta in ball}, if some outermost disk is over a ball component of $V-S\cap V$ then we have Theorem \ref{2-string tangle}. So, we can assume that all outermost disks are over components of $V-S\cap V$ other than balls.
\end{rem}

We say that two arcs of $E\cap P$ are \textit{parallel} in $E$ if the union of these arcs cuts a disk component of $E-E\cap P$ from $E$. An arc outermost in $E$ between the arcs of $E\cap P$ not in a sequence of parallel arcs to a outermost arc is said to be a \textit{second-outermost arc}. A disk of $E-E\cap P$ in the outermost side of a second-outermost arc is called a \textit{second-outermost disk}. The arcs $\alpha$ and $\gamma$ in Figure \ref{typeofarcsinE.pdf} are examples of second-outermost arcs.\\ 

Let $\gamma$ and $\gamma'$ be two type I arcs of $E\cap P$ parallel in $E$ attached to disks $D$ and $D'$ of $S\cap V$, resp., parallel in $V$. (See Figure 4(a).) Denote by $\Gamma$ the disk cut by $\gamma\cup \gamma'$ from $E$, and by $C$ the ball component of $S-S\cap V$ cut by $D\cup D'$ from $V$. Suppose that $C$ and $\Gamma$ are in the same ball component bounded by $S$, say $B_1$. Then $\Gamma$ is a proper surface in the complement of the solid torus $B_2\cup_{D\cup D'} C$ in $S^3$, which is $B_1-int C$. The curve $\partial \Gamma$ is inessential in the boundary of $B_2\cup_{D\cup D'} C$, and as we are in $S^3$, it bounds a disk in $\partial(B_2\cup_{D\cup D'} C)$. Let $L$ be the disk bounded by $\partial \Gamma$ in $\partial (B_1-int C)$. Note that $L$ intersects $S$ in two disks and $C$ in a disk band from $D$ to $D'$.  Let $R$ be the ball bounded by $\Gamma\cup L$ in $B_1-int C$. (See Figure 4(b).) For the next lemma, denote by $q$ a core arc of $C$ in $B_1$, this is a proper arc in $B_1$ with regular neighborhood $C$. This construction and the following lemma will be frequently used throughout this paper.

\begin{figure}[htbp]
\centering
\includegraphics[width=0.83\textwidth]{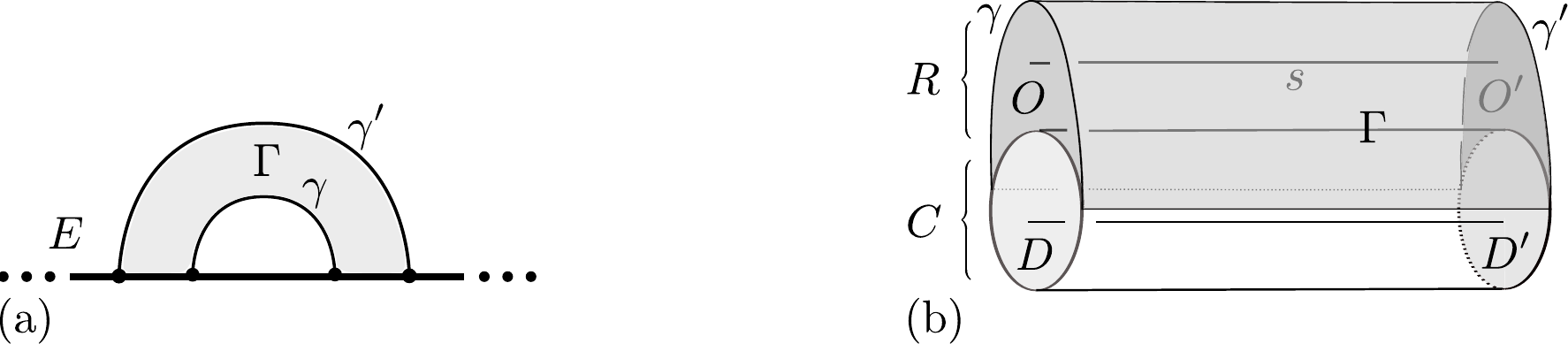}
\caption{}
\label{parallelskarcs.pdf}
\end{figure}

\begin{lem}\label{parallel arcs}
The ball $R$ contains a single string of $\mathcal{T}_1$, and this string is parallel to $q$ in $B_1$.
\end{lem}
\begin{proof}
Denote by $O$ and $O'$ the disks of $L\cap S$, which are the disks cut by $\gamma$ and $\gamma'$, resp., in $S-int (D\cup D')$. As $\gamma$, and $\gamma'$, is a st or sk-arc we have that $O$ and $O'$ contain some component of $\mathcal{D}^*$. This means that $R$ contains some string(s) of $\mathcal{T}_1$.\\
Suppose that $\mathcal{T}_1$ is in $R$. From Lemma \ref{inner ball}(a) and (c), we have that $\partial R$ defines a $2$-string essential tangle decomposition of $K$. As the $2$-string essential tangle decomposition of $K$ is unique, we have that the tangle decompositions defined by $S$ and $\partial R$ are isotopic. But $|\partial R\cap V|<|S\cap V|$, which is a contradiction to the minimality of $|S\cap V|$.\\
Then $R$ contains a single string, $s$, of $\mathcal{T}_1$. As there are no local knots, $s$ is trivial in $R$. The intersection of $R$ with $C$ is the disk $L-O\cup O'$, that intersects each $D$ and $D'$ at an arc. Let $a$ be an arc in $L-O\cup O'$ with one end in $D\cup O$ and other in $D'\cup O'$. Then, $s$ (resp., $q$) is parallel to $a$ in $R$ (resp., $C$) through a disk with boundaries being $s\cup a$ (resp., $a\cup q$) and two arcs in $O\cup O'$ (resp., $D\cup D'$). As $R$ intersects $C$ in $L-O\cup O'$, we have that $s$ and $q$ are parallel through a disk with boundaries being $s\cup q$ and two arcs in $S$. Consequently, $s$ and $q$ are parallel in $B_1$.
\end{proof}

\begin{cor}\label{no parallel sk-arcs}
The disks $D$ and $D'$ cannot be disks of $\mathcal{D}^*$.
\end{cor}
\begin{proof}
As no disk of $\mathcal{D}$ is parallel to a disk of $\mathcal{D}^*$ in $V$, suppose both $D$ and $D'$ are disks of $\mathcal{D}^*$. Then $C$ contains some string(s) of $\mathcal{T}_1$. As $R$ contains a string of $\mathcal{T}_1$ we have that $C$ contains a single string of $\mathcal{T}_1$, and from Lemma \ref{strings parallel to boundary}(b) this string is also a core of $C$. Therefore, from Lemma \ref{parallel arcs}, the strings of $\mathcal{T}_1$ in $R$ and in $C$ are parallel in $B_1$, which is a contradiction to Lemma \ref{no parallel strings}.
\end{proof}

\begin{lem}\label{no gamma with equal ends}
Let $D_k$, $D_i^*$ and $D_j^*$ be disks of $S\cap V$ where $D_k\cup D_i^*\cup D_j^*$ cuts a ball component $C$ of $V-V\cap S$ from $V$; assume that $C$ intersects $K$ at a single string, with one end at $D_i^*$ and the other at $D_j^*$. Suppose there is a disk component of $E-E\cap P$, in the same tangle component as $C$, that intersects $S$ in arcs where all but one of these arcs have both ends in $D_k$, and the remaining arc has either at least one end in $D_k$, or one end in $D_i^*$ and the other in $D_j^*$. Then some string of some tangle is unknotted.
\end{lem}
\begin{proof} 
Denote by $\Gamma$ the disk component of $E-E\cap P$ referred to in the statement, and by $\gamma$ the arc of $\Gamma\cap S$ that doesn't have by assumption both ends in $D_k$, as in Figure \ref{secondoutermostarcs.pdf}(a). Without loss of generality, suppose $C$ is in $B_1$. Let $s$ and $s'$ be the strings in this tangle with $s$ in $C$, and $C_i$ denote the cylinder obtained from $C$ after an isotopy pushing $D_j^*$ away from $S$ in $B_1$. Consider also the solid torus $T_i$ defined by $B_2\cup_{D_i^*\cup D_k} C_i$.\\
\begin{figure}[htbp]
\centering
\includegraphics[width=0.86\textwidth]{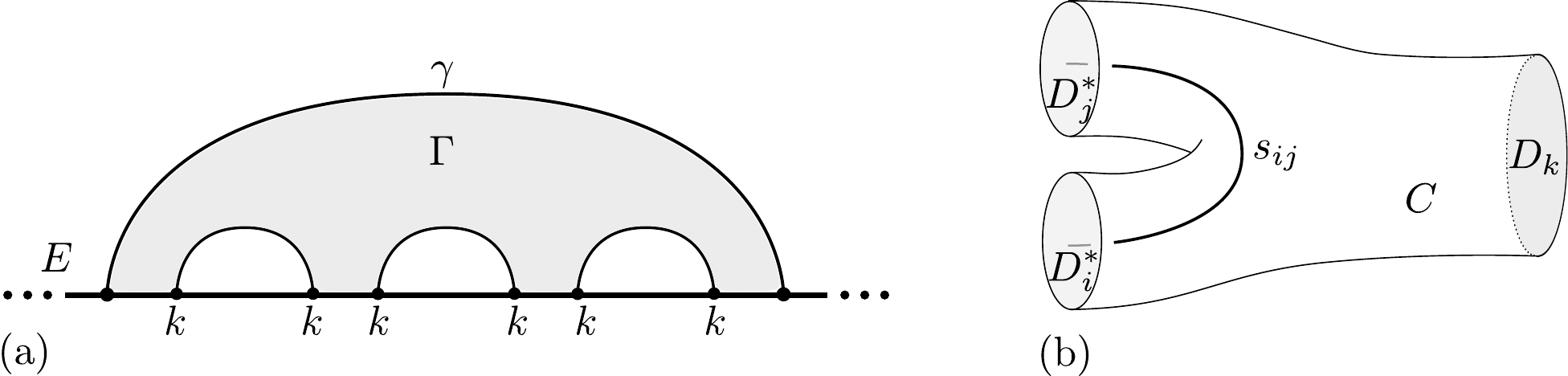}
\caption{: (a) Arc $\gamma$ after the outermost arcs attached to $D_k$. The label $k$ at an end of an arc means the end is at the disk $D_k$. (b) The ball $C$ cut by $D_i^*\cup D_j^*\cup D_k$ from $V$, and the string $s_{ij}$ in it.}
\label{secondoutermostarcs.pdf}
\end{figure}
Assume that $\gamma$ also has both ends in $D_k$.\\
The curve $\partial \Gamma$ is inessential in $T_i$ and it bounds a disk in $\partial T_i$ that we denote by $L$. The disk $L$ intersects $\mathcal{D}^*$. In fact, suppose $L$ is disjoint from $\mathcal{D}^*$, and consider a disk $D$ of $L\cap C_i$ or $L\cap S$ with boundary intersecting $\partial L$ at a single component. Then, if $D\subset C_i$ it is also a disk in $C$ and we get a contradiction to the minimality of $|P\cap E|$, and if $D\subset S$ we obtain a contradiction with Lemma \ref{property}(e).\\ 
Consider the ball $R$ in $B_1$ bounded by $\Gamma\cup L$. As $L$ contains some component of $\mathcal{D}^*$ the ball $R$ intersects $\mathcal{T}_1$; it contains at most two arcs, the string $s'$ or a portion of the string $s$.\\
Suppose $R$ contains the string $s'$ only. As there are no local knots in the tangle, $s'$ is parallel to $L$. By pushing $L$ to $S$ from $\partial C_i$ we have that the string $s'$ is unknotted in $(B_1, \mathcal{T}_1)$.\\
Suppose $R$ contains also a portion of the string $s$. From Lemma \ref{inner ball}(a), we have that the tangle $(R^c, R^c\cap K)$, where $R^c$ is the complement of $R$ in $S^3$, is essential. As $|\partial R\cap V|<|S\cap V|$, if the tangle $(R, R\cap K)$ is essential, we have a contradiction to the minimality of $|S\cap V|$. Therefore, $(R, R\cap K)$ is a trivial tangle. Then, the string $s'$ is unknotted in $R$ and parallel to the disk $L$. By an isotopy of $L$ from $T_i$ to $S$ we have that the string $s'$ is also unknotted in $(B_1, \mathcal{T}_1)$.\\
So, we can assume that $R\cap K$ is only a portion of the string $s$. Consider the solid torus $T_i'=T_i\cup R$, and the complement in $B_1$ of the ball obtained by cutting $T_i'$ along $D_i^*$ that we denote by $Q$. Then, $Q$ is a ball in $B_1$ containing $s'$ and a portion of $s$. The $2$-sphere $\partial Q$ is isotopic to $S$ rel. $Q\cap S$ in $B_1$. Then, if $s$ is unknotted in $Q$ it is also unknotted in $B_1$. As $|\partial Q\cap V|<|S\cap V|$, following a similar reasoning as when $R$ contains two arcs, we also have that some string of some tangle is unknotted.

Assume now that $\gamma$ has only one end at $D_k$.\\
Suppose, without loss of generality, that the other end of $\gamma$ is in $D_i^*$. We isotope $S$ along a regular neighborhood of a disk in $C$ intersecting $K$ once, intersecting the disk $D_k$ along a single arc, and separating $D_i^*$ from $D_j^*$ in $\partial C$. In this way, we split $D_k$ into two disks $D_{k}$ and $D_{k'}$, and $C$ into two cylinders from $D_{k}$ to  $D_i^*$, $C_{k, i^*}$, and from $D_{k'}$ to $D_j^*$, $C_{k', j^*}$. The boundary of $\Gamma$ lies in $S$, and in the boundaries of the balls $C_{k, i^*}$ and $C_{k', j^*}$. The arcs of $\partial \Gamma\cap C_{k', j^*}$ have both ends attached to $D_{k'}$. Hence, we can isotope these arcs to $S$. Also, all but one arc of $\partial \Gamma\cap C_{k, i^*}$ has both ends in $D_{k}$. The other arc has one end in $D_i^*$ and the other in $D_{k}$. We isotope all arcs of $\partial \Gamma\cap (C_{k, i^*}\cup C_{k', j^*})$ with both ends in $D_{k}$ or both ends in $D_{k'}$ from $\partial C_{k, i^*}$ or $\partial C_{k', j^*}$ to $S$, respectively. We are left with the disk $\Gamma$ with boundary defined by one arc in $S$ and other arc in the boundary of $C_{k, i^*}$ from $D_i^*$ to $D_{k}$. Using this disk we can isotope $C_{k, i^*}$ through $S$. After this isotopy of $V$ we obtain a new $2$-string tangle decomposition of $K$, that contains in each tangle a string from the original tangle decomposition defined by $S$. We also reduced $|S\cap V|$. So, the new tangle decomposition cannot be essential, which implies that some string of the original tangle decomposition defined by $S$ is unknotted.

Assume at last that $\gamma$ has one end in $D_i^*$ and the other end in $D_j^*$.\\
Then each arc of $\Gamma\cap S$ co-bounds a disk in $S-int D_{k}$, with $\partial D_k$, containing $D_i^*\cup D_j^*$, and other disk containing none of these disks. Hence, we can isotope $\partial \Gamma$ to lie in $\partial C$ with the exception of $\gamma$. So, after the isotopy $\partial \Gamma$ is defined by $\gamma$ and an arc in $C$ from $D_i^*$ to $D_j^*$. The string $s$ is trivial in $C$ and therefore it is parallel to the arc $\partial \Gamma\cap C$. Therefore, the string $s$ is unknotted in $(B_1, \mathcal{T}_1)$.
\end{proof}

%% file: Outermost_disks_over_torus_components.tex
In this section, we prove several lemmas related with the existence of outermost disks over tori components of $V-V\cap S$ disjoint or intersecting $K$ at a single arc. These lemmas are fundamental on the proof of Theorem \ref{2-string tangle}. 

\begin{lem}[(c.f. Morimoto \protect{\cite[Lemma 2.4]{Morimoto2}})]\label{no beta in simple torus}
There is no outermost disk over a solid torus component of $V-V\cap S$ disjoint from $K$ and containing a single disk of $V\cap S$.
\end{lem}
\begin{proof} Denote by $T$ the solid torus as in the statement and by $D$ the disk $T\cap S$. Let $\delta$ be the outermost arc co-bounding an outermost disk $\Delta$ as in the statement. Consider, also, the corresponding arc $\beta$ and a disk $O$ cut by $\delta$ in $S-int D$. Let $L=O\cup \Delta$. The disk $L$ is a meridian for the complement of $T$ and intersects a meridian of $T$ once. Consider a regular neighborhood of $\Delta$ in $W$, $N(\Delta)$. So, $N(\Delta)\cap S$ is a regular neighborhood of $\delta$ in $S$, $N(\delta)$, and $N(\Delta)\cap \partial T$ is a regular neighborhood of $\beta$ in $\partial T$, $N(\beta)$. We isotope the annulus $N(\delta)\cup D$ through $N(\Delta)$ to the annulus $N(\beta)\cup D$. As $\beta$ intersects a meridian of $T$ once, the annulus $A=N(\beta)\cup D$ is such that $T=A\times I$. Therefore, we can isotope $A\subset S$ through $T$ to $\partial T - A$ and out of $V$. Let $S'$ be the $2$-sphere obtained from $S$ after this isotopy. The tangle decomposition of $K$ obtained from $S'$ is the same as the one given by $S$. However, $|S'\cap V|<|S\cap V|$ and we contradict the minimality of $|S\cap V|$. 
\end{proof}

\begin{lem}\label{no beta in torus}
Assume $V-V\cap S$ has a solid torus component $T$ intersecting $K$ in a single string and with $T\cap \mathcal{D}^*$ being a single disk. If there is an outermost disk over $T$ then some string of some tangle is unknotted.
\end{lem}
\begin{proof} Suppose $T$ is in the tangle $(B_1, \mathcal{T}_1)$. Let $s$ be the string $T\cap K$, that is a component of $\mathcal{T}_1$, and $D^*$ be the component of $T\cap S$ that intersects $K$. Then, both ends of $s$ are in $D^*\subset \partial T$. Let $\Delta$ be an outermost disk over $T$ and $\delta$ the respective outermost arc in $E$ attached to the disk $D$ of $T\cap S$. Consider also the disk $O$ cut by $\delta$ in $S-int D$ and disjoint from $D^*$. Let $L=O\cap\Delta$. Isotope the disks (if any) of $T\cap O$ away from $S$ in $B_1$, and denote the resulting solid torus by $T'$. By adding the $2$-handle with core $L$ to $T'$ we define a ball $Q$ that intersects $S$ at disk components.

Suppose $D=D^*$. As $\delta$ is a sk-arc, and two ends of strings are in $D$, we have that $O$ intersects $\mathcal{D}^*$ in a a single disk. In this case, the ball $Q$ contains the string $s$, and also an unknotted portion of the other string of $\mathcal{T}_1$. From Lemma \ref{inner ball}(b), some string of some tangle defined by $S$ is unknotted or the tangle $(Q, Q\cap K)$ is trivial. So, we can assume that $(Q, Q\cap K)$ is trivial, and consequently that $s$ is trivial in this tangle. As $s$ has both ends at the same disk component of $Q\cap S$, we have that $s$ is unknotted in $(B_1, \mathcal{T}_1)$.

Suppose $D\neq D^*$. If $O$ intersects $K$ at a single point, then following the argument used in the previous case we have that some string of some tangle is unknotted. So, assume that $O$ intersects $K$ at a collection of two points. In this case, the ball $Q$ contains the string $s$, and also two portions of the other string, $s'$ that are unknotted in $Q$. So, $\partial Q$ defines a $3$-string tangle decomposition of $K$. Let $Q^c$ denote the complement of $Q$ in $S^3$. From  Ozawa's unicity theorem, either the tangle $(Q, Q\cap K)$ or the tangle $(Q^c, Q^c\cap K)$ is inessential. As there are no local knots in the tangles defined by $S$ and the tangles $(Q, Q\cap K)$ and $(Q^c, Q^c\cap K)$  are $3$-string tangles, the tangle that is inessential has a trivial string. If the tangle $(Q^c, Q^c\cap K)$ is inessential then, following an argument as in the proof of Lemma \ref{inner ball}, either some string of the tangle $(B_2, \mathcal{T}_2)$ is trivial, which is a contradiction, or the string $Q^c\cap s'$ is trivial in $Q^c$. In the latter case isotope $Q$ from $S$ in such a way that $Q\cap S$ is only $D\cup O$, and denote by $Q'$ the ball after the isotopy. Then, $(Q', Q'\cap s')$ is the product tangle with respect to $D\cup O$, and $Q'^c\cap s'$ is isotopic to $\partial Q'-\partial Q'\cap S$. Therefore, after the isotopy of $Q'^c\cap s'$ to $Q'$ we have that $s'$ is unknotted in $(Q', Q'\cap K)$. As $s'$ has both ends in the same disk component of $Q'\cap S$ we have that $s'$ is unknotted in the tangle $(B_1, \mathcal{T}_1)$. Suppose now that the $3$-string tangle $(Q, Q\cap K)$ is inessential. Then one of the strings $Q\cap K$ is trivial in this tangle. If such a string is $s$ then the string $s$ is unknotted in $(B_1, \mathcal{T}_1)$. If such a string is one of the arcs obtained from $Q\cap s'$, then consider the compressing disk $C$ for $\partial Q$ in the interior of $Q$ and the ball $Q''$, containing the string $s$, obtained after cutting $Q$ along $C$. From Lemma \ref{inner ball}(b), either some string of $(B_1, \mathcal{T}_1)$ is unknotted or the tangle $(Q'', Q''\cap K)$ is trivial. Then $s$ is trivial in $Q''$ and unknotted in $Q$ (that is obtained from $Q''$ after gluing a ball along a disk). As $s$ has both strings in the disk $T\cap \mathcal{D}^*$ we also have that $s$ is unknotted in the tangle $(B_1, \mathcal{T}_1)$.   
\end{proof}

\begin{lem}\label{no ball containing T}
Let $T$ be a solid torus component of $V-V\cap S$ with more than one component from $V\cap S$ in its boundary. Then there is no ball $Q$, in the tangles defined by $S$, with the following properties,
\begin{itemize}
\item[(1)] $T\subset Q$, $\partial Q\cap \partial T$ is the disks $T\cap S$ union an annulus $A$ (that contains at least two components of $T\cap S$);
\item[(2)] $(\partial Q \cap S)\cup A$ is an annulus $A'$, and $A'-A$ is a collection of disks, attached to some components of $ T\cap S$, containing the disks of $T\cap S$ not in $A$;
\item[(3)] the two strings of a tangle are in $Q$ and the tangle in $Q$ with these two strings is essential.
\end{itemize}
\end{lem}
\begin{proof}
Suppose there is a ball $Q$ as in the statement. From Lemma \ref{inner ball}(a) the complement of $Q$ in $S^3$ contains an essential tangle. As $(Q, Q\cap K)$ is an essential tangle, from Ozawa's unicity theorem, we have that the tangle decomposition of $K$ defined by $S$ and $\partial Q$ are isotopic. Note that as $A'-A$ is a collection of disks in $S$ attached to $\partial T\cap S$, $A-A\cap S=A'-A'\cap S$. Consider an arc $a$ in $A-A\cap S$ connecting the two components of $\partial A$.  Then $a$ is also an arc in $A'-A'\cap S$ connecting the two different components of $\partial A'$. Consider $S'$ after an isotopy of $\partial Q$ along a regular neighborhood of the arc $a$ in the complement of $Q$. Then, all disks of $S\cap T$ that are in $A$ are now in a single disk component of $S'\cap T$. The sphere $S'$ defines the same tangle decomposition of $K$ than $S$ does. And also, as $A$ contains at least two components of $T\cap S$, we have $|S\cap V|>|S'\cap V|$, which contradicts the minimality of $|S\cap V|$.
\end{proof}

For the next lemmas assume that $n_1\geq 3$ and consider a solid torus component $T$ of $V-V\cap S$ intersecting $K$ at a single arc component. Suppose there is an outermost disk $\Delta$ over $T$, and let $\delta$ be the respective outermost arc attached to the disk $D$ of $V\cap S$. Assume also without loss of generality that $T$ is in $(B_1, \mathcal{T}_1)$. Denote by $s_{11}$ and $s_{12}$ the strings of $\mathcal{T}_1$, and by $s_{21}$ and $s_{22}$ the strings of $\mathcal{T}_2$. Suppose that $s_{11}$ is the string of $\mathcal{T}_1$ that $T$ contains.

\begin{lem}\label{meridional outermost arc}
If one of the disks separated by $\delta$ in $S- int D$ contains just one disk of $V\cap S$ and it intersects $K$ once, some string of some tangle is unknotted.
\end{lem}
\begin{proof}
Suppose one of the disks cut by $\delta$ in $S-int D$, say $O$, contains a single disk of $V\cap S$. As $\delta$ is a st or sk-arc the disk of $V\cap S$ in $O$ is a disk $D^*$ of $\mathcal{D}^*$, which from the statement intersects $K$ once.\\
Consider the disk $L=\Delta\cup O$. As we are in $S^3$, by attaching the $2$-handle with core $L$ to $T$ we obtain a ball. Consequently, as $D^*$ is a disk of $V\cap S$ (separating or non-separating in $V$), by attaching a regular neighborhood of the annulus $A= L-int D^*$ to $V$ we have a handlebody $V'$ also of genus three. Furthermore, as $A$ is incompressible and non-separating in $W$, by cutting $W$ along $A$ we obtain a handlebody $W'$ of genus three. Altogether, by cutting $W$ along $A$ and simultaneously adding a regular neighborhood of $A$ to $V$, we obtain a Heegaard decomposition of $S^3$, $V'\cup W'$, of the same genus as the one defined by $\partial V$.\\

Let $T'$ be a solid torus obtained from $T$ by an ambient isotopy of $B_1\cap V$ taking $D^*$ away from $S$ in $B_1$. We denote by $Q$ the ball obtained by attaching a regular neighborhood of $L$ to $T'$.
As $T$ intersects $K$ at a single arc and as $L$ intersects $K$ at a single point, we have that $Q$ intersects $K$ at two arcs, with one being unknotted.\\ 
Let $T''$ be a solid torus obtained from $T$ by an ambient isotopy of $(T\cap S)-D$ away from $S$ in $B_1$. We denote by $R$ the ball obtained by attaching a regular neighborhood of $L$ to $T''$.
As $T$ intersects $K$ at a single arc and as $L$ intersects $K$ at a single point, we have that $R$ intersects $K$ at two arcs, with one being unknotted.\\

Suppose $n_1=4$.

(1) Suppose $D$ is in $\mathcal{D}$ and $D^*$ is not in $T$. Then $Q$ intersects each string of $\mathcal{T}_1$ at a single arc. Then by Lemma \ref{inner ball}(b) some string of some tangle defined by $S$ is unknotted or the tangle $(Q, Q\cap K)$ is trivial. So, we can assume the latter. Each disk of $Q\cap (V'-int Q)$ intersects $K$ at most at a single point. Therefore, the arcs $Q\cap K$ can be isotoped to $\partial Q$ intersecting $Q\cap (V'-int Q)$ only at the end points. From Lemma \ref{strings parallel to boundary}, all the other components of $(V'-V'\cap S)-Q$ intersecting $K$ have the same property. Furthermore, if two consecutive arcs are in adjoint components of $V'-V'\cap S$ then, after the isotopy to the boundary of the arcs in the respective components, we can choose that the common ends are at the same point of the disk of intersection between the components. (In this case, this is a consequence from each component of $V'\cap S$ intersecting $K$ at most once and the tangle in each component of $V'-V'\cap S$ being trivial.) So, with $V'$ being the union of components with these properties, $K$ is parallel to $\partial V'$. We also note that there is a meridian disk of $V'$ intersecting $K$ once. Altogether, we have that $(V'-N(K))\cup W'$ is a genus three Heegaard decomposition of the knot $K$ exterior. But $|S\cap V'|< |S\cap V|$, which is a contradiction to the minimality of $|S\cap V|$.

(2) Suppose $D$ is in $\mathcal{D}$ and $D^*$ is in $T$. Note that $Q$ intersects $\mathcal{T}_1$ at $s_{11}$ in two arcs, with one of the arcs being unknotted in $Q$. If the tangle $(Q, Q\cap K)$ is trivial then following a similar argument as in (1), we obtain contradiction with the minimality of $|S\cap V|$. So, we can assume that $(Q, Q\cap K)$ is essential. Consider the complement of $Q$ in $S^3$, $Q^c$. If the tangle $(Q^c, Q^c\cap K)$ is also essential then the tangle decomposition defined by $S$ is isotopic to the one defined by $\partial Q$; as $(Q, Q\cap K)$ contains an unknotted string, this means that some string of some tangle defined by $S$ contains an unknotted string. So, we can assume that the tangle $(Q^c, Q^c\cap K)$ is trivial. Let $s_{1}$ be the intersection of $Q^c$ with $s_{11}$, and $s_{2}$ the other string of $Q^c\cap K$. As  $(Q^c, Q^c\cap K)$ is trivial and $K$ is prime, $s_1$ or $s_{2}$ are trivial in $Q^c$. Suppose that $s_2$ is trivial in $Q^c$. By following a similar argument as in the proof of Lemma \ref{inner ball}(a), we have that either $s_{21}$ and $s_{22}$ are trivial in $(B_2, \mathcal{T}_2)$, or $s_{12}$ is trivial in $(B_1, \mathcal{T}_1)$, which is a contradiction to these tangles being essential. Suppose $s_2$ is knotted in $Q^c$. As $(Q^c, Q^c\cap K)$ is trivial, there is a proper disk in $Q^c$ separating $s_1$ and $s_2$; let $B$ be the ball separated by this disk containing $s_2$. Then $s_2$ is knotted in $B$. As $K$ is prime, the string in the complement of $B$ in $S^3$, $B^c\cap K$, is trivial. We have $B^c\cap K$ being $s_1$ and $Q\cap K$. Then, following a similar argument as in Lemma \ref{inner ball}(a), we have that one of the strings of $Q\cap K$ is trivial in $Q$, which contradicts $(Q, Q\cap K)$ being essential, or the string $s_1$ is trivial in $Q^c$ and the strings of $Q\cap K$ are parallel in $Q$. But one of the strings $Q\cap K$ is unknotted in $Q$, which is a contradiction to the assumption that $(Q, Q\cap K)$ is essential.

(3) Suppose $D$ is in $\mathcal{D}^*$ and $D^*$ is not in $T$. The ball $R$ intersects each string of $\mathcal{T}_1$ at a single arc component, with one of them being unknotted in $R$. From Lemma \ref{inner ball}(a), some string of some tangle defined by $S$ is unknotted or the tangle $(R, R\cap K)$ is trivial. Let $R_1$ be the complement of $R$ in $B_1$, and $R_1^c$ the complement of $R_1$ in $S^3$. Suppose the tangle $(R_1, R_1\cap K)$ is trivial then, as there are no local knots in $(B_1, \mathcal{T}_1)$, $R_1\cap s_{11}$ is unknotted in $R_1$. As $R\cap s_{11}$ is also unknotted in $R$ we have that $s_{11}$ is unknotted in $B_1$. So, we can assume that $(R_1, R_1\cap K)$ is essential. Again from Lemma \ref{inner ball}(a), we have that the tangle $(R_1^c, R_1^c\cap K)$ is essential. Therefore, the tangle decompositions defined by $S$ and $\partial R_1$ are isotopic. This means that the tangle $(R, R\cap K)$ is the following product tangle: it is ambient isotopic to the tangle in the ball $(D\cup O)\times I$, that is $R$, with strings being $((D\cup O)\cap K)\times I$. Let $V'$ be obtained from $V$ by replacing $T''$ by $R$, as in $(1)$, and $W'=S^3-int V'$. Then, the arcs $R\cap K$ can be isotoped to $\partial R$ intersecting $R\cap (V'-int R)$ only at the end points. Also, if two arcs are in adjoint components of $V'-V'\cap S$ then, after the isotopy to the boundary of the respective components, we can assume that the common ends are at the same point of the disk of intersection between the components. (In this case, this is a consequence from $(R, R\cap K)$ being the product tangle described, the tangle in each component of $V'-V'\cap S$ being trivial and each component of $V\cap S$ intersecting $K$ at most once.) So, as in $(1)$, $(V'-intN(K))\cup W'$ is a Heegaard decomposition of the knot exterior with $|S\cap V'|<|S\cap V|$, and we have a contradiction to the minimality of $|S\cap V|$.

(4) Suppose $D$ is in $\mathcal{D}^*$ and $D^*$ is in $T$. So, the ball $R$ intersects $s_{11}$ at two arcs, and $R_1$ intersects $K$ at a portion of $s_{11}$ and the string $s_{12}$. If the tangle $(R_1, R_1\cap K)$ is trivial we have that the string $s_{12}$ is trivial in $R_1$, and as it has ends in the same disk component of $R_1\cap S$ it is unknotted in $(B_1, \mathcal{T}_1)$. So, we can assume that $(R_1, R_1\cap K)$ is essential. From Lemma \ref{inner ball}(a), the tangle $(R_1^c, R_1^c\cap K)$ is essential. Then the tangle decompositions defined by $S$ and $\partial R_1$ are isotopic. This means that the tangle $(R, R\cap K)$ is the product tangle as in (3). Following a similar argument as in (3), we obtain a contradiction to the minimality of $|S\cap V|$.\\

Suppose $n_1=3$.

Assume that the ends of $s_{11}$ are at the same disk of $T\cap S$. Then $Q$ intersects each string of $\mathcal{T}_1$ at a single component. Therefore, from Lemma \ref{inner ball}(b) some string of some tangle is unknotted or the tangle $(Q, Q\cap K)$ is trivial. In the latter case we have that $s_{11}$ is trivial in $Q$ and unknotted in $B_1$. In case the ends of $s_{11}$ are in distinct components of $T\cap S$, we can follow a similar argument as in case (4). (Note that, as $n_1=3$ and the genus of $V$ is three the solid torus $T$ cannot contain two disks of $\mathcal{D}^*$ and components of $\mathcal{D}$; so, in this case we have necessarily $D$ in $\mathcal{D}^*$ and $D^*$ in $T$.) 
\end{proof}

\begin{lem}\label{T with two D*}
Suppose $T$ intersects $\mathcal{D}^*$ at two disks, $D$ and $D'$, and is disjoint from $\mathcal{D}$. Then some string of some tangle is unknotted, or there is a ball $Q$, in $B_1$, where
\begin{itemize}
\item[(1)] $Q\cap S$ is a disk intersecting $\mathcal{D}^*$ in two components;
\item[(2)] $Q\cap K$ is a collection of two arcs each with one end in $Q\cap S$;
\item[(3)] $(Q, Q\cap K)$ is a product tangle with respect to the disk $Q\cap S$ and its intersection with $K$;
\item[(4)] the complement of $Q$ in $B_1$ intersects $T$ either in a cylinder between $D'$ and a disk parallel to it in $V$, or in a cylinder between two disks parallel to $D'$ in $V$.
\end{itemize}
\end{lem} 
\begin{proof} As $T$ contains a single component from the intersection with $K$, we have that $D$ and $D'$ intersects $K$ once. As $D$ intersects $K$ at one point, one of the disks separated by $\delta$ in $S- int D$ intersects $K$ once; denote by $O$ this disk. Let $T'$ be the solid torus obtained by an isotopy of $T$ taking $D'$ away from $S$ in $B_1$. Consider the ball $Q$ defined by adding the $2$-handle with core $L=O\cup \Delta$ to $T'$. Denote by $Q_1$ the complement of $Q$ in $B_1$.

First assume that $O\cap \mathcal{D}^*$ is a disk not in $T$.
Then $Q_1\cap T$ is a cylinder between $D'$ and a disk parallel to it in $V$. The arc $Q\cap s_{12}$ is unknotted in $Q$. From Lemma \ref{inner ball}(b), either some string in the tangle $(B_1, \mathcal{T}_1)$ is unknotted or the tangle $(Q, Q\cap K)$ is trivial. So, we can assume the latter. Also, from Lemma \ref{inner ball}(a), the tangle in the complement of $Q_1$ in $S^3$ is essential. If the tangle defined in $(Q_1, Q_1\cap K)$ is also essential then the tangle decompositions defined by $S$ and $\partial Q_1$ are isotopic. Then the tangle in $Q$ is the product tangle as in the statement. Otherwise, if the tangle $(Q_1, Q_1\cap K)$ is inessential then, as the strings of $Q\cap K$ are trivial in $Q$, both strings of the tangle $(B_1, \mathcal{T}_1)$ are unknotted. So, we either have that one string of $(B_1, \mathcal{T}_1)$ is unknotted or that $(Q, Q\cap K)$ is the product tangle described with $Q_1$ intersecting $T$ in a cylinder between $D'$ and a disk parallel to it in $V$.

Assume now that $O\cap \mathcal{D}^*$ is a disk in $T$.
In this case, $Q_1\cap T$ is a cylinder having intersection with $Q$ in two disks parallel to $D'$ in $V$. From Lemma \ref{inner ball}(a), the tangle in the complement of $Q$, or of $Q_1$, in $S^3$ is essential. If the tangle $(Q_1, Q_1\cap K)$ is essential then the tangle decompositions defined by $\partial Q_1$ and $S$ are isotopic. This implies that the tangle $(Q, Q\cap K)$ is the product tangle as in the statement. Otherwise, the tangle $(Q_1, Q_1\cap K)$ is trivial. As the string $s_{12}$ is in $Q_1$ with ends in the disk $Q_1\cap S$, it is also unknotted in $(B_1, \mathcal{T}_1)$. Hence, we either have that one string of $(B_1, \mathcal{T}_1)$ is unknotted or that $(Q, Q\cap K)$ is the product tangle described with $Q_1$ intersecting $T$ in a cylinder between two disks parallel to $D'$ in $V$.
\end{proof}

%% file: Outermost_disks_when_n1_is_3.tex
In this section we consider the several cases when $n_1=3$ with respect to the existence of a genus two or a genus one component of $V-V\cap S$. 

So assume $n_1=3$ and let $D_1^*$, $D_2^*$ and $D_3^*$ be the disk components of $S\cap V$ that intersect $K$. As $|S\cap K|=4$, without loss of generality, we assume that $|D_1^*\cap K|=2$ and $|D_i^*\cap K|=1$, for $i=2,3$. As no $2$-sphere is non-separating in $S^3$, we have that $D_1^*$ is not parallel to $D_2^*$ or $D_3^*$ in $V$. So, $D_1^*$ isn't parallel in $V$ to any other disk of $S\cap V$.

\begin{lem}\label{genus 2 component n_1=3}
If $V-V\cap S$ has a genus two component then some string of some tangle is unknotted.
\end{lem}
\begin{proof}
Assume there is a component of $V-S\cap V$ with genus two that we denote by $V_2$.
As the genus of $V$ is three, $S\cap V_2$ is a collection of at most two disks.\\
If $S\cap V_2$ is a collection of two disks or a single disk disjoint from $K$, then, as the genus of $V$ is three, some disk of $\mathcal{D}$ is parallel to a disk of $\mathcal{D}^*$, or $D_1^*$ is parallel to $D_2^*$ or $D_3^*$ in $V$. This is impossible as observed before. Therefore, $S\cap V_2$ is a single disk intersecting $K$.

As $S\cap V_2$ is also separating, we can only have $S\cap V_2=D_1^*$, as in Figure \ref{threeDA.pdf}.
\begin{figure}[htbp]
\centering
\includegraphics[width=0.5\textwidth]{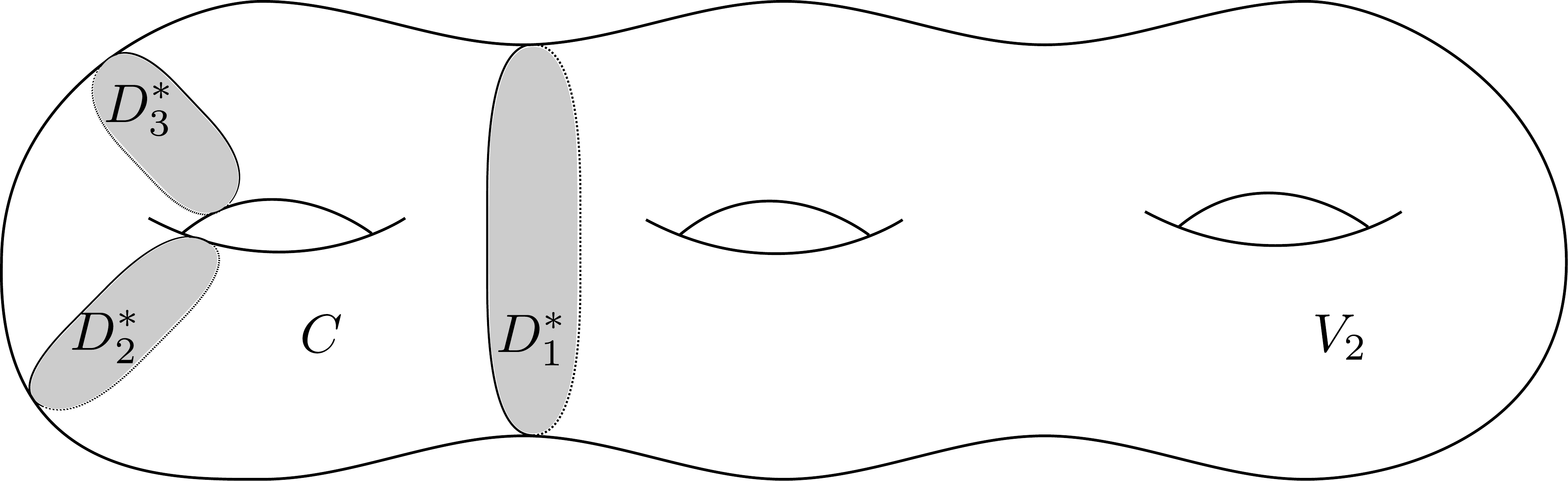}
\caption{}
\label{threeDA.pdf}
\end{figure}
So, the disks $D_2^*$ and $D_3^*$ are necessarily parallel in the solid torus separated by $D_1^*$ in $V$, and we have $n_2=0$. Also, as $V_2$ is the only non-ball component of $V-V\cap S$, from Remark \ref{remark no beta in ball}, all outermost disks are over $V_2$ and attached to $D_1^*$.\\ 
Let $C$ be the ball component of $V-V\cap S$ cut from $V$ by $D_1^*\cup D_2^*\cup D_3^*$ and suppose it lies in the tangle $(B_1, \mathcal{T}_1)$. The ball $C$ contains both strings of the tangle $(B_1, \mathcal{T}_1)$: the string $s_{12}$ with one end in $D_1^*$ and the other in $D_2^*$, and the string $s_{13}$ with one end in $D_1^*$ and the other end in $D_3^*$, and from Lemma \ref{strings parallel to boundary} both strings are mutually trivial in $C$.\\
Between the arcs of $E\cap P$ with end in $D_2^*$ or $D_3^*$ we choose one that is outermost in $E$, say $\gamma$, as in Figure \ref{threeDoutermostarcsA.pdf}(a). We note that $\gamma$ cannot have one end in $D_2^*$ and the other in $D_3^*$, as, otherwise $\delta$ wouldn't be essential in $P$. (See Figure \ref{threeDoutermostarcsA.pdf}(b).) So, without loss of generality, assume that $\gamma$ has one end in $D_2^*$.
\begin{figure}[htbp]
\centering
\includegraphics[width=0.85\textwidth]{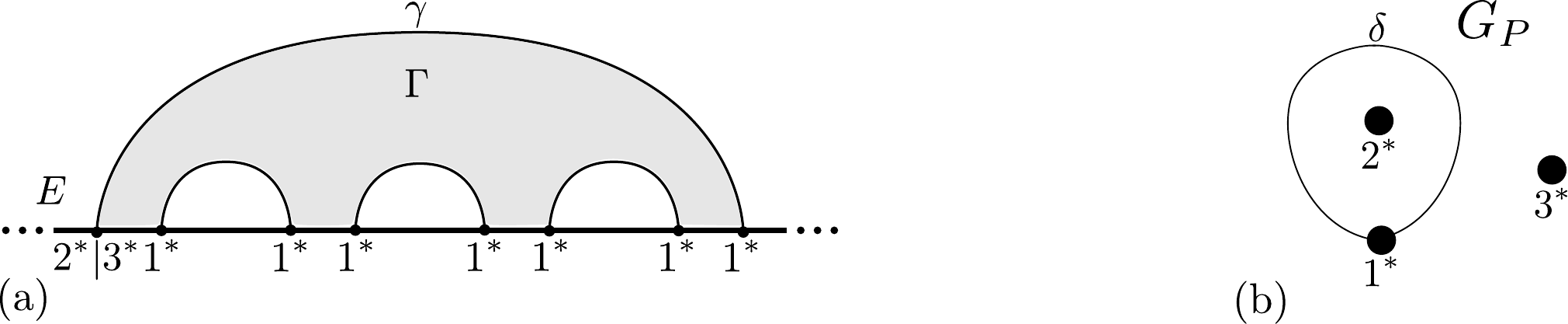}
\caption{: In (a) the arc $\gamma$ represents an arc of $E\cap P$ outermost in $E$ between the ones with at least one end in either $D_2^*$ or $D_3^*$. The label $2^*| 3^*$ at an end of the arc $\gamma$ means that this end is either at the disk $D_2^*$ or at the disk $D_3^*$.}                 
\label{threeDoutermostarcsA.pdf}
\end{figure}
The disk $\Gamma$ is in the complement of $C$ in $B_1$ and its boundary intersects $D_2^*$ only once. So, $D_2^*$ is a primitive disk with respect to the complement of $C$ in $B_1$, which is a genus two handlebody. Then, by an isotopy of $C$ along $D_2^*$ away from $S$ in $B_1$, we are left with the ball $C_{1^*, 3^*}$ that intersects $S$ at $D_1^*$ and $D_3^*$, whose complement in $B_1$ is a solid torus and with the string $s_{13}$ as a core. Hence, the string $s_{13}$ is unknotted in $(B_1, \mathcal{T}_1)$.
\end{proof}

\begin{lem}\label{solid torus component}
If there is a solid torus component of $V-V\cap S$ then both strings of some tangle are $\mu$-primitive.
\end{lem}
\begin{proof}
As the genus of $V$ is three, and one component of $V-V\cap S$ is a solid torus, the components of $V-V\cap S$ are balls or solid tori. From Remark \ref{remark no beta in ball}, all outermost disks are over solid torus components of $V-V\cap S$. Let $T$ be a torus component of $V-S\cap V$ with a outermost disk over it, and suppose $T$ is in $B_1$. The collection of disks $T\cap S$ cannot be bigger than four as the genus of $V$ is three. If the number of disks in $T\cap S$ is four then $D_1^*$ is parallel to some other disk of $V\cap S$, which is impossible as previously observed. So, $|T\cap S|$ is at most three.

Suppose $T\cap S$ is a single disk. If $T\cap S$ is disjoint from $K$ we get a contradiction to Lemma \ref{no beta in simple torus}. If $T\cap S$ intersects $K$, from Lemma \ref{no beta in torus} some string of some tangle is unknotted.

In case $T\cap S$ is a collection of two disks then we have several cases to consider. If these two disks don't intersect $K$ then $D_1^*$ is necessarily separating. Furthermore, one string from a tangle lies in a ball of $V-V\cap S$ cut by $T\cap S$ and $D_1^*$ with the two ends in $D_1^*$. Then, from Lemma \ref{strings parallel to boundary} this string is trivial in the respective tangle, which is a contradiction to the tangle being essential. If only one disk of $T\cap S$ intersects $K$ then it is necessarily $D_1^*$, because $K$ intersects $T\cap S$ an even number of times. In this situation, $T$ intersects $K$ at a single arc and from Lemma \ref{no beta in torus} some string of some tangle is unknotted.\\
If the two disks of $T\cap S$ intersect $K$ then $T\cap S=D_2^*\cup D_3^*$.
\begin{figure}[htbp]
\centering
\includegraphics[width=0.5\textwidth]{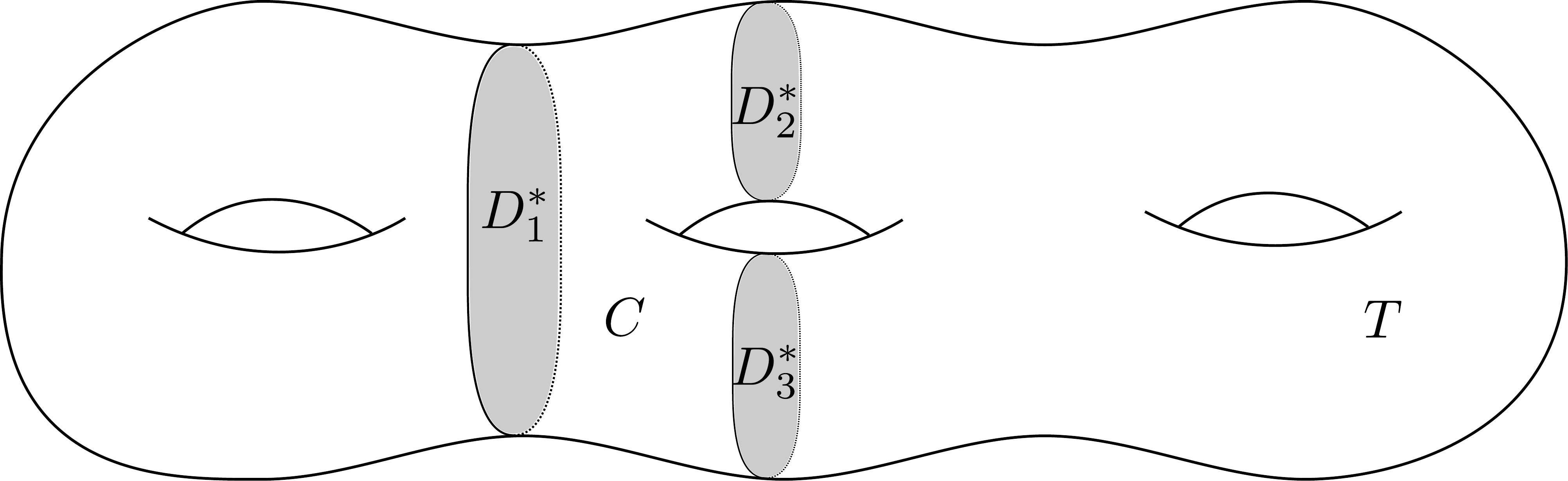}
\caption{}                 
\label{threeD.pdf}
\end{figure}
In this case, $D_2^*\cup D_3^*$ separate $V$ into two solid tori components, $T$ and $V_1$. The disk $D_1^*$ is in $V_1$ and is necessarily separating. (See Figure \ref{threeD.pdf}.) We also have $n_2=0$. Then, for the respective outermost arc of an outermost disk over $T$ we are always under the statement of Lemma \ref{meridional outermost arc}, which means that some string of some tangle defined by $S$ is unknotted.

Assume now that $T\cap S$ is a collection of three disks. 
At least some disk of $T\cap S$ intersects $K$, as otherwise $D_1^*$ would have to be parallel in $V$ to some other disk of $V\cap S$, which is impossible as previously observed.\\ 
If only one disk of $T\cap S$ intersects $K$ then this disk is $D_1^*$, and from Lemma \ref{no beta in torus} some string of some tangle is unknotted.\\
If two disks of $T\cap S$ intersect $K$ then these disks have to be $D_2^*$ and $D_3^*$. As the genus of $V$ is three either $T\cap S$ or $D_2^*\cup D_3^*$ cuts a ball from $V$. In either case, $D_1^*$ would have to be parallel to some other disk, which is impossible as previously observed.\\
The last case is when $T\cap S=D_1^*\cup D_2^*\cup D_3^*$. The disk $D_1^*$ can be separating or non-separating. In the latter case $D_1^*\cup D_2^*\cup D_3^*$ separates a ball from $V$, and in the former case the disks $D_2^*$ and $D_3^*$ are parallel and the disk $D_1^*$ separates a solid torus  $V_1$ in $V$. (See Figure \ref{threeDB.pdf}.) So, from Lemma \ref{strings parallel to boundary} and the fact that no disk of $\mathcal{D}$ is parallel to a disk of $\mathcal{D}^*$, we can assume that $n_2=0$. From $|S\cap V|=3$ and Lemma \ref{property}(b), there is only one disk attached to outermost arcs.\\
Assume $D_1^*$ is non-separating, then $D_1^*\cup D_2^*\cup D_3^*$ separates a ball $C$ from $V$, which is in the tangle $(B_2, \mathcal{T}_2)$.
\begin{figure}[htbp]
\centering
\includegraphics[width=\textwidth]{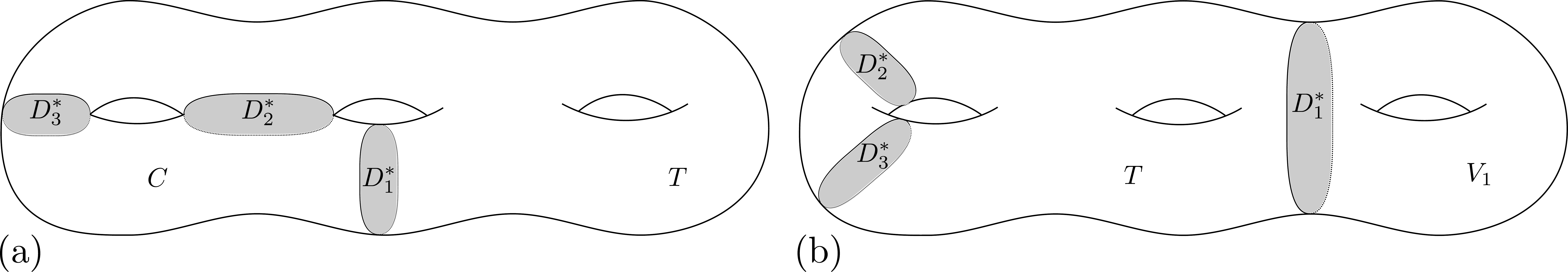}
\caption{}                 
\label{threeDB.pdf}
\end{figure}
 If there is a string in $C$ with both ends in $D_1^*$ then, from Lemma \ref{inner ball}(d), this string is unknotted in $(B_2, \mathcal{T}_2)$. So, we can assume that each string in $C$ has only one end in $D_1^*$. 
Consider a second-outermost disk $\Gamma^*$, and the respective second-outermost arc $\gamma^*$. Then $\Gamma^*$ is in the complement of $C$ in $B_2$. If $\gamma^*$ has equal ends then, following the proof Lemma \ref{no gamma with equal ends}, we have that some string of some tangle is unknotted. If the ends of $\gamma^*$ are distinct, then $D_1^*$, $D_2^*$ or $D_3^*$ is primitive in the complement of $C$ in $B_2$. Suppose $D_2^*$ (or $D_3^*$) is primitive with respect to the complement $C$ in $B_2$. After an isotopy of $C$ along $D_2^*$ (resp., $D_3^*$) away from $S$, we have that the complement of a regular neighborhood of the string $s_{13}$ (resp., $s_{12}$) is a solid torus, which implies that this string is unknotted in $(B_2, \mathcal{T}_2)$. Suppose $D_1^*$ is primitive with respect to the complement of $C$ in $B_2$. As the complement of $C$ in $B_2$ is a handlebody, after an isotopy of $C$ along $D_1^*$ away from $S$, we obtain a cylinder from $D_2^*$ to $D_3^*$, with core $t$, whose complement in $B_2$ is a solid torus. Then $t$ is unknotted in $B_2$. As $s_{12}$ and $s_{13}$ are trivial in $C$, we have that $C$ is the union of the regular neighborhoods of $t\cup s_{12}$, and also of $t\cup s_{13}$. Consequently, both $s_{12}$ and $s_{13}$ are $\mu$-primitive.\\
Assume now that $D_1^*$ is separating. Suppose $D_2^*$ and $D_3^*$ are the only disks attached to outermost arcs. As $D_2^*$ is parallel to $D_3^*$ by the finiteness of outermost arcs, if we consider a second-outermost arc we have that both disks have loops attached in $G_P$, which contradicts Lemma \ref{property}(b). So, $D_1^*$ has outermost arcs attached and all second-outermost arcs are after outermost arcs attached to $D_1^*$. If there is an outermost disk over $V_1$, from Lemma \ref{no beta in torus} some string of some tangle is unknotted. So, we can assume that all outermost disks are over $T$. Let $\Gamma^*$ be a second-outermost disk, then $\Gamma^*$ is in the complement of $V_1$ in $B_2$. Suppose $\partial \Gamma^*$ is essential in $\partial V_1\cup _{D_1^*} S$. Then the complement of $V_1$ in $B_2$ is also a solid torus (intersecting $S$ at a single disk). From Lemma \ref{strings parallel to boundary} the string $s_{11}$ is trivial in $V_1$. Then, from Lemma \ref{mu-primitive characterization}, $s_{11}$ is $\mu$-primitive. We note also that $B_2\cap V$ is $V_1$ together with the cylinder cut by $D_2^*\cup D_3^*$ in $V$, $C_{2^*3^*}$, where the string $s_{23}$ is a core. As the complement of $B_2\cap V$ in $B_2$ is a handlebody we have that $s_{23}$ is trivial in the complement of $V_1$ in $B_2$. Therefore, from Lemma \ref{mu-primitive characterization}, $s_{23}$ is also $\mu$-primitive. Suppose now that $\partial \Gamma^*$ is inessential in $\partial V_1\cup _{D_1^*} S$. Then $\partial \Gamma^*$ bounds a disk $L$ in $\partial V_1\cup _{D_1^*} S$. Let $R$ be the ball in $B_2$ bounded by $\Gamma^*\cup L$. By similar arguments as in the proof of Lemma \ref{no gamma with equal ends}, we have that $s_{23}$ is in $R$ and is parallel to $L$. So, $s_{23}$ is trivial in the complement of $V_1$ in $B_2$. As the complement of $B_2\cap V$ in $B_2$ is a handlebody, this implies that the complement of $V_1$ in $B_2$ is a solid torus. Then, as when $\partial \Gamma^*$ is essential in $\partial V_1\cup _{D_1^*} S$, we have that both $s_{11}$ and $s_{23}$ are $\mu$-primitive. 
\end{proof}

%% file: Outermost_disks_when_n1_is_4.tex
Along this section we consider the several cases when $n_1=4$ with respect to the components of $V-V\cap S$ topology and their intersection with $S\cap V$.

So assume that $n_1=4$. As $|S\cap K|=4$ we have $|D_i^*\cap K|=1$, for $i=1, 2, 3, 4$. Therefore, $D_i^*$ is a non-separating disk in $V$.\\
Denote by $\gamma_i^*$ the outermost arcs of $E\cap P$, in $E$, among the arcs with at least one end in $D_i^*$, for $i=1, 2, 3, 4$. Also, let $\Gamma_i^*$ denote the disk of $E-E\cap P$ co-bounded by $\gamma_i^*$ in the outermost side of this arc in $E$, for $i=1, 2, 3, 4$.


\begin{lem}\label{genus 2 component} 
If $V-V\cap S$ contains a genus two component then some string in some tangle is unknotted.
\end{lem}
\begin{proof}
\begin{figure}[htbp]
\centering
\includegraphics[width=\textwidth]{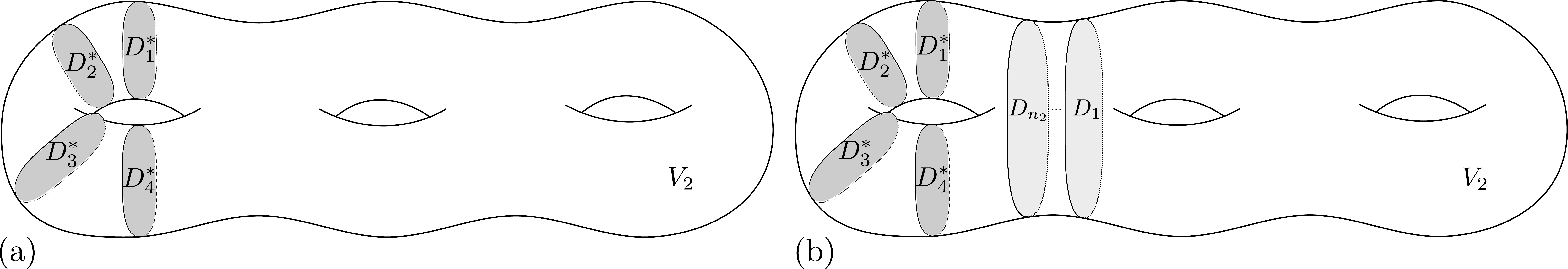}
\caption{}                
\label{T2.pdf}
\end{figure}
Assume that $V-V\cap S$ contains a genus two component, $V_2$. As the genus of $V$ is three $S\cap V_2$ is a collection of at most two disks. If $S\cap V_2$ is a collection of two disks then $S\cap V$ are all parallel disks in $V$ and, from Remark \ref{remark no beta in ball}, all outermost disks are over $V_2$. Therefore, $n_2=0$ and all disks of $\mathcal{D}^*$ are parallel, as in Figure \ref{T2.pdf}(a). Consequently, by the finiteness of outermost arcs, we have parallel type I $\mathcal{\text{d}^*}$-arcs in $E$, as in Figure \ref{T2outermostarcs.pdf}(a1) or (a2), in contradiction to Corollary \ref{no parallel sk-arcs}.
Then, $S\cap V_2$ is a single disk. As each disk of $\mathcal{D}^*$ intersects $K$ once, $S\cap V_2$ is a disk of $\mathcal{D}$.
\begin{figure}[htbp]
\centering
\includegraphics[width=\textwidth]{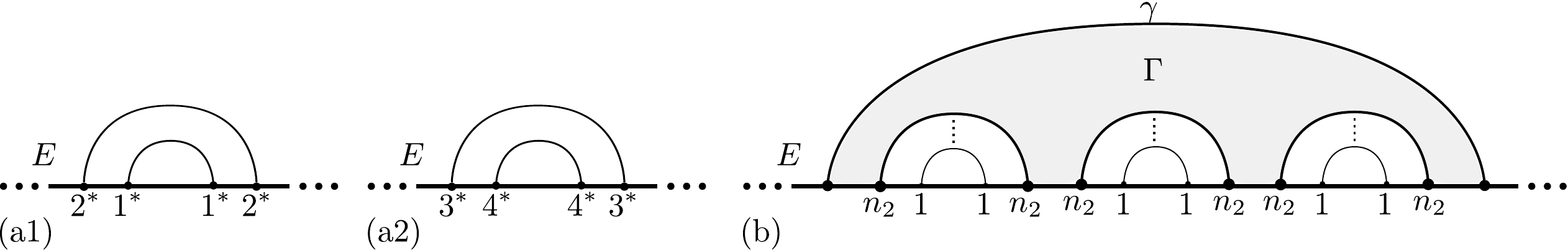}
\caption{}                 
\label{T2outermostarcs.pdf}
\end{figure}
Then, all disks of $\mathcal{D}^*$ are parallel in the solid torus cut from $V$ by $S\cap V_2$, and all disks of $\mathcal{D}$ are parallel to $S\cap V_2$ in $V$, as in Figure \ref{T2.pdf}(b). Let $D_1,\ldots, D_{n_2}$ be the disks of $\mathcal{D}$, with $S\cap V_2$ being $D_1$. The outermost disks are all adjacent to $D_1$ and are over $V_2$. Consider a second-outermost arc $\gamma$, as in Figure \ref{T2outermostarcs.pdf}(b). If the arc $\gamma$ has at least one end in $D_{n_2}$, or has one end in $D_1^*$ and the other in $D_4^*$, by Lemma \ref{no gamma with equal ends}, some string in the tangle decomposition defined by $S$ is unknotted. Otherwise, if all second-outermost arcs have both ends in $D_1^*$ or both ends in $D_4^*$, as when $S\cap V_2$ is two disks, by the finiteness of outermost arcs we have a contradiction to Corollary \ref{no parallel sk-arcs}.
\end{proof}

Assume now that $V-V\cap S$ has a solid torus component $T$ with some outermost disk over it. Hence, as the genus of $V$ is three, the components of $V-V\cap S$ are solid tori or balls, and the solid torus $T$ components intersect $S$ at most in four disks. As each disk of $\mathcal{D}^*$ intersects $K$ once, the solid torus $T$ intersects $\mathcal{D}^*$ at an even number of disks.


\begin{lem}\label{4D*}
Suppose $V-V\cap S$ contains a solid torus component intersecting $\mathcal{D}^*$ at the four disks. Then some string in some tangle is unknotted.
\end{lem}
\begin{proof}
Let $T$ be the solid torus component of $V-V\cap S$ as in the statement, and suppose it lies in the tangle $(B_1, \mathcal{T}_1)$. As the genus of $V$ is three and $T$ intersects $\mathcal{D}^*$ at the four disks, we have that the disks of $\mathcal{D}^*$ are parallel two-by-two in $V$, say $D_1^*$ parallel to $D_2^*$ and $D_3^*$ parallel to $D_4^*$. So, $n_2=0$, and $V\cap S$ is as in Figure \ref{4D.pdf}. Also, from Remark \ref{remark no beta in ball}, we can assume all outermost disks are over $T$. From Lemma \ref{property}(b), at most two disks are adjacent to outermost arcs.\\
If the outermost arcs are attached to a single disk or if they are attached to two non-parallel disks, by the finiteness of outermost arcs of $E\cap P$ in $E$ we have a contradiction to Corollary \ref{no parallel sk-arcs}.\\
Then, the outermost arcs are attached to two parallel disks. 
\begin{figure}[htbp]
\centering
\includegraphics[width=0.5\textwidth]{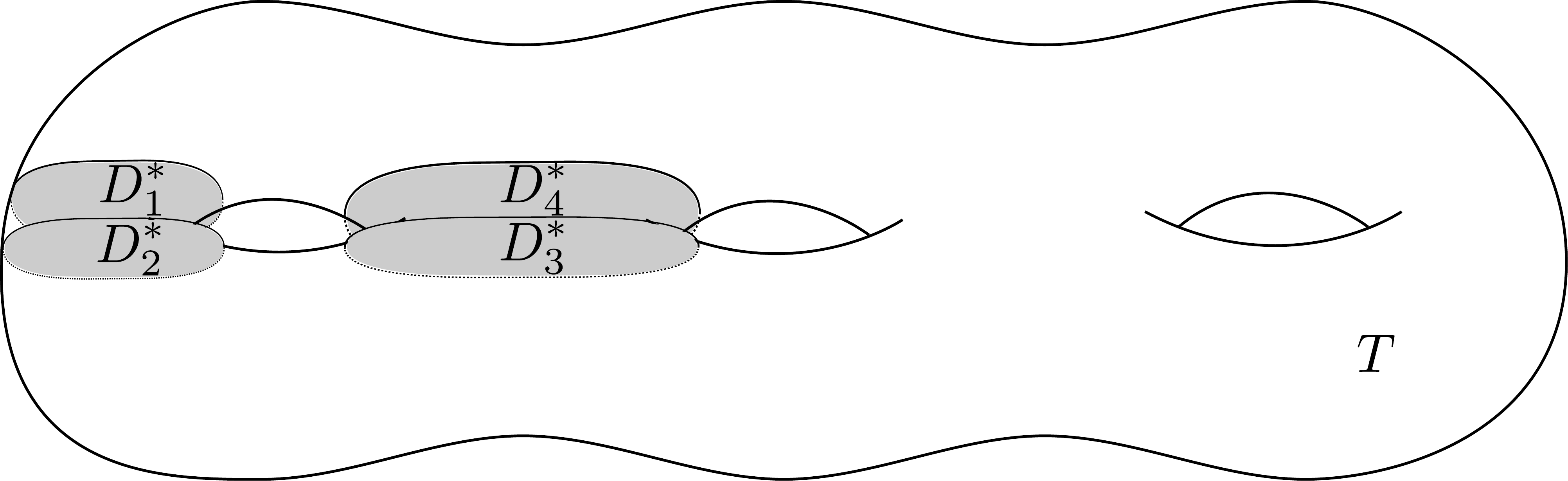}
\caption{}                 
\label{4D.pdf}
\end{figure}
Without loss of generality, assume that the only disks adjacent to outermost arcs are $D_1^*$ and $D_2^*$.
\begin{figure}[htbp]
\centering
\includegraphics[width=0.9\textwidth]{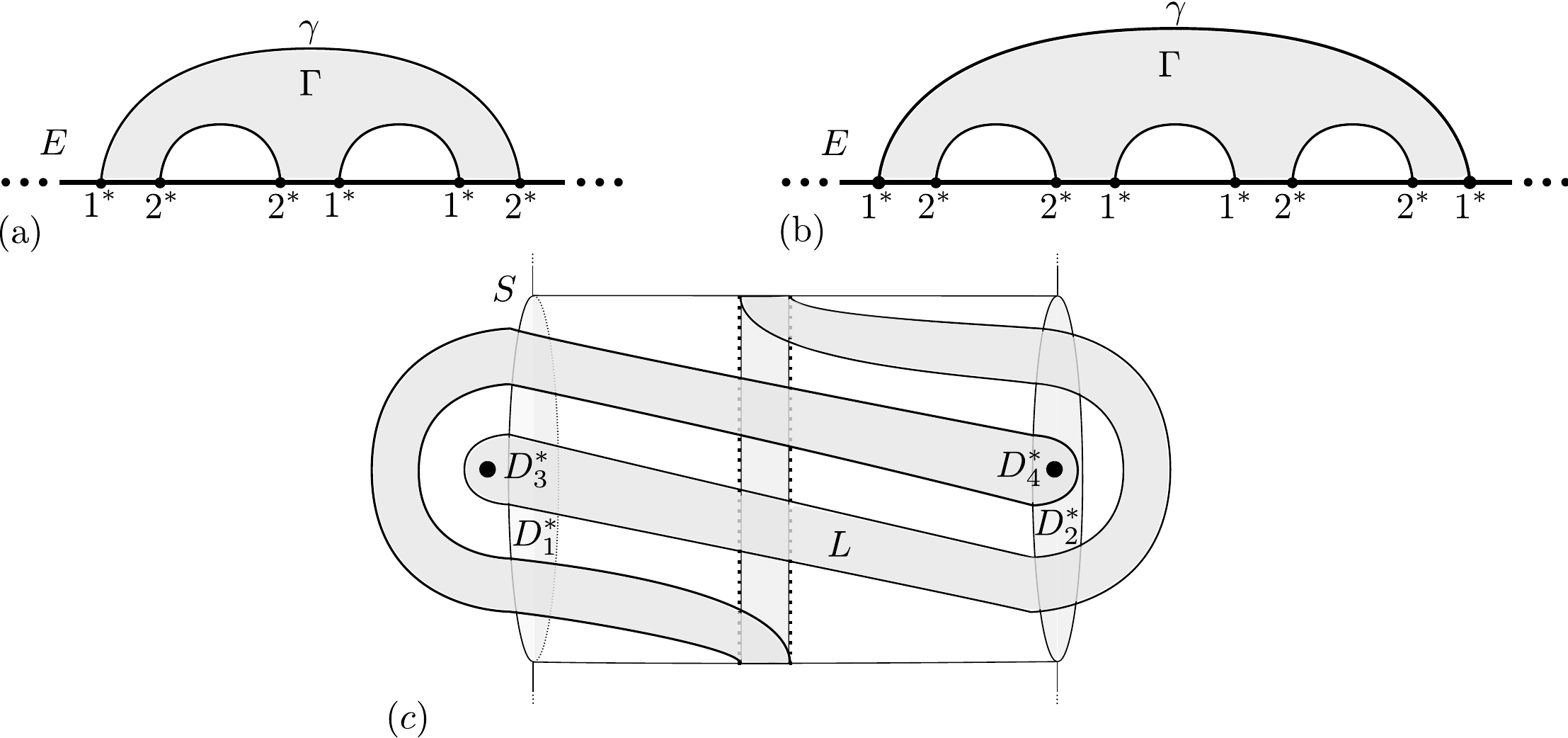}
\caption{}                 
\label{4Doutermostarcs.pdf}
\end{figure}
Consider the second outermost arc $\gamma$ of $E\cap P$ in $E$, after the outermost arcs attached to $D_1^*$ and $D_2^*$, and the disk component of $E-E\cap P$, $\Gamma$, co-bounded by $\gamma$ on the outermost side of $\gamma$ in $E$. Let $C_{1 2}$ and $C_{3 4}$ be the cylinders cut from $V$ by $D_1^*\cup D_2^*$ and $D_3^*\cup D_4^*$, resp.. We have that $\Gamma$ is in $B_2- int C_{1 2}$. If $\Gamma$ is essential in $S\cup \partial C_{1 2}$, as in Figure \ref{4Doutermostarcs.pdf}(a), then $\Gamma$ is a meridian disk to $B_2- int C_{1 2}$, which implies that the string $s_{12}$, a core of $C_{12}$, is unknotted in the tangle $B_2$. Otherwise, if $\Gamma$ is inessential in $S\cup \partial C_{12}$, we have that $\partial \Gamma$ bounds a disk $L$ in the torus $S\cup_{D_1^*\cup D_2^*} \partial C_{1 2}$. (See Figure \ref{4Doutermostarcs.pdf}(b), (c).) Let $R$ be the ball in $B_2$ bounded by $\Gamma\cup L$. The string $s_{34}$, as a core of $C_{34}$, is in $R$ and, as there are no local knots, it is trivial in $R$ and parallel to $L$. Hence, as the complement of $C_{1 2}\cup C_{3 4}$ in $B_2$ is a handlebody, we have that the complement of $C_{1 2}$ in $B_2$ is a solid torus. Therefore, in this case, the string $s_{12}$ is also unknotted.
\end{proof}


\begin{lem}\label{2D*1DS}
Suppose $V-V\cap S$ contains a solid torus that intersects $\mathcal{D}^*$ in a collection of two disks and $\mathcal{D}$ in a single separating disk. Then both strings of some tangle are $\mu$-primitive.
\end{lem}
\begin{proof}
Let $T$ be the solid torus component of $V-V\cap S$ as in the statement, and suppose it lies in the tangle $(B_1, \mathcal{T}_1)$. Assume that $T\cap \mathcal{D}^*=D_1^*\cup D_4^*$ and that $\mathcal{D}\cap T=D_1$, and denote by $V_1$ the solid torus separated by $D_1$ in $V$. As $\mathcal{D}\cap T$ is separating, and $K$ is connected, the four disks of $\mathcal{D}^*$ have to be parallel in $V$. If all outermost arcs are attached to $D_1^*$ or to $D_4^*$ then, by the finiteness of outermost arcs, we have a contradiction to Corollary \ref{no parallel sk-arcs}. Hence, there is an outermost arc attached to $D_1$. The set $\mathcal{D}$ contains a collection of separating disks, $D_1, \ldots, D_k$ in $V$, and might also contain a collection of non-separating parallel disks $D_{k+1}, \ldots, D_{n_2}$ in $V_1$, as in Figure \ref{2D1ds.pdf}(a), (b). From Remark \ref{remark no beta in ball},  the outermost disks are over $T$ or $V_1$, and from Lemma \ref{no beta in simple torus} there are no outermost disks over $V_1$. So, all outermost disks are over $T$, attached to $D_1^*$, $D_4^*$ or $D_1$, with no sequence of parallel arcs of $E\cap P$ in $E$ after an outermost arc attached to $D_1^*$ or $D_4^*$.

\begin{figure}[htbp]
\centering
\includegraphics[width=\textwidth]{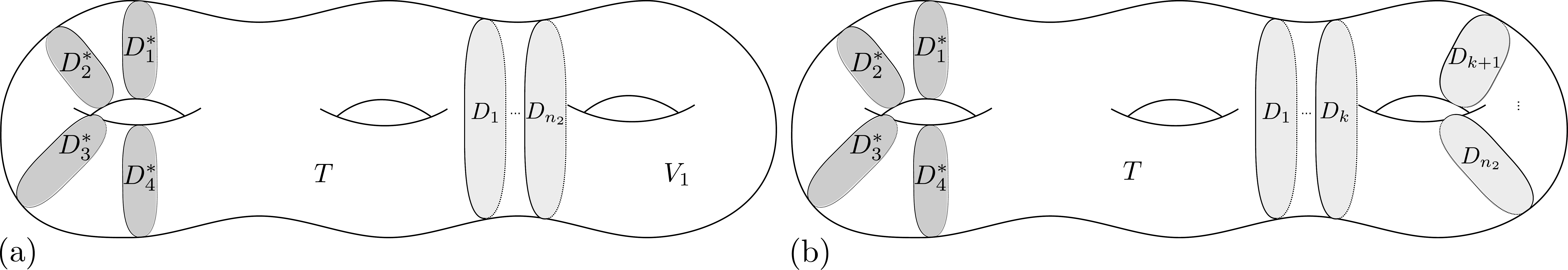}
\caption{}                 
\label{2D1ds.pdf}
\end{figure}

\textbf{Case 1.} \textit{Assume that all disks of $\mathcal{D}$ are parallel and separating as in Figure \ref{2D1ds.pdf}(a).}\\

If $n_2>1$, by the finiteness of outermost arcs, there is a sequence of parallel arcs of $E\cap P$ in $E$, $\delta_1,\ldots ,\delta_{n_2}$, as in the Figure \ref{2D1doutermostarcsSI.pdf}(a), where $\delta_i$ has both ends in $D_i$ and $\delta_1$ is an outermost arc attached to $D_1$. Denote the outermost disk that $\delta_1$ co-bounds by $\Delta$, and the disk between $\delta_i$ and $\delta_{i+1}$ by $\Delta_i$. Considering the disks $\Delta_i$ and the cylinder cut from $V$ by $D_i\cup D_{i+1}$ we define a ball $R_i$ as in Lemma \ref{parallel arcs}. The balls $R_i$ intersect $S$ at disks $O_i$ and $O_{i+1}$, co-bounded by $\delta_i$ and $\delta_{i+1}$ resp., each containing an end of the string in $R_i$. Then, in particular, $O_1$ contains a single disk of $\mathcal{D}^*$. If $n_2=2$, as $R_1$ contains a single string, we have that $O_1$ intersects $S\cap V$ at a single disk, that is of $\mathcal{D}^*$. Assume $n_2\geq 3$. If $D_2^*$, or $D_3^*$, is in $O_1$ then $R_2$ contains $T$ and consequently two strings of the tangle, which is a contradiction as $R_2$ contains a single string. Therefore, without loss of generality, we can assume that $D_1^*$ is in $O_1$. Suppose that some disk of $\mathcal{D}$, say $D_i$, is in $O_1$. Then $O_i\subset O_1$  and $D_1^*\subset O_i$. Consequently, following the strings in the sequence of balls $R_j$, we have $T\subset R_{i-1}$ and $D_1$ in $O_i$, which is a contradiction as $D_i$ is in $O_1$. Therefore, $D_1^*$ is the only disk of $S\cap V$ in $O_1$. Then, by Lemma \ref{meridional outermost arc}, if $n_2>1$ some string of some tangle is unknotted.

Suppose $n_2=1$. As before we denote by $\delta_1$ an outermost arc attached to $D_1$. If a disk cut from $S-int D_1$ by $\delta_1$ intersects $\mathcal{D}^*$ at a single disk from Lemma \ref{meridional outermost arc} some string of some tangle is unknotted. Therefore, we can assume that all outermost arcs $\delta_1$ separate $S-int D_1$ into two disks each intersecting $\mathcal{D}^*$ at two disks. Consequently they are all parallel in $P$. Let $\Gamma$ be a second-outermost disk. The disk $\Gamma$ is in the complement of $V_1$ in $B_2$. If $\partial \Gamma$ is inessential in the solid torus $B_1\cup_{D_1} V_1$ then $\Gamma$ bounds a disk $L$ in $S\cup_{D_1} \partial V_1$. Let $R$ be the ball bounded by $\Gamma\cup L$ in $B_2$. By similar arguments as in the proof of Lemma \ref{no gamma with equal ends}, we have that the strings $s_{12}$ and $s_{34}$ are in $R$ and are parallel to $L$. Hence, the complement of $V_1$ in $B_2$ is a solid torus intersecting $S$ at a single disk. Altogether, from Lemma \ref{mu-primitive characterization} we have that both strings $s_{12}$ and $s_{34}$ are $\mu$-primitive. Suppose now that $\partial \Gamma$ is essential in the solid torus $B_1\cup \partial_{D_1} V_1$. Then the complement of $V_1$ in $B_2$ is also a solid torus. Consider an outermost arc among the arcs with one end in $D_2^*$ or $D_3^*$, and denote these arcs by $\gamma^*$. Suppose there are arcs $\gamma^*$ with both ends in $D_2^*$ and also in $D_3^*$. Then there are arcs $\gamma_1^*$ and $\gamma_4^*$ of Type II outermost among the $\text{d}^*$-arcs, and the disks $\Gamma_1^*$ and $\Gamma_4^*$ are in $B_1$ and intersect $D_1^*$ and $D_4^*$, resp., exactly once. Then, $D_1^*$ and $D_4^*$ are primitive with respect to the complement of $V\cap B_1$ in $B_1$. Let $T'$ be the solid torus obtained by an isotopy of $T$ along $D_1^*\cup D_4^*$ away from $S$. We also have that an outermost disk $\Delta$ intersects a meridian of $T'$ once. Altogether, the complement of the cylinder from $D_2^*$ to $D_3^*$ in $B_1$ is a solid torus; as the core of this cylinder is the string $s_{23}$, this string is unknotted in $(B_1, \mathcal{T}_1)$. Otherwise, without loss of generality, suppose there is an arc $\gamma^*$ with only one end in $D_2^*$. This means $\gamma^*$ is an $\gamma_2^*$ arc, and we can consider the respective disk $\Gamma_2^*$. Let $C_{12}$ (resp., $C_{34}$) be the cylinder from $D_1^*$ to $D_2^*$ (resp., $D_3^*$ to $D_4^*$) in $V$.  As $D_2^*$ is primitive with respect to the complement of $V\cap B_2$ in $B_2$, a core of $C_{34}$, as the string $s_{34}$, is trivial in the complement of $V_1$ in $B_2$. If the other end of $\gamma^*$ is in $D_3^*$ then $D_3^*$ is also primitive with respect to the complement of $V\cap B_2$ in $B_2$, and similarly a core of $C_{12}$, as the string $s_{12}$, is trivial in the complement of $V_1$ in $B_2$. Otherwise, if the other end of $\gamma^*$ is not in $D_3^*$, using the disk $\Gamma_2^*$, we have that a core of $C_{12}$, as the string $s_{12}$, is trivial in the complement of $V_1$ in $B_2$. Then, from Lemma \ref{mu-primitive characterization}, both strings $s_{12}$ and $s_{34}$ are $\mu$-primitive.\\

\textbf{Case 2.} \textit{Assume now that $\mathcal{D}$ also has a collection of non-separating disks in $V$, as in Figure \ref{2D1ds.pdf}(b).}

\begin{claim}\label{2D*1DS claim 5}
If the outermost arcs attached to $D_1$ are not parallel in $P$ then some string of some tangle is unknotted.
\end{claim}
\textit{Proof of Claim \ref{2D*1DS claim 5}.}
In fact, let $\delta_1$ and $\delta_1'$ be outermost arcs attached to $D_1$, non-parallel in $P$. Consider the disjoint disks $O_1$ and $O_1'$ co-bounded, respectively, by $\delta_1$ and $\delta_1'$ in $S-D_1$, and also the respective outermost disk $\Delta_1$, $\Delta_1'$. Consider the disks $L_1=O_1\cup \Delta_1$ and $L_1'=O_1'\cup \Delta_1'$. Let $Q$ be the ball obtained by attaching a regular neighborhood of $L_1$ and $L_1'$ to $T$ and adding a ball to the  respective boundary component disjoint from $S$. If $\mathcal{D}^*\subset O_1\cup O_1'$ then the arcs $\delta_1$ and $\delta_1'$ are parallel. If $(O_1\cup O_1')\cap \mathcal{D}^*$ is only $D_1^*$ and $D_4^*$, then $\partial Q- \partial Q\cap S$ is a compressing disk for $P$. Otherwise, $D_2^*\cup D_3^*$ is in $O_1\cup O_1'$ and the string $s_{23}$ is in $Q$. From Lemma \ref{no ball containing T} the tangle $(Q, Q\cap K)$ is trivial. Therefore, the string $s_{23}$ is trivial in $Q$. As the ends of $s_{23}$ are in the same disk component of $Q\cap S$ we have that $s_{23}$ is unknotted in $(B_1, \mathcal{T}_1)$.
\hspace*{\fill}$\triangle$\\

From the previous claim, we assume that the outermost arcs attached to $D_1$ are parallel in $P$.

If $k>1$, by the finiteness of outermost arcs we have a sequence of arcs, $\delta_i$ for $i=1, \ldots, k$, after an outermost arc, $\delta_1$, as in the Figure \ref{2D1doutermostarcsSI.pdf}(b). Following the construction at the beginning of Case 1, from each sequence of parallel arcs after an outermost arc $\delta_1$ we have a sequence of balls $R_i$, $i=1, \ldots, k$. Also, as there are no t-arcs, the outermost arc, $\delta_{k+1}$, after these arcs is a st-arc. If some arc $\delta_{k+1}$ has both ends in $D_k$, following an argument as in Lemma \ref{no gamma with equal ends}, we have that some string of some tangle is unknotted. Then, the arcs $\delta_{k+1}$ have both ends in $D_{k+1}$, as in Figure \ref{2D1doutermostarcsSI.pdf}(c), or in $D_{n_2}$.\\
For any $k$, suppose we have both situations, that there are arcs $\delta_{k+1}$ and $\delta_{k+1}'$ with both ends in $D_{k+1}$ and $D_{n_2}$, resp.. Consider the component disks $\Delta_k$ and $\Delta_k'$ of $E-E\cap P$ co-bounded by $\delta_{k+1}$ and $\delta_{k+1}'$, resp., in the outermost side of these arcs in $E$. As the outermost arcs $\delta_1$ are parallel in $P$, and the balls $R_i$, $i=1, \ldots, k-1$, contain only one string, the arcs of $\partial \Delta_k$ and $\partial \Delta_k'$ that have both ends in $D_k$ are parallel in $P$. Let $C$ be the ball cut from $V$ by $D_k\cup D_{k+1}\cup D_{n_2}$, and $C_{k, k+1}$ (resp., $C_{k, n_2}$) be the ball obtained from $C$ by an isotopy of $D_{n_2}$ (resp., $D_{k+1}$) away from $S$.
\begin{figure}[htbp]
\centering
\includegraphics[width=\textwidth]{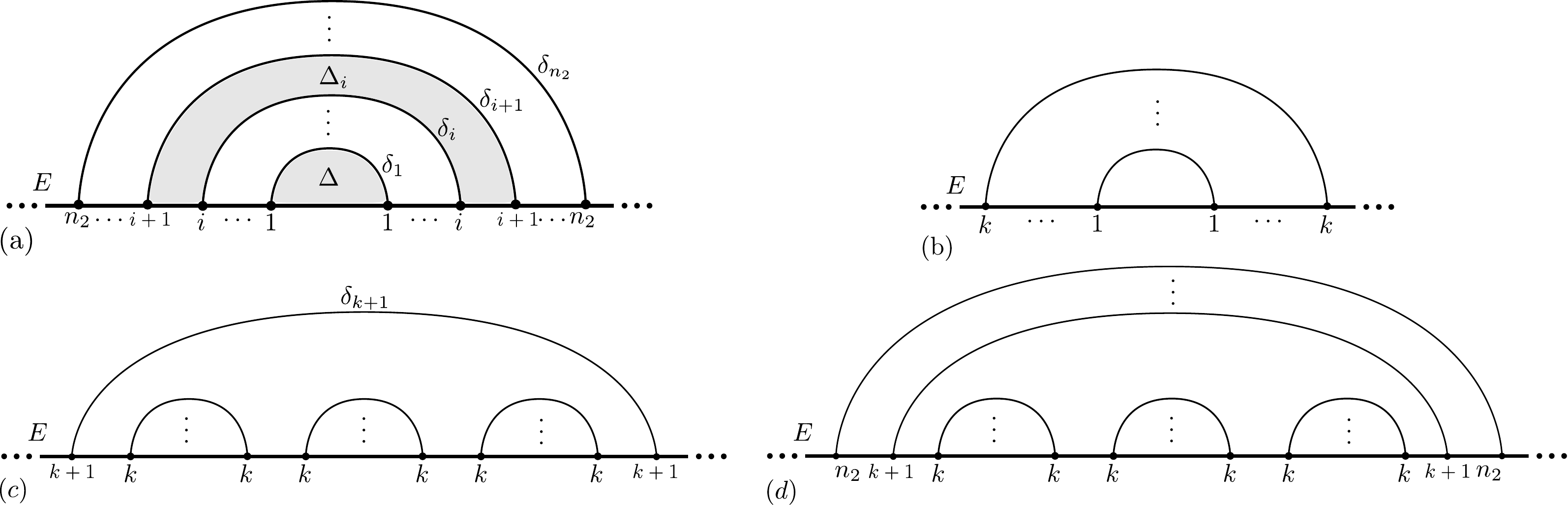}
\caption{}                 
\label{2D1doutermostarcsSI.pdf}
\end{figure}
Let $L_k$ and $L_k'$ be the disks bounded by $\partial \Delta_k$ and $\partial \Delta_k'$, resp., in $\partial C_{k, k+1}\cup_{D_k\cup D_{k+1}} S$ and $\partial C_{k, n_2}\cup_{D_k\cup D_{n_2}}S$, resp.. Consider the balls $R_k$ and $R_k'$ bounded by $L_k\cup \Delta_k$ and $L_k'\cup \Delta_k'$, not containing $S$. Similarly, as observed in Case 1, the balls $R_k$ and $R_k'$ contain only one string. Suppose none of these balls contains the other, as in  Figure \ref{2D1doutermostarcsSI2.pdf}(a). Hence, each of the disks $L_k$ and $L_k'$ intersection with $S$ contains a disk component, $O_k$ and $O_k'$ resp., co-bounded, with $\partial D_k$, by a single arc of $\partial \Delta_k\cap S$, $\delta_k$, and $\partial \Delta_k'\cap S$, $\delta_k'$. Each of the arcs $\delta_k$ and $\delta_k'$ is in a sequence of arcs after an outermost arc, as in Figure \ref{2D1doutermostarcsSI.pdf}(b). As observed before in this Claim, we are assuming that these arcs are parallel in $P$. Then one of the disks $O_k$ or $O_k'$ has to be contained in the other, which is a contradiction with the assumption that $R_k$ and $R_k'$ are disjoint. So, assume that, say, $R_k'$ is contained in $R_k$, as in  Figure \ref{2D1doutermostarcsSI2.pdf}(b). Then $L_k$ contains $D_{n_2}$ and $L_k'$. Therefore, from the minimality of $|E\cap P|$ and from the arcs of $\partial \Delta_k\cap S$ and $\partial \Delta_k'\cap S$ that have both ends in $D_k$ being parallel in $P$, we have that $\partial \Delta_k$ intersects $S$ in two arcs, one with two ends in $D_k$ and the other with two ends in $D_{k+1}$; similarly, $\partial \Delta_k'$ intersects $S$ in two arcs, one with two ends in $D_k$ and the other with two ends in $D_{k+1}$. Let $O_{k+1}$, resp. $O_{k+1}'$, be the disks cut from $S-int D_{k+1}$, resp. $S-int D_{n_2}$, by $\delta_{k+1}$, resp. $\delta_{k+1}'$, disjoint from $D_k$. As there are no local knots, the string in $R_{k}'\subset R_k$ is trivial. Then, from the minimality of $|S\cap V|$, we have $|O_{k+1}'\cap V|$ equal to $|O_k'\cap V|$. Also, $O_k\cap (V\cap S)$ is the same as $O_k'\cap (V\cap S)$. Therefore, $|O_{k+1}\cap V|$ is bigger than $|O_k\cap (V\cap S)|$. So, we can isotope $D_{k+1}\cup O_{k+1}$ along $R_k$ union the ball $C_{k, k+1}$ to reduce $|S\cap V|$, which is a contradiction.\\
So, assume without loss of generality that all arcs $\delta_{k+1}$ have both ends in $D_{k+1}$.
\begin{figure}[htbp]
\centering
\includegraphics[width=0.75\textwidth]{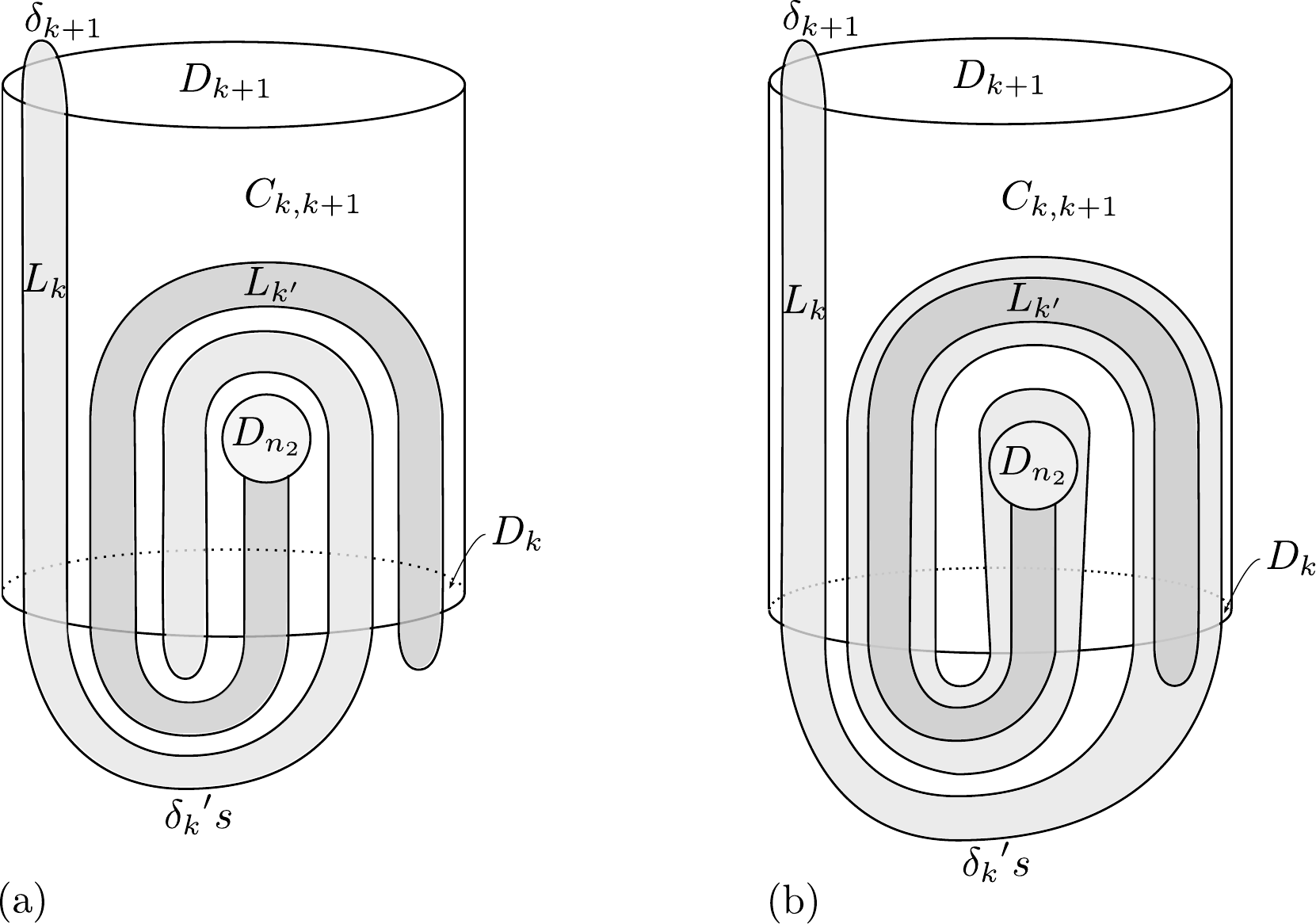}
\caption{: The disks $L$ and $L_k$ when $R_{k}'$ is disjoint from $R_k$ and when $R_k'$ is contained in $R_k$, resp.: the arcs $\delta_k$ won't be parallel in $P$ as previously observed.}                 
\label{2D1doutermostarcsSI2.pdf}
\end{figure}
By the finiteness of outermost arcs we have a sequence of parallel arcs, $\delta_{k+2}, \ldots, \delta_{n_2}$, as in Figure \ref{2D1doutermostarcsSI.pdf}(d), and the respective sequence of balls $R_{k+2}, \ldots, R_{n_2-1}$. Then, we have a sequence of arcs parallel to an outermost arc, $\delta_1, \ldots, \delta_{n_2}$, and the respective balls $R_1, \ldots, R_{n_2-1}$. Following a similar argument as in Case 1, we have that $\delta_1$ is as in Lemma \ref{meridional outermost arc}, which means that some string of some tangle is unknotted.
\end{proof}


\begin{lem}\label{2D*1DNS}
Suppose $V-V\cap S$ contains a solid torus that intersects $\mathcal{D}^*$ at two disks and $\mathcal{D}$ at a single non-separating disk. Then both strings of some tangle are $\mu$-primitive.
\end{lem}
\begin{proof}
Let $T$ be the solid torus component of $V-V\cap S$ as in the statement, and suppose it lies in the tangle $(B_1, \mathcal{T}_1)$. Assume that $T\cap \mathcal{D}^*=D_1^*\cup D_4^*$ and that $T\cap \mathcal{D}=D_1$. 
The disks $D_1^*$ and $D_4^*$ are not parallel, otherwise $D_1$ would be separating. Then, $D_1\cup D_1^*\cup D_4^*$ separate a ball from $V$, as in Figure \ref{2D1d.pdf}, and all outermost disks are over $T$ with corresponding outermost arcs attached to $D_1^*$, $D_4^*$ or $D_1$. The disks $D_1, D_2, \ldots, D_{n_2}$ are all parallel and non-separating in $V$.

\begin{figure}[htbp]
\centering
\includegraphics[width=\textwidth]{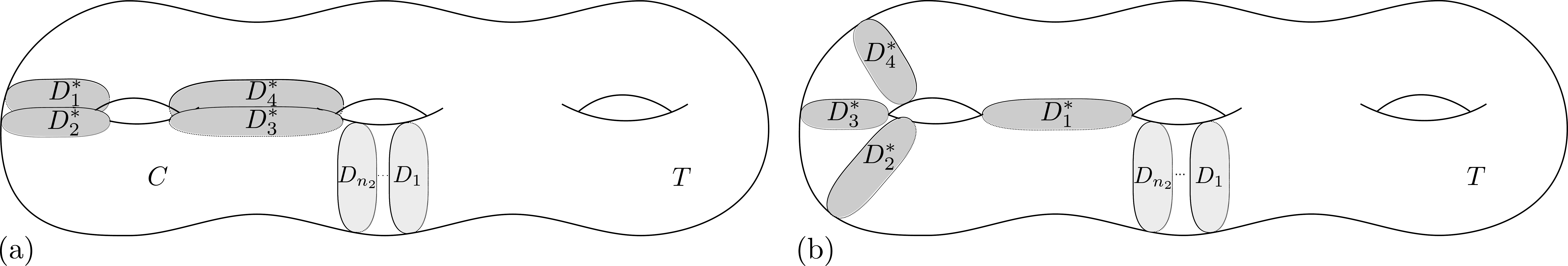}
\caption{}                 
\label{2D1d.pdf}  
\end{figure}

\begin{claim}\label{2D*1DNS claim 0}
If the disks of $\mathcal{D}^*$ are parallel two-by-two then some string of some tangle is unknotted.
\end{claim}
\textit{Proof of Claim \ref{2D*1DNS claim 0}.}
Assume that $D_2^*$ is parallel to $D_1^*$ and that $D_3^*$ is parallel to $D_4^*$ in $V$, as in Figure \ref{2D1d.pdf}(a). If $D_1^*$ or $D_4^*$ are the only disks with outermost disks attached then by the finiteness of outermost arcs we have parallel sk-arcs in $E$,  as in Figures \ref{T2outermostarcs.pdf}(a1), (a2), which is a contradiction to Corollary \ref{no parallel sk-arcs}. So, $D_1$ has an outermost arc attached. Furthermore, from Corollary \ref{no parallel sk-arcs}, even if $D_1^*$, or $D_4^*$, has outermost arcs attached we cannot have a sequence of parallel sk-arcs in $E$ after such outermost arcs. So, by the finiteness of outermost arcs only some outermost arc attached to $D_1$ is before a sequence of parallel arcs of $E\cap P$ in $E$, as in Figure \ref{2D1doutermostarcsSI.pdf}(a). Consider a second-outermost arc $\gamma$, and the disk component of $E-E\cap P$, $\Gamma$, co-bounded by $\gamma$ in the outermost side of this arc in $E$, as in Figure \ref{T2outermostarcs.pdf}(b). The boundary of $\Gamma$ intersects $S$ in $\gamma$ and arcs with both ends in $D_{n_2}$. If $\gamma$ has at least one end in $D_{n_2}$, or one end in $D_2^*$ and the other in $D_3^*$, then from Lemma \ref{no gamma with equal ends} we have that some string in some tangle is unknotted. Otherwise, the ends of all second outermost arcs are both in $D_2^*$ or both in $D_3^*$, and by the finiteness of outermost arcs we have a contradiction to Corollary \ref{no parallel sk-arcs}.\hspace*{\fill}$\triangle$\\

From this claim, we can assume that the disks of $\mathcal{D}^*$ are not parallel two-by-two in $V$. Therefore, as no disk of $\mathcal{D}^*$ can be parallel in $V$ to a disk of $\mathcal{D}$, without loss of generality, we assume that the disks $D_2^*$ and $D_3^*$ are parallel to $D_4^*$, as in Figure \ref{2D1d.pdf}(b). Under this setting, we continue the lemma's proof in several steps with respect to which disks are attached to outermost arcs and to the value of $n_2$.

\begin{claim}\label{2D*1DNS claim 1} 
The disks $D_1$ or $D_1^*$ have outermost arcs attached; and the disks $D_1^*$ and $D_4^*$ cannot have simultaneously outermost arcs attached.
\end{claim}
\textit{Proof of Claim \ref{2D*1DNS claim 1}.}
If all outermost arcs are attached to $D_4^*$ then there is a sequence of parallel sk-arcs, as in Figure \ref{T2outermostarcs.pdf}(a2), which is a contradiction to Corollary \ref{no parallel sk-arcs}. Then $D_1$ or $D_1^*$ have outermost arcs attached.\\
Suppose $D_1^*$ and $D_4^*$ have simultaneously outermost arcs attached. Let $\delta_i^*$ be an outermost arc attached to $D_i^*$, and $\Delta_i^*$ the respective outermost disk, for $i=1, 4$. Consider also the disjoint  disks $O_1^*$ and $O_4^*$, in $S-int\{D_1^*\cup D_4^*\}$, co-bounded by $\delta_1^*$ and $\delta_4^*$, respectively. Let $L_i^*=\Delta_i^*\cup O_i^*$, for $i=1, 4$. As the arcs $\delta_i^*$ are sk-arcs, $D_2^*\cup D_3^*$ is in $O_1^*\cup O_4^*$. Taking a regular neighborhood of the disks $L_i^*$ together with $T$, and by capping off the boundary component of $N(T)\cup_{i=1, 4}N(L_i^*)$ disjoint from $S$ with the ball it bounds, we get a ball $Q$ in the tangle $(B_1, \mathcal{T}_1)$ containing both strings $s_{14}$ and $s_{23}$. Each string of $\mathcal{T}_1$ in $Q$ has ends in two distinct disk components of $\partial Q\cap S$,  $D_1^*\cup O_1^*$ and $D_4^*\cup O_4^*$. Then with the tangle $(Q, \mathcal{T}_1)$ we have a contradiction between Lemma \ref{no ball containing T} and Lemma \ref{inner ball}(c).
\hspace*{\fill}$\triangle$\\

\begin{claim}\label{2D*1DNS claim 2}
If $D_1$ or $D_1^*$ is not attached to outermost arcs then $n_2\leq 3$.
\end{claim}
\textit{Proof of Claim \ref{2D*1DNS claim 2}.}
If $D_1$ or $D_1^*$ is not attached to outermost arcs then all outermost d-arcs have either both ends in $D_{n_2}$ or in $D_1$. Then by the finiteness of outermost arcs there is a sequence of parallel arcs, $\delta_i$, as in Figure \ref{2D1doutermostarcsSI.pdf}(a), parallel to some outermost d-arc attached to $D_{n_2}$ or $D_1$. As in Case 1 of Lemma \ref{2D*1DS}, using the disks $\Delta_i$ between the arcs $\delta_i$ and $\delta_{i+1}$ in $E$, attached to the disks $D_i$ and $D_{i+1}$, resp., and the disk that $\partial \Delta_i$ bounds in the torus $C_{i, i+1}\cup_{D_i\cup D_{i+1}}S$, we define a ball $R_i$. Each of these balls contains a single string of the tangle decomposition and it is regular neighborhood of it. If $n_2\geq 5$ then all components of $V-S\cap V$ are contained in some ball $R_i\cup C_{i, i+1}$. We note that these balls are either disjoint or intersect at a disk, wether the strings they contain are disjoint or intersect at an end. Then, taking the union of the largest balls $R_i\cup C_{i, i+1}$ for each string, we have a solid torus with $K$ as its core, $V$ in its interior and boundary essential in $W$, which is a contradiction as $W$ is a handlebody. So, given that $n_2$ is odd, $n_2\leq 3$.\\
(If both $D_1$ and $D_1^*$ have outermost arcs attached, in Claim \ref{2D*1DNS claim 4} we also prove that $n_2\leq 3$.)
\hspace*{\fill}$\triangle$\\

\begin{claim}\label{2D*1DNS claim 3}
If  $D_4^*$ and $D_1$ are the only disks with outermost arcs attached then some string of some tangle is unknotted.
\end{claim}
\textit{Proof of Claim \ref{2D*1DNS claim 3}.}
Suppose both disks $D_1$ and $D_4^*$ are attached to outermost arcs, $\delta_1$ and $\delta_4$, resp., . If $n_2=1$ then either $\delta_1$ or $\delta_4$ are as in Lemma \ref{meridional outermost arc}, which means that some string of some tangle is unknotted.\\
So, from Claim \ref{2D*1DNS claim 2}, we can assume that $n_2=3$. From Corollary \ref{no parallel sk-arcs} there are no parallel sk-arcs after an outermost arc attached to $D_4^*$, as in Figure \ref{T2outermostarcs.pdf}(a2). Consequently, from the finiteness of outermost arcs, we have such a sequence of parallel arcs after an outermost arc attached to $D_1$, as in Figure  \ref{2D1doutermostarcsSI.pdf}(a), and consider the respective balls $R_i$, for $i=1, 2$. Let $O_i$ and $O_{i+1}$ be the disk components of $R_i\cap S$ that are co-bounded by $\delta_i$ and $\delta_{i+1}$, resp., for $i=1, 2$. As $R_1$ contains a single string, $O_1$ intersects $\mathcal{D}^*$ at a single disk. Then, as $D_4^*$ has a type I arc attached, this disk is not in $O_1$. If $D_2^*$ or $D_3^*$ are in $O_1$ then $R_2$ contains $T$ and, consequently, two strings of the tangle, which we know is impossible. Then $D_1^*$ is in $O_1$, the string $s_{12}$ is in $R_1$ and the string $s_{23}$ is in $R_2$. So, if $D_2$ or $D_3$ is in $O_1$, also $O_2$ or $O_3$ will be, and consequently the same for $D_2^*$ or $D_3^*$, which is impossible as observed before. Then, $O_1\cap (S\cap V)$ is only $D_1^*$. So, $\delta_1$ is an outermost arc as in Lemma \ref{meridional outermost arc}. Then, some string of some tangle is unknotted.
\hspace*{\fill}$\triangle$\\


\begin{claim}\label{2D*1DNS claim 4}
If  $D_1^*$ and $D_1$ are both attached to outermost arcs then both strings of some tangle are $\mu$-primitive.
\end{claim}
\textit{Proof of Claim \ref{2D*1DNS claim 4}.}
Suppose that both $D_1$ and $D_1^*$ have outermost arcs attached, denoted by $\delta_1$ and $\delta_1^*$ resp.. Let $O_1$ and $O_1^*$ be the disjoint disks in $S-int (D_1\cup D_1^*)$ co-bounded by $\delta_1$ and $\delta_1^*$, resp.. Consider also the disks $L_1^*=\Delta_1^*\cup O_1^*$ and $L_1=\Delta_1\cap O_1$. Let $Q$ be the ball obtained by adding a regular neighborhood of $L_1^*$ and $L_1$ to $T$, together with the ball that the boundary component of $N(T)\cup N(L_1)\cup N(L_1^*)$, disjoint from $S$, bounds. As $\delta_1^*$ and $\delta_1$ are an sk-arc and an st-arc, we have that $O_1^*$ and $O_1$ intersect $\mathcal{D}^*$. As $D_1^*$ is not in $O_1^*\cup O_1$, in this particular case $D_2^*\cup D_3^*$ is necessarily in $O_1^*\cup O_1$, and the string $s_{23}$ is also in $Q$. The disk $D_4^*$ may or not be in $O_1^*\cup O_1$.\\
If $D_4^*$ is in $O_1^*\cup O_1$ then $Q$ intersects $S$ in two components: $D_1^*\cup O_1^*$ and $D_1\cup O_1$. From Lemma \ref{no ball containing T}, the tangle $(Q, \mathcal{T}_1)$ is trivial, which is a contradiction to Lemma \ref{inner ball}(c).\\
So, we can assume that $D_4^*$ is not in $O_1^*\cup O_1$ and $Q$ intersects $S$ in three component disks: $D_4^*$, $D_1^*\cup O_1^*$ and $D_1\cup O_1$. Also, $O_1\cap \mathcal{D}^*$, and $O_1^*\cap\mathcal{D}^*$, is either $D_2^*$ or $D_3^*$. Furthermore, from Lemma \ref{no ball containing T}, both strings $s_{14}$ and $s_{23}$ are trivial in $Q$.\\ 
If $n_2=1$ then $\delta_1$ is as in Lemma \ref{meridional outermost arc}, which means that some string of some tangle is unknotted. So we can assume that $n_2\geq3$.\\
Suppose there is a sequence of parallel arcs in $E$ after an outermost arc $\delta_1$, $\delta_{2}, \ldots, \delta_{n_2}$, as in Figure \ref{2D1doutermostarcsSI.pdf}(a), and consider the balls $R_i$ as in Case 1 of Lemma \ref{2D*1DS}. Then, as $O_1\cap \mathcal{D}^*$ is either $D_2^*$ or $D_3^*$ we have that the ball $R_2$ contains two strings, which is a contradiction to the balls $R_i$ containing a single string. Consequently, there is no sequence of parallel arcs in $E$, $\delta_{2}, \ldots, \delta_{n_2}$, after an outermost arc $\delta_1$.\\
Consider an arc parallel to an outermost arc $\delta_1^*$ or otherwise a second-outermost arc, $\gamma$, and denote by $\Gamma$ the disk of $E-E\cap S$, co-bounded by $\gamma$, in the outermost side of this arc in $E$. (See Figure \ref{2D1doutermostarcsB.pdf}(a).) As there is no sequence $\delta_2, \ldots, \delta_{n_2}$ after the outermost arcs $\delta_1$, as in Figure \ref{2D1doutermostarcsB.pdf}(a), we have that $\Gamma$ intersects $S$ in $\gamma$ and outermost arcs $\delta_1^*$. Note that $\gamma$ cannot have only one end in $D_{n_2}$, otherwise $\gamma$ would be a t-arc, which is a contradiction to Lemma \ref{property}(c). If $\gamma$ has two ends in $D_1^*$ or one end in $D_1^*$ and the other end in $D_2^*$, following reasoning as in the proof of Lemma \ref{no gamma with equal ends}, we have that some string in some tangle is unknotted. If the ends of all arcs $\gamma$ are both in $D_2^*$, by the finiteness of outermost arcs, we have a contradiction to Corollary \ref{no parallel sk-arcs}. Hence, we can assume that some arc $\gamma$ has both ends at $D_{n_2}$.\\
Let $O_{n_2}$ be the disk in $S-int D_{n_2}$ cut by $\gamma$, disjoint from $D_1^*$.
\begin{figure}[htbp]
\centering
\includegraphics[width=0.85\textwidth]{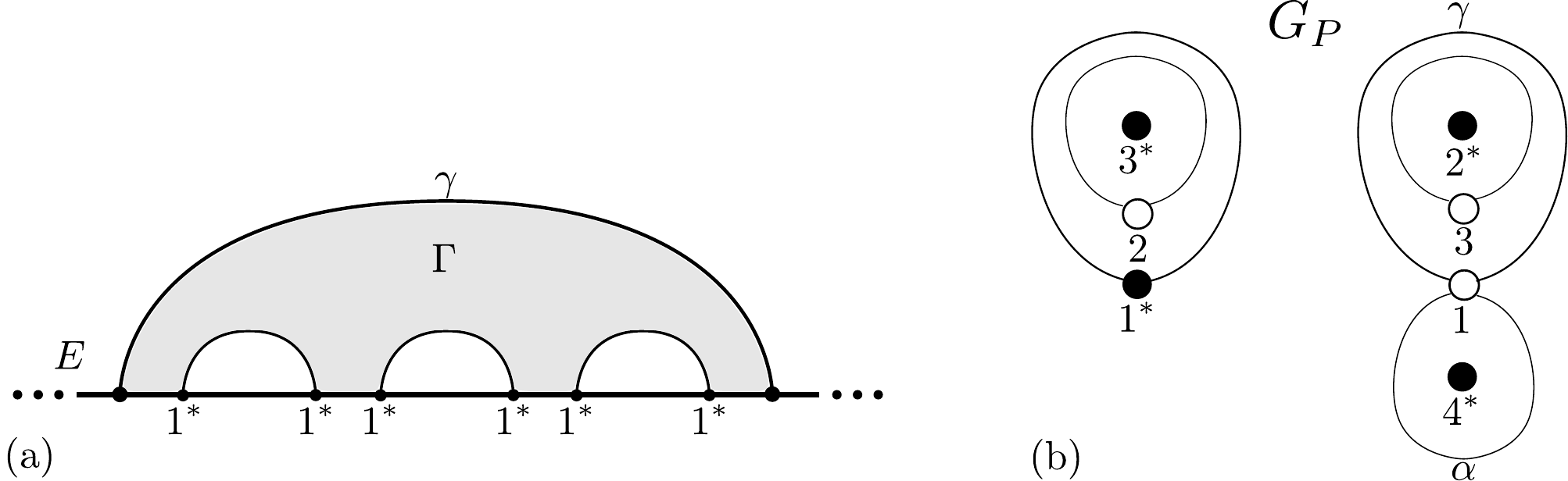}
\caption{}                 
\label{2D1doutermostarcsB.pdf}
\end{figure}
Denote by $C$ the ball cut from $V$ by $D_1^*\cup D_2^*\cup D_{n_2}$, and by $C_{1^*, n_2}$ the cylinder obtained from $C$ by an isotopy of $C$ along $D_2^*$ away from $S$. Note that $C$ is in $B_2$. Consider the disk $L$ bounded by $\partial \Gamma$ in the torus $\partial C_{1^*, n_2}\cup_{D_1^*\cup D_{n_2}} S$. Let $R$ be the ball bounded by $\Gamma\cup L$ in $B_2$. If $R$ intersects $K$ in two components, then we can prove that $\gamma$ is parallel to $\delta_1^*$ in $E$. By taking $R$ together with $C_{1^*, n_2}$ we define a cylinder containing the two strings of $\mathcal{T}_2$ with ends in the disks $D_1^*\cup O_1^*$ and $D_{n_2}\cup O_{n_2}$. Then from Lemma \ref{inner ball}(a), (c), and because $\partial C_{1^*, n_2}-L\cap \partial C_{1^*, n_2}$ is a single disk containing $D_1^*\cup D_{n_2}$, we obtain a contradiction to the minimality of $|S\cap V|$. So, we have that $R$ intersects $\mathcal{T}_2$ at a single component. Naturally $O_{n_2}\subset L$, and also $O_{n_2}\cap \mathcal{D}^*$ is $D_2^*$. In fact, if $D_3^*$ is in $O_{n_2}$ then, $s_{34}$ is in $R$. As $R$ intersects $\mathcal{T}_2$ at a single component, and $O_1^*$ intersects $\mathcal{D}^*$, we have $D_4^*$ in $O_1^*$, which contradicts our assumption that $O_1^*\cap\mathcal{D}^*$ is only $D_2^*$ or $D_3^*$. If $D_4^*$ is in $O_{n_2}$ then, following a similar reasoning, $D_3^*$ is in $O_1^*$ and $D_2^*$ is in $O_1$. As before, with the existence of parallel arcs to $\gamma$ or $\delta_1$ in $E$ we can define the balls $R_{n_2-1}$ or $R_1$. But then, in this case, $R_1$ or $R_{n_2}$ contain two strings, which is a contradiction. Then, $D_2^*$ is in $O_{n_2}$. As $O_{n_2}$ is disjoint from $O_1^*$, and $O_1^*\cap \mathcal{D}^*$ is either $D_2^*$ or $D_3^*$, we have that $D_3^*$ is in $O_1^*$. Then, $D_2^*$ is in $O_1$ and if $R_1$ exists it has two strings, which is impossible. So, we can assume that there is a sequence of balls $R_{n_2-1},\ldots, R_2$ exists, related to a sequence of parallel arcs of $E\cap P$ to $\gamma$ in $E$, $\delta_{n_2-1}, \ldots, \delta_2$. As $O_{n_2}$ contains $D_2^*$, if $n_2\geq 5$ we have that the ball $R_{n_2-3}$ contains $T$ and consequently two strings, which is a contradiction. Therefore $n_2=3$, and the ball $R_2$ contains the string $s_{23}$. But $R_2$ cannot contain $T$, otherwise it would contain two strings. Hence, $O_2\subset O_1^*$ and $O_3\subset O_1$. (See Figure \ref{2D1doutermostarcsB.pdf}(b).)\\
Consider an arc $\alpha$ outermost after the outermost arcs $\delta_1$ and parallel arcs to $\gamma$.
Then $\alpha$ has ends in $D_1\cup D_2$. If the arcs $\alpha$ have one end in $D_1$ and the other in $D_2$ then we get a contradiction to $D_2\subset O_1^*$ and $O_1$ being disjoint from $O_1^*$. Then $\alpha$ has equal ends. If the ends of $\alpha$ are in $D_2$ then $\alpha$ is in $O_1^*$ (because $D_2$ is in $O_1^*$). All loops attached to $D_2$, as $\alpha$, have to be parallel in $P$ to the arc parallel to $\gamma$ in $P$ attached to $D_2$. Otherwise, $D_4^*$ is contained in $O_1^*$, which contradicts the assumption that it is not. Let $A$ be the disk of $E-E\cap P$ co-bounded by $\alpha$ in the outermost side of the arc in $E$. Suppose $\alpha$ is attached to $D_2$ or is parallel to $\delta_1$ in $P$. The boundary of $A$ bounds a disk in $S \cup_{D_1\cup D_2} \partial C_{1, 2}$ that contains $O_1$, and the union of these two disks bounds a ball, $R_1'$, in $B_2$. The ball $R_1'$ has similar properties to the balls $R_i$; including containing a single string of $\mathcal{T}_2$, which is a consequence of Lemma \ref{inner ball} (a), (c), the arcs $\partial A\cap S-\gamma$ with both ends in $D_1$ and $D_2$ being parallel in $P$ resp., and also from the minimality of $|S\cap V|$. But has $R_1'$ contains $O_1$, it also contains two strings, which a contradiction to the previous observation. Then, $\alpha$ is attached to $D_1$ and is not parallel to $\delta_1$. In this case, $R_1'$ contains the string $s_{34}$ as a core, that is parallel to the core of the cylinder $C_{1, 2}$. Consider the outermost arcs $\gamma'$ among the arcs of $E\cap P$ with distinct ends in $D_1^*\cup \mathcal{D}$. Given the configuration of $G_P$, as in Figure \ref{2D1doutermostarcsB.pdf}(b), the only possible ends for $\gamma'$ are one end in $D_1^*$ and the other in $D_1$, one end in $D_1^*$ and the other in $D_2$ and one end in $D_1$ and the other in $D_3$. The only possible case, because the disks involved belong to the same component of $V-V\cap S$, is having $\gamma'$ with one end in $D_1^*$ and the other in $D_1$. Let $\Gamma'$ be the disk, of $E-E\cap S$, co-bounded by $\gamma'$, in the outermost side of $\gamma'$ in $E$. Then $\Gamma'$ is over $Q$ and $S$, in $B_1$. All the arcs of $\partial \Gamma'\cap S$ that intersect $D_1^*$ are either $\gamma'$ or have both ends in $D_1^*$ and are parallel to $\delta_1^*$ in $P$. By an isotopy of these arcs to $Q$ we get that $D_1^*\cup O_1^*$ is primitive with respect to the complement of $Q$ in $B_1$, that is a handlebody. Then the core of the cylinder from $D_4^*$ to $D_1$ is unknotted. As the string $s_{14}$ is parallel to the core from $D_1^*$ to $D_4^*$ in $Q$ and the string $s_{23}$ is parallel to the core from $D_1^*$ to $D_1$ in $Q$, we have that both strings are $\mu$-primitive.
\hspace*{\fill}$\triangle$\\

\begin{claim}\label{2D*1DNS claim 5}
If only $D_1^*$ is attached to outermost arcs then some string of some tangle is unknotted.
\end{claim}
\textit{Proof of Claim \ref{2D*1DNS claim 5}.}
Denote by $\delta_1^*$ the outermost arcs attached to $D_1^*$. Consider a second outermost arc, $\gamma$, and let $\Gamma$ be the disk of $E-E\cap V$ co-bounded by $\gamma$ in the outermost side of this arc in $E$. (See Figure \ref{2D1doutermostarcsB.pdf}(a).) The curve $\partial \Gamma$ bounds a disk $L$ in the torus $S\cup_{D_1^*\cup D_{n_2}} \partial C_{1^*, n_2} $.  Following a similar argument as in Claim \ref{2D*1DNS claim 4}, we can assume $\gamma$ has both ends in $D_{n_2}$ and we define similarly the ball $R$ in $B_2$ with boundary $\Gamma\cup L$. So, either the string $s_{34}$ or a portion the string $s_{12}$ with end in $D_2^*$ is in $R$, and therefore, this string is parallel to the core of the cylinder $C_{1^*, n_2}$. Let $O_1^*$ and $O$ be the disjoint disks in $S^3-int \{D_1^*\cup D_{n_2}\}$ co-bounded by $\delta_1^*$ and $\gamma$, resp.. Note that $O$ is in $L\subset \partial R$. As $R$ contains a single string, we have that $O$ intersects $\mathcal{D}^*$ at a single disk. From Claim \ref{2D*1DNS claim 2}, we have $n_2\leq 3$; also, when $n_2=3$ we consider the balls $R_1$, $R_2$ and the respective disks of intersection with $S$, $O$, $O_2$ and $O_1$, attached to $D_3$, $D_2$ and $D_1$, resp. .\\

Assume $R$ contains the string $s_{34}$. In this case $O_1^*$ is in $R$, and each $O$ and $O_1^*$ contain a single disk of $\mathcal{D}^*$, $D_3^*$ or $D_4^*$. Then, if $n_2=3$ one of the balls $R_1$ or $R_2$ contains two strings of a tangle, which is impossible. Hence, $n_2=1$. As $O_1^*$ is disjoint from $D_1$ we have that $O_1^*$ intersects $S\cap V$ at a single disk of $\mathcal{D}^*$. Therefore, some arc $\delta_1^*$ is as in Lemma \ref{meridional outermost arc}, which means that some string of some tangle is unknotted.\\

Assume now that $R$ contains a portion of the string $s_{12}$.

Suppose $n_2=3$. We have $O\cap \mathcal{D}^*=D_2^*$ and consequently $s_{23}$ is in $R_2$ and $s_{34}$ is in $R_1$, which means that $O_2\cap \mathcal{D}^*=D_3^*$ and $O_1\cap \mathcal{D}^*=D_4^*$, as in Figure \ref{2D1doutermostarcsF.pdf}(a). Consider an outermost arc, $\alpha$, between the arcs with ends in distinct disk components of $D_1^*\cup \mathcal{D}$, and $A$ the disk of $E-E\cap P$ co-bounded by $\alpha$ in its outermost side in $E$. Note that $\alpha$ can only have ends in disks in the same component of $V-V\cap S$. So, $\alpha$ can only have ends in $D_1^*$ and $D_1$, $D_3$ and $D_2$, $D_2$ and $D_1$, or also, $D_1^*$ and $D_3$, as in Figure \ref{2D1doutermostarcsF.pdf}(b).\\ 
If the ends of $\alpha$ are in $D_1^*$ and $D_3$, $D_3$ and $D_2$, or $D_2$ and $D_1$, then the strings $s_{12}$, $s_{23}$ or $s_{34}$ are unknotted, respectively.\\
So, assume that all arcs $\alpha$ have ends in $D_1^*$ and $D_1$.  
\begin{figure}[htbp]
\centering
\includegraphics[width=0.9\textwidth]{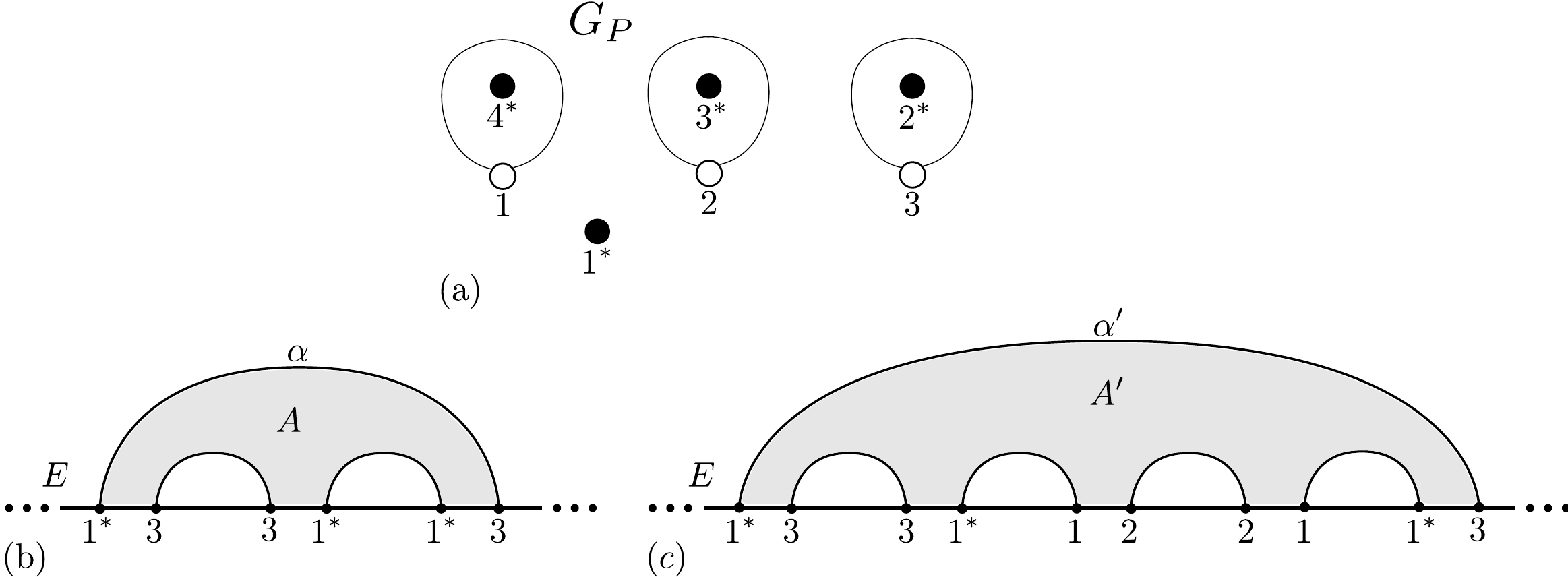}
\caption{}                 
\label{2D1doutermostarcsF.pdf}
\end{figure}
Consider now the outermost arc $\alpha'$ between the ones with ends in distinct components of $D_1^*\cup \mathcal{D}-D_1$ or that have ends in distinct components of $\mathcal{D}$. Let $A'$ be the disk of $E-E\cap S$ co-bounded by $\alpha'$ in the outermost side of the arc in $E$. (See Figure \ref{2D1doutermostarcsF.pdf}(c).) The arc $\alpha'$ can only connect components of $V-V\cap S$ with the disks $D_1^*$ and $D_1$ in them . Hence, the disk $A'$ is in the tangle with the strings $s_{12}$, $s_{34}$. Using the disk $A'$ and depending on the ends of $\alpha'$ we can prove that $s_{12}$ or $s_{34}$ is unknotted.

Suppose $n_2=1$. Suppose that $O_1^*\cap \mathcal{D}^*$ is either $D_3^*$ or $D_4^*$. As $O_1^*$ and $D_1$ are disjoint, $\delta_1^*$ is as in Lemma \ref{meridional outermost arc}, which means that some string of some tangle is unknotted.\\
Suppose, now, that $O_1^*$ intersects $\mathcal{D}^*$ in $D_3^*\cup D_4^*$, as in Figure \ref{2D1doutermostarcsG.pdf}(c).
\begin{figure}[htbp]
\centering
\includegraphics[width=0.85\textwidth]{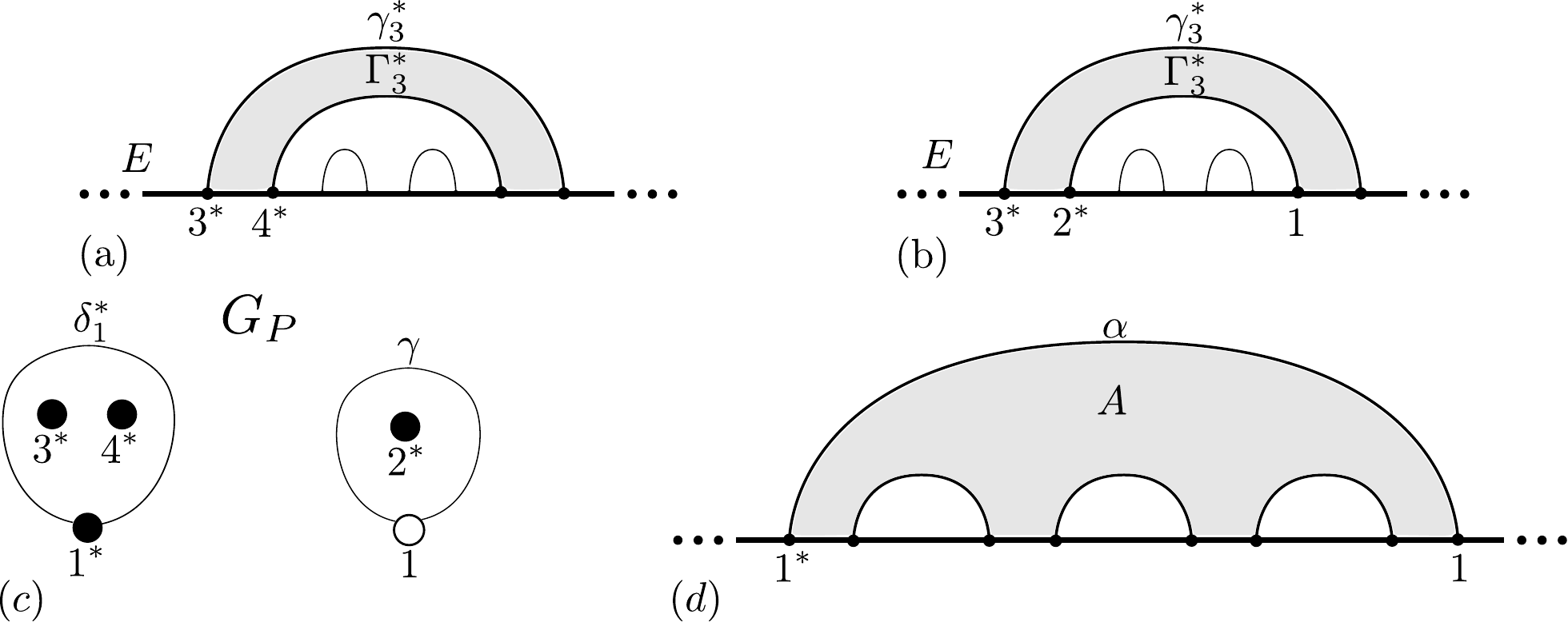}
\caption{}                 
\label{2D1doutermostarcsG.pdf}
\end{figure} 
Consider the arcs $\gamma_3^*$ and $\gamma_4^*$, and the respective disks $\Gamma_3^*$ and $\Gamma_4^*$.  From Lemma \ref{property}(b), the two disks $D_3^*$ or $D_4^*$ cannot have simultaneously loops attached in $G_P$. Then, all arcs $\gamma_3^*$ or all arcs $\gamma_4^*$ have distinct ends.\\ 
Assume that all arcs $\gamma_3^*$ have distinct ends. Suppose also that some $\Gamma_3^*$ intersects $D_4^*$ as in Figure \ref{2D1doutermostarcsG.pdf}(a). Then the disks $D_3^*$ and $D_4^*$ are primitive with respect to the complement of $V\cap B_2$ in $B_2$. Consequently, the complement of $C_{1^*, 1}$ in $B_2$ is a solid torus.  As $s_{12}$ is parallel to the core of the ball $C_{1^*, 1}$ we have that $s_{12}$ is unknotted. Otherwise, suppose that  all disks $\Gamma_3^*$ intersect $D_2^*$ as in Figure \ref{2D1doutermostarcsG.pdf}(b). Then, the disks $D_2^*$ and $D_3^*$ are primitive with respect to the complement of $V\cap B_1$ in $B_1$.
Consider an outermost arc $\alpha$ between the arcs with one end in $D_1^*$ and the other end in $D_1$. Let $A$ be the disk, of $E-E\cap V$, co-bounded by $\alpha$ in the outermost side of $\alpha$ in $E$, as in Figure \ref{2D1doutermostarcsG.pdf}(d). Suppose that $A$ is in $B_2$. The components of $\partial A\cap S$ that intersect $D_1$ are $\alpha$ and eventually arcs with both ends in $D_1$ parallel to $\gamma$. The disk $D_2^*$ is primitive in the complement of $V\cap B_2$ in $B_2$. Then after adding the $2$-handle with core $D_2^*$ to the complement of $V\cap B_2$ in $B_2$ we are left with the complement of $C_{1^*, 1}\cup C_{3^*, 4^*}$. We isotope the arcs of $A \cap S$ parallel to $\gamma$, through $O$, to the boundary of the cylinder $C_{1^*, 1}$. 
\begin{figure}[htbp]
\centering
\includegraphics[width=0.85\textwidth]{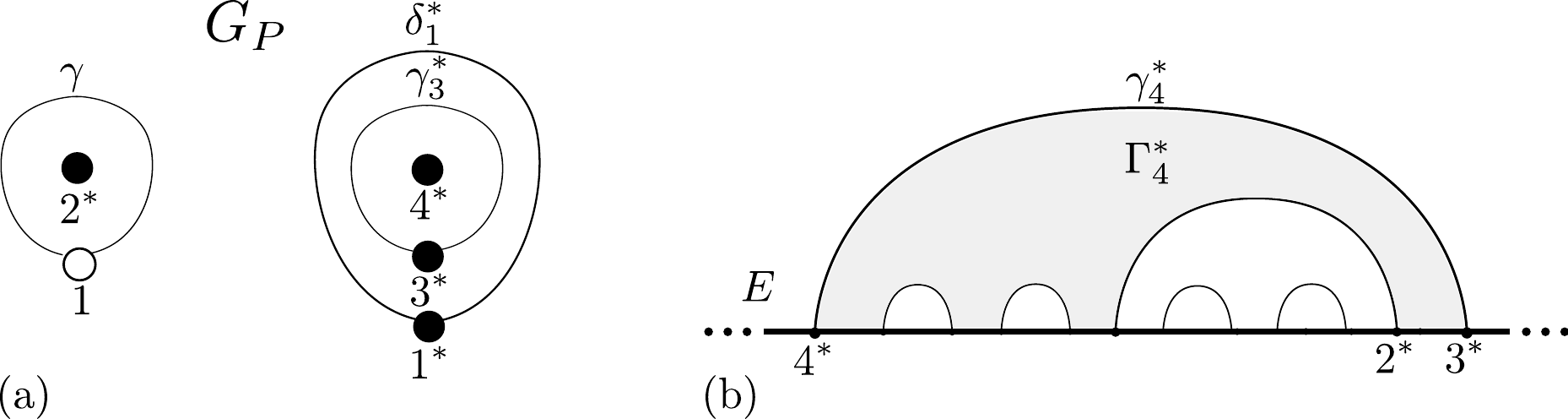}
\caption{}                 
\label{2D1doutermostarcsH.pdf}
\end{figure}
After this isotopy, $A$ intersects $D_1$ geometrically once. Then, the complement of $C_{3^*, 4^*}$ in $B_2$ is a solid torus, which means that the string $s_{34}$ is unknotted. Otherwise, assume that $A$ is in $B_1$. The components of $\partial A\cap S$ that intersect $D_1$ are $\alpha$ and eventually arcs with both ends in $D_1$ parallel to $\gamma$, or arcs with one end in $D_1$ and the other in $D_2^*$. As $\Gamma_3^*$ is in $B_1$, we have that $A$ doesn't intersect any arc $\gamma_3^*$. Then we can proceed as follows. Take $T$ union with a regular neighborhood of $O$. Isotope to $N(D_1\cup O)$ the arcs of $A\cap S$ parallel to $\gamma$. Then, the disk $A$ intersects $D_1\cup O$ geometrically once, and $\Gamma_3^*$ intersects $D_2^*$ geometrically once. As $A$ is disjoint from any $\gamma_3^*$, cut $T\cup N(O)$ along $D_2^*$ and, afterwards, we isotope  $T\cup N(O)$ along $D_1\cup O$ away from $S$. Denote the solid torus after the isotopy as $T'$. Then, the complement of $T'$ in $B_1$ is a handlebody.  Let ${O_1^*}^c$ be the complement of $O_1^*$ in $S-int D_1^*$. Denote by $Q'$ the ball obtained by adding the two handle with core ${O_1^*}^c\cup \Delta_1^*$ to $T'$. The ball $Q'$ intersects $S$ in $D_1^*\cup {O_1^*}^c$ and $D_4^*$, and its complement in $B_1$ is a solid torus. The ball $Q'$ contains $s_{14}$ and intersects the string $s_{23}$ at an unknotted component. Then, by Lemma \ref{inner ball}(b), either one string of the tangle decomposition given by $S$ is unknotted or the tangle in $(Q', Q'\cap \mathcal{T}_1)$ is trivial. Hence, we can assume the latter and that the string $s_{14}$ is trivial in $Q'$. As the string $s_{14}$ has one end in each of the two components of $Q'\cap S$, it is a core of the cylinder $Q'$. Consequently, the string $s_{14}$ is unknotted.\\
Suppose now that some $\gamma_3^*$ has identical ends. Then, all arcs $\gamma_4^*$ have distinct ends, and from Figure \ref{2D1doutermostarcsH.pdf}(a), the other end of $\gamma_4^*$ is in $D_3^*$. As $\gamma_4^*$ is the outermost $\text{d}^*$-arc with one end in $D_4^*$ and $\gamma_4^*$ has one end in $D_3^*$, we have that $\Gamma_4^*$ intersects $D_2^*$ once. (See Figure \ref{2D1doutermostarcsH.pdf}(b).) This means that $\Gamma_4^*$ is in $B_1$ and that $D_2^*$ and $D_3^*$ are primitive with respect to the complement of $V\cap B_2$ in $B_2$. Then, considering the arc $\alpha$ and disk $A$ and proceeding as before, we have that some string of some tangle is unknotted.\hspace*{\fill}$\triangle$\\

\begin{claim}\label{2D*1DNS claim 6}
If only $D_1$ is attached to outermost arcs then some string of some tangle is unknotted.
\end{claim}
\textit{Proof of Claim \ref{2D*1DNS claim 6}.}
Let $\delta_1$ denote the outermost arcs attached to $D_1$. If $n_2=3$ by the finiteness of outermost arcs there is a sequence of parallel arcs to some $\delta_1$, that is $\delta_2$ and $\delta_3$, as in Figure \ref{2D1doutermostarcsSI.pdf}(a), and with this sequence we can consider the balls $R_i$ as in Case 1 of Lemma \ref{2D*1DS}. Let $\gamma$ be a second-outermost arc of $E\cap P$ in $E$, as in Figure \ref{T2outermostarcs.pdf}(b). From Lemma \ref{no gamma with equal ends}, if $\gamma$ has one end in $D_{n_2}$ or one end in $D_1^*$ and the other in $D_2^*$ then some string of some tangle defined by $S$ is unknotted. If all arcs $\gamma$ have both ends in $D_2^*$, then, by the finiteness of outermost arcs, we have a contradiction to Corollary \ref{no parallel sk-arcs}. Then some arc $\gamma$ has both ends in $D_1^*$. Consider this arc $\gamma$ and let $\Gamma$ be the disk component of $E-E\cap P$ co-bounded by $\gamma$ in the outermost side of this arc in $E$. The disk $\Gamma$ bounds a disk $L$ in $\partial C_{1^*, n_2}\cup_{D_1^*\cup D_{n_2}} S$ that together with $\Gamma$ bound a ball $R$ in $B_2$. As in the previous claim, we have that either a portion of the string $s_{12}$ with end in $D_2^*$ is in $R$, or the string $s_{34}$ is in $R$. Let $O_1^*$ be the disk co-bounded by $\gamma$ in $S-int D_1^*$, disjoint from $D_{n_2}$. Then $O_1^*\subset L$.

Assume $R$ contains a portion of the string $s_{12}$. In this case,  $D_2^*$ is the only disk of $\mathcal{D}^*$ in $L$ and $D_2^*\subset O_1^*$.
 Consider the ball $C_{1^*}$ obtained by an isotopy of the ball $C$, cut from $V$ by $D_1^*\cup D_2^*\cup D_{n_2}$, along $D_2^*\cup D_{n_2}$ away from $S$ in $B_2$. From Lemma \ref{strings parallel to boundary}, the arc $C_{1^*}\cap s_{12}$ is trivial in $C_{1^*}$. We also have that the portion of $s_{12}$ in the complement of $C_{1^*}$ in $B_2$ is unknotted. In fact, we can assume that this arc is $R\cap s_{12}$. As there are no local knots, $R\cap s_{12}$ is trivial in $R$, and therefore, it is parallel to $L$. We can isotope the components of $L\cap S-O_1^*$ from $S$ to $\partial C_{1^*, n_2}$. With the isotopy we verify that the arc $R\cap s_{12}$ is parallel to the boundary of $C_{1^*}$.  Altogether, we have that $s_{12}$ is unknotted in $(B_2, \mathcal{T}_2)$.

Assume now that $R$ contains the string $s_{34}$. Following along an argument of the similar situation in Claim \ref{2D*1DNS claim 5}, we have some string of some tangle is unknotted.
\hspace*{\fill}$\triangle$\\
\end{proof}


\begin{lem}\label{2D*2D}
Suppose $V-V\cap S$ contains a solid torus that intersects $\mathcal{D}^*$ and $\mathcal{D}$ at two disks. Then some string of some tangle is unknotted.
\end{lem}
\begin{proof}
Let $T$ be the solid torus component of $V-V\cap S$ as in the statement, and suppose it lies in the tangle $(B_1, \mathcal{T}_1)$. As the genus of $V$ is three, all disks of $\mathcal{D}^*$ are parallel in $V$, and the same is true for the disks of $\mathcal{D}$. Assume that $\partial T\cap \mathcal{D}^*=D_1^*\cup D_4^*$ and $\partial T\cap\mathcal{D}=D_1\cup D_{n_2}$, as in Figure \ref{2D2d.pdf}). From Remark \ref{remark no beta in ball}, we can assume that all outermost disks are over $T$ with respective outermost arcs attached to $D_1^*$, $D_4^*$, $D_1$ or $D_{n_2}$. If all outermost arcs are attached to $D_1^*$ or $D_4^*$, then by the finiteness of outermost arcs there are parallel sk-arcs in  contradiction to Corollary \ref{no parallel sk-arcs}. Then, some outermost arc is attached to $D_1$ or $D_{n_2}$. Furthermore, even if $D_1^*$ or $D_4^*$ is attached to outermost arcs the only sequences of arcs parallel to outermost arcs in $E$ are with respect to outermost arcs attached to $D_1$ or $D_{n_2}$. Without loss of generality, we assume that $D_1$ is always attached to some outermost arc, and denote by $\delta_1$ an outermost arc attached to $D_1$.
\begin{figure}[htbp]
\centering
\includegraphics[width=0.5\textwidth]{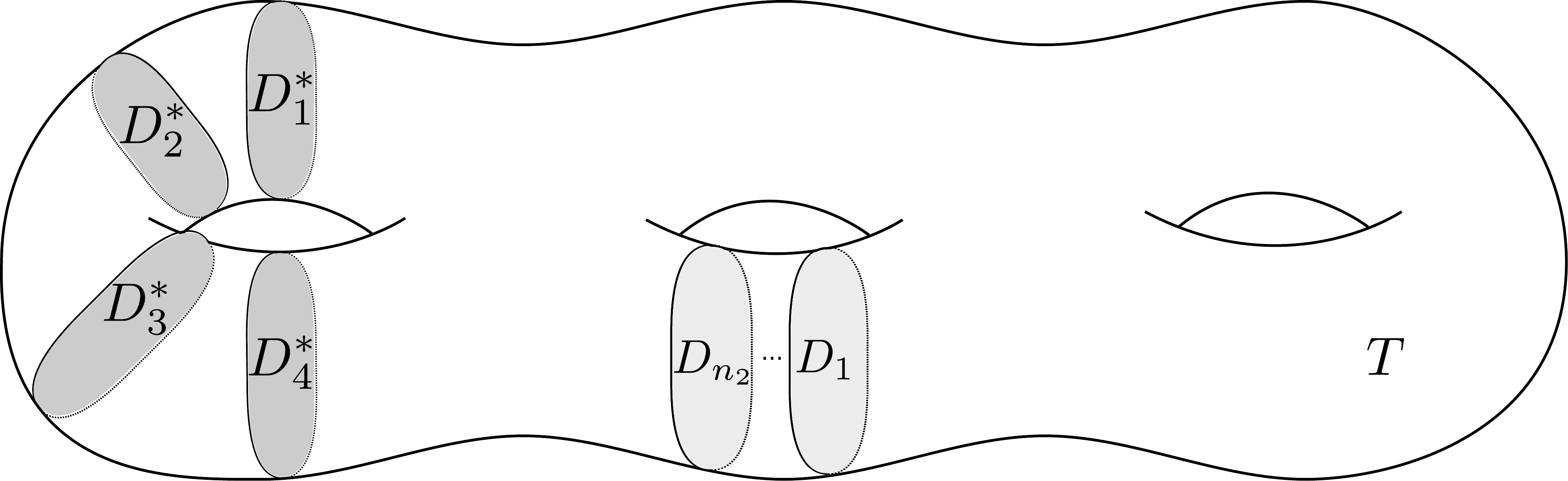}
\caption{}                 
\label{2D2d.pdf}
\end{figure}

\begin{claim}\label{2D*2D claim 1}
If $D_{n_2}$ is not attached to outermost arcs then some string of some tangle is unknotted.
\end{claim}
\textit{Proof of Claim \ref{2D*2D claim 1}.}
Assume that $D_{n_2}$ is not attached to outermost arcs. By the finiteness of outermost arcs and Lemma \ref{no parallel sk-arcs}, there is a sequence of arcs of $E\cap P$ in $E$ parallel to an outermost arc $\delta_1$, that is $\delta_2, \ldots, \delta_{n_2}$, as in Figure \ref{2D1doutermostarcsSI.pdf}(a). As in Case 1 of Lemma \ref{2D*1DS} we define the balls $R_i$; consider also the respective disks $O_i$ and $O_{i+1}$. As $R_1$ contains a single string we have that $O_1$ intersects $\mathcal{D}^*$ at a single disk. If $n_2=2$ then as $O_1$ and $O_2$ are disjoint, we have that $O_1$ intersects $S\cap V$ at a single disk. Hence, $\delta_1$ is as in Lemma \ref{meridional outermost arc}, which means that some string of some tangle is unknotted. Suppose $n_2\geq 4$. (Note that $n_2$ is necessarily even.) If $D_2^*$ or $D_3^*$ are in $O_1$ then $R_2$ contains two strings, which is impossible. Then, without loss of generality, we can assume that $D_1^*$ is in $O_1$. Suppose some disk of $\mathcal{D}$ is in $O_1$, say $D_i$. Then $D_1^*$ is also in $O_i$. This means that $T$ is in $R_{i-1}$, and consequently, $D_1$ is in $O_i$, which is a contradiction as $D_i$ is in $O_1$. Therefore, $\delta_1$ is under the conditions of Lemma \ref{meridional outermost arc}, which means that some string of some tangle is unknotted.
\hspace*{\fill}$\triangle$\\

\begin{claim}\label{2D*2D claim 2}
If  $D_{n_2}$ is attached to outermost arcs then some string of some tangle is unknotted.
\end{claim}
\textit{Proof of Claim \ref{2D*2D claim 2}.}
Assume that both $D_1$ and $D_{n_2}$ have outermost arcs attached, denoted by $\delta_1$ and $\delta_{n_2}$ resp.. Let the outermost disk co-bounded by $\delta_1$ (resp., $\delta_{n_2}$) be denoted by $\Delta_1$ (resp., $\Delta_{n_2}$) and let $O_1$ (resp., $O_{n_2}$) be the disk in $S-int (D_1\cup D_{n_2})$ separated by $\delta_1$ (resp., $\delta_{n_2}$). By adding a regular neighborhood of $O_{n_2}\cup \Delta_{n_2}$ and $O_1\cup \Delta_1$ to $T$, and the ball bounded by the boundary component that is disjoint from $S$, we define a ball $Q$. If $Q$ contains both strings of $\mathcal{T}_1$, from Lemma \ref{no ball containing T}, we have that the tangle $(Q, Q\cap K)$ is trivial. If $\mathcal{D}^*$ is in $O_1\cup O_{n_2}$ then we get a contradiction between Lemma \ref{inner ball}(c)  and Lemma \ref{no ball containing T}. Then, $O_1$ or $O_{n_2}$ intersects $\mathcal{D}^*$ at a single disk.\\ 
If $n_2=2$, $O_1$ and $O_2$ are disjoint, and $\delta_1$ or $\delta_{n_2}$ are as in Lemma \ref{meridional outermost arc}, which means that some string of some tangle is unknotted.\\
Assume $n_2\geq 4$.\\
Suppose that $O_1\cup O_{n_2}$ intersect $\mathcal{D}^*$ in three disks and, without loss of generality, that $O_1$ intersects $\mathcal{D}^*$ at a single disk. If there is any arc of $E\cap P$ parallel to $\delta_{n_2}$ in $E$, the respective ball $R_{n_2-1}$ contains two strings, which is impossible as observed in Lemma \ref{parallel arcs}. Then, there is a sequence of parallel arcs of $E\cap P$ in $E$, $\delta_1,\ldots, \delta_{n_2-1}$ and we can consider the respective balls $R_1, \ldots, R_{n_2-2}$. If $O_1\cap \mathcal{D}^*$ is $D_2^*$ or $D_3^*$, then $R_2$ contains two strings, which is impossible. Then, $O_{n_2}\cap \mathcal{D}^*$ is $D_2^*\cup D_3^*$. The string $s_{23}$ is trivial in $Q$ and has ends in the same disk component of $Q\cap S$, then $s_{23}$ is unknotted in $(B_1, \mathcal{T}_1)$.\\
Suppose that $O_1\cup O_{n_2}$ intersect $\mathcal{D}^*$ in two disks. Assume $D_2^*\cup D_3^*$ is in $O_1\cup O_{n_2}$. If there are two consecutive balls $R_i$ after $O_1$ or $O_{n_2}$, then some ball $R_i$ contains two strings, which is impossible. Then, $n_2=4$ and both $\delta_1$ and $\delta_{4}$ have a parallel arc of $E\cap P$ in $E$, $\delta_2$ and $\delta_3$, resp., from where we define the balls $R_1$ and $R_3$. As $R_1$ and $R_3$ have a single string, we have that $O_1$ is disjoint from $D_2$, $D_3$ and $D_4$. Hence, $\delta_1$ is as in Lemma \ref{meridional outermost arc} and some string of some tangle is unknotted. Assume now that $D_1^*\cup D_4^*$ is in $O_1\cup O_{n_2}$.  If there is a sequence of parallel arcs to $\delta_1$ (or to $\delta_{n_2}$) attached to all disks $D_1, \ldots, D_{n_2}$, as in Figure \ref{2D1doutermostarcsSI.pdf}(a), following the same argument as in Claim \ref{2D*2D claim 1}, we prove that some string of some tangle is unknotted. Otherwise, the sequences of parallel arcs from $\delta_1$ go up to some arc $\delta_i$ and from $\delta_{n_2}$ go up to some arc $\delta_{i+1}$. If $n_2=4$, then as before $\delta_1$ or $\delta_{n_2}$ are as in Lemma \ref{meridional outermost arc}, which means that some string of some tangle is unknotted. Suppose $n_2\geq 6$. Then, again using arguments as in Claim \ref{2D*2D claim 1}, if $O_1$ intersects $\mathcal{D}$, it is in $D_i$ or $D_{i+1}$. From the sequences of parallel arcs we can consider the respective balls $R_1, \ldots, R_{i-1}$ and $R_{i+1},\dots, R_{n_2-1}$. Denote by $C_{k, k+1}$ the cylinder in $V$ between $D_k$ and $D_{k+1}$. If $D_i$ and $D_{i+1}$ are not in $O_1$ then $\delta_1$ resp., is as in Lemma \ref{meridional outermost arc}, which means that some string of some tangle is unknotted. Otherwise, without loss of generality, suppose that $D_i$ is in $O_1$. Then, then as $R_{i-1}$ cannot be in Q, we have that $C_{i, i+1}$ is in $Q$. Each string of the tangles defined by $S$ is in $Q$ or is some ball $R_k$. Following as in the previous claim, consider $Q$ union with $R_k\cup C_{k, k+1}$, for $k=1, \ldots, i-1, i+1, \ldots, n_2$, we define a solid torus that is a neighborhood of $K$, containing $V$, and with boundary essential in $W$, which is a contradiction to $W$ being a handlebody.
\hspace*{\fill}$\triangle$\\
\end{proof}


\begin{lem}\label{2D*0D}
Suppose $V-V\cap S$ contains a solid torus component disjoint from $\mathcal{D}$ and intersecting $\mathcal{D}^*$ at two disks. Then both strings of some tangle are $\mu$-primitive.
\end{lem}
\begin{proof}
Let $T$ be a solid torus component as in the statement and suppose $\mathcal{D}^*\cap T=D_1^*\cup D_4^*$. Assume that $T$ is in the tangle $(B_1, \mathcal{T}_1)$. From Remark \ref{remark no beta in ball}, all outermost disks are over solid torus components of $V-V\cap S$. Suppose some outermost disk is attached to some disk of $\mathcal{D}$. As the genus of $V$ is three, this outermost disk is over a solid torus disjoint from $K$ intersecting $S\cap V$ at a single disk, which is a contradiction to Lemma \ref{no beta in simple torus}. Then all outermost disks are attached to disks of $\mathcal{D}^*$.

\begin{claim}\label{2D*0D claim 1}
If the disks of $\mathcal{D}^*$ are parallel two-by-two then some string in some tangle is unknotted.
\end{claim}
\textit{Proof of Claim \ref{2D*0D claim 1}.}
Suppose only one disk or two non-parallel disks are adjacent to outermost arcs. By the finiteness of outermost arcs we have parallel sk-arcs, as in Figure \ref{T2outermostarcs.pdf}(a1), (a2), and we get a contradiction to Corollary \ref{no parallel sk-arcs}.\\
Otherwise, we are left with the case when the outermost arcs are only adjacent to two parallel disks of $\mathcal{D}^*$. Following an argument of a similar situation in Lemma \ref{4D*}, we have that some string in some tangle is unknotted.
\hspace*{\fill}$\triangle$\\

\begin{claim}\label{2D*0D claim 2}
If the disks of $\mathcal{D}^*$ are not parallel two-by-two then both strings of some tangle are $\mu$-primitive.
\end{claim}
\textit{Proof of Claim \ref{2D*0D claim 2}.}
Assume, without loss of generality, that no other disk of $S\cap V$ is parallel to $D_1^*$. If there are disks (of $\mathcal{D}^*$) parallel to $D_4^*$, and $D_4^*$ or one disk parallel to it are the only disks with outermost arcs attached, then we get a contradiction to Corollary \ref{no parallel sk-arcs}. So, without loss of generality, assume there is some outermost arc attached to $D_1^*$, and that it is over $T$. Under these conditions, we define a ball $Q$ as in Lemma \ref{T with two D*}, using the outermost disk attached to $D_1^*$ over $T$. From Lemma \ref{T with two D*}, the tangle $(Q, Q\cap K)$ is the product tangle with respect to $Q\cap S$. So, we can isotope $S$ through $Q$, and we replace $D_1^*$ with a disk parallel to $D_4^*$. So, if $n_2=0$ we reduce this case to either the case when there is a genus two component, as in Lemma \ref{genus 2 component} or to the case when $V-V\cap S$ contains a solid torus component intersecting $\mathcal{D}^*$ at the four disks as in Lemma \ref{4D*}. If $n_2>0$ we also reduce to the cases when $V-V\cap S$ contains a solid torus component intersecting $\mathcal{D}^*$ in a collection of two disks and $\mathcal{D}$ in one or two disks, as in Lemmas \ref{2D*1DS}, \ref{2D*1DNS} and \ref{2D*2D}. From these lemmas we get that both strings in some tangle are $\mu$-primitive. 
\hspace*{\fill}$\triangle$\\
\end{proof}


\begin{lem}\label{0D*}
Suppose $V-V\cap S$ contains a solid torus component disjoint from $K$. Then both strings of some tangle are $\mu$-primitive.
\end{lem}
\begin{proof}
As the genus of $V$ is three and no disk of $\mathcal{D}^*$ is parallel to a disk of $\mathcal{D}$, $\mathcal{D}\cap T$ is a collection of at most three disks. If there is some solid torus component of $V-V\cap S$ intersecting $\mathcal{D}^*$ we follow as in the previous lemmas to get the conclusion that two strings of some tangle are $\mu$-primitive. Otherwise, the solid torus components of $V-V\cap S$ are disjoint from $K$. From Remark \ref{remark no beta in ball}, without loss of generality, we can assume that some outermost disk is over $T$.\\
If $\mathcal{D}\cap T$ is a single disk we get a contradiction to Lemma \ref{no beta in simple torus}. Then, we have that $T$ intersects $\mathcal{D}$ at more than one disk, in which case $T$ is the only solid torus component of $V-V\cap S$ and all outermost disks are over $T$.\\
Assume that $\mathcal{D}\cap T$ is a collection of two disks, $D_1$ and $D_1'$.
\begin{figure}[htbp]
\centering
\includegraphics[width=\textwidth]{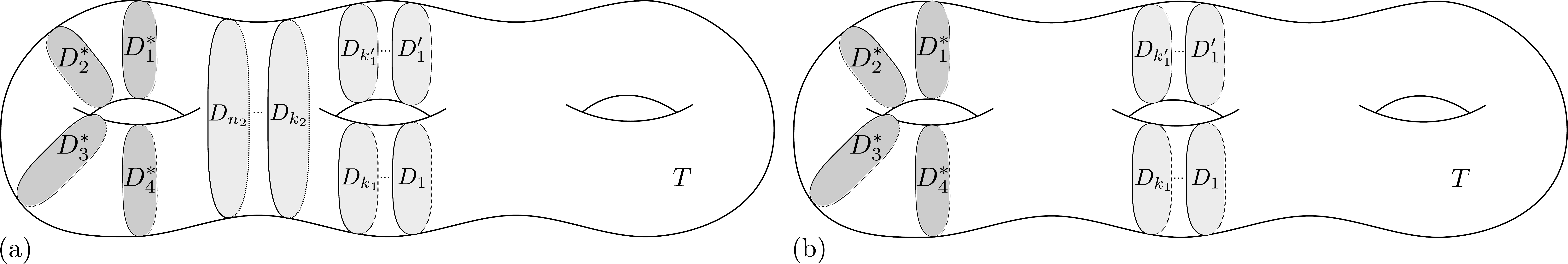}
\caption{}                 
\label{2d.pdf}
\end{figure}
The outermost arcs are attached to $D_1$ or to $D_1'$, with outermost disks over $T$. Let $D_2, D_3, \ldots, D_k$ be the disks of $\mathcal{D}$ parallel to $D_1$ in $V$, in case there exists such a sequence. Without loss of generality, assume there is an outermost arc $\delta_1$ attached to $D_1$, and that there is a sequence of arcs, $\delta_i$, after an outermost arc, $\delta_1$, as in Figure \ref{2D1doutermostarcsSI.pdf}(a). Let $\Delta$ be the outermost disk bounded by $\delta_1$, in $E$, and also, $\Delta_i$ be the disk of $E-E\cap S$ between $\delta_i$ and $\delta_{i+1}$. As $S^3$ has no lens space summand, we have that $\partial \Delta$ intersects a meridian of $T$ geometrically once. So, we can perform an isotopy of  the annulus in $S$, $A=D_1\cup (S\cap N(\Delta))$ through $N(\Delta)$ to the annulus $A'=D_1\cup (\partial T\cap N(\Delta))$. As $\partial \Delta_1$ intersects a meridian of $T$ geometrically once, we isotope $A'$ through $T$ to a disk in $T$ parallel to $D_1'$, that we also denote by $D_1$. Using the disk $\Delta_1\cup \Delta$ we can perform a similar isotopy, and from the disk $D_2$ of $E\cap S$ we get a disk in $T$ parallel to the new disk $D_1$ (c.f. Morimoto \cite[Lemma 3.3]{Morimoto2}). In this way we can perform a sequence of isotopies of $S$  to get from the disks $D_1, D_2, \ldots, D_{k}$ new disks in $T$ parallel to $D_1'$. With this isotopy we reduce this case to other cases:
If the disks of $\mathcal{D}^*$ are not parallel in $V$ we can reduce this case to the case when $T\cap \mathcal{D}^*$ is a collection of two disks and $T\cap \mathcal{D}$ is one non-separating disk, as in Lemma \ref{2D*1DNS}. So, we are left with the situation when the disk components of $\mathcal{D}^*$ are parallel. The disk components of $\mathcal{D}$ in $V$ can be parallel to $D_1$ or to $D_1'$, or can be separating. Assume there is a disk of $\mathcal{D}$ that is separating in $V$, as in Figure \ref{2d.pdf}(a). By the previous isotopy we reduce this case to the case, considered next, when $\mathcal{D}^*\cap T$ is empty and $\mathcal{D}\cap T$ is a collection of three disks. Otherwise, suppose that no disk of $\mathcal{D}$ is separating, as in Figure \ref{2d.pdf}(b). Similarly, we reduce this case to the case when $\mathcal{D}^*\cap T$ and $\mathcal{D}\cap S$ is a collection of two disks, as in Lemma \ref{2D*2D}.\\
At last, suppose that $\mathcal{D}\cap T$ is a collection of three disks. 
\begin{figure}[htbp]
\centering
\includegraphics[width=0.5\textwidth]{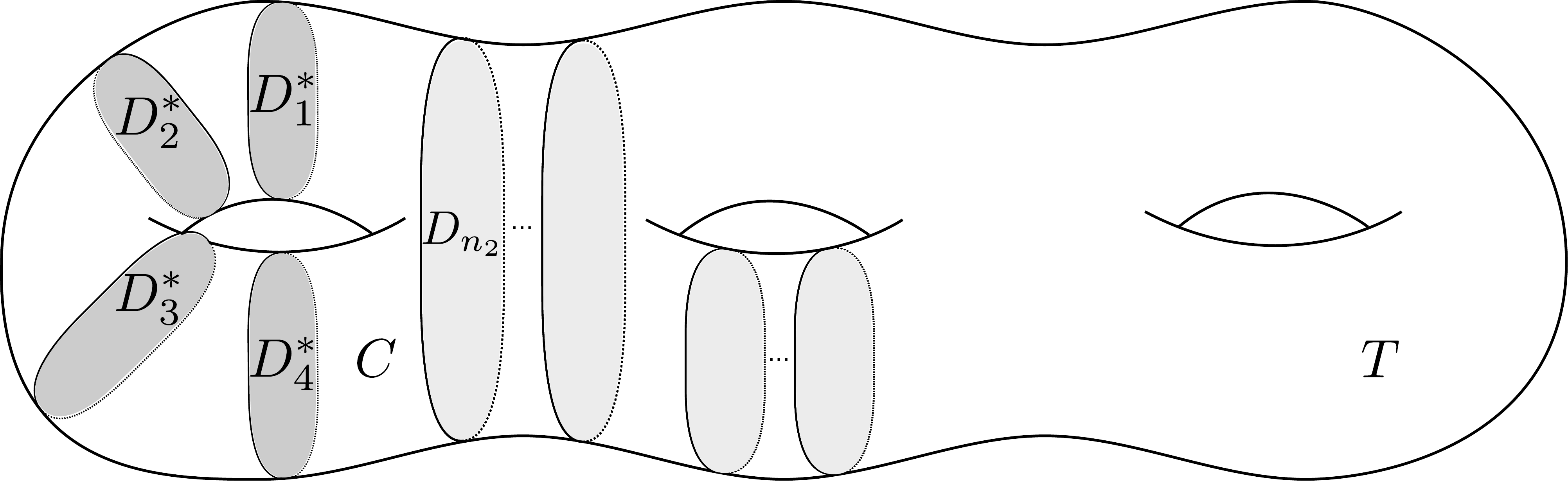}
\caption{}                 
\label{3d.pdf}
\end{figure}
We have a collection of parallel non-separating disks of $\mathcal{D}$, and a collection of separating disks of $\mathcal{D}$ in $V$, as in Figure \ref{3d.pdf}. As in Lemma \ref{no gamma with equal ends}, let $C$ be the ball component of $V-V\cap S$ cut from $V$ by $D_1^*\cup D_4^*\cup D_{n_2}$. Every outermost arc of $E\cap S$ in $E$ attached to $D_{n_2}$ has both ends attached to it (otherwise, it would be a t-arc, which don't exist from Lemma \ref{property}(c)). By the finiteness of outermost arcs, we consider an outermost arc $\gamma$ after outermost arcs with both ends in $D_{n_2}$. From Lemma \ref{no gamma with equal ends}, if $\gamma$ has at least one end in $D_{n_2}$ or one end in $D_1^*$ and the other in $D_4^*$, some string of some tangle is unknotted. In case, all arcs $\gamma$ have both ends in $D_1^*$ or in $D_4^*$, by the finiteness of outermost arcs, we have a contradiction to Corollary \ref{no parallel sk-arcs}.
\end{proof}

%% file: 2-string_tangle.tex
For the proof of Theorem \ref{2-string tangle}, we study all cases of $S\cap V$ with respect to the value $n_1$.

\begin{prop}\label{n_1=1}
If $n_1=1$ then both strings of some tangle are $\mu$-primitive.
\end{prop} 
\begin{proof}
Suppose $n_1=1$. If $n_2>0$ we have a contradiction between Lemma \ref{property}(b) and (f). So, $n_2=0$, $P$ is a disk and $|P\cap E|=0$.\\
Let $D_1^*=S\cap V$. 
\begin{figure}[htbp]
\centering
\includegraphics[width=0.5\textwidth]{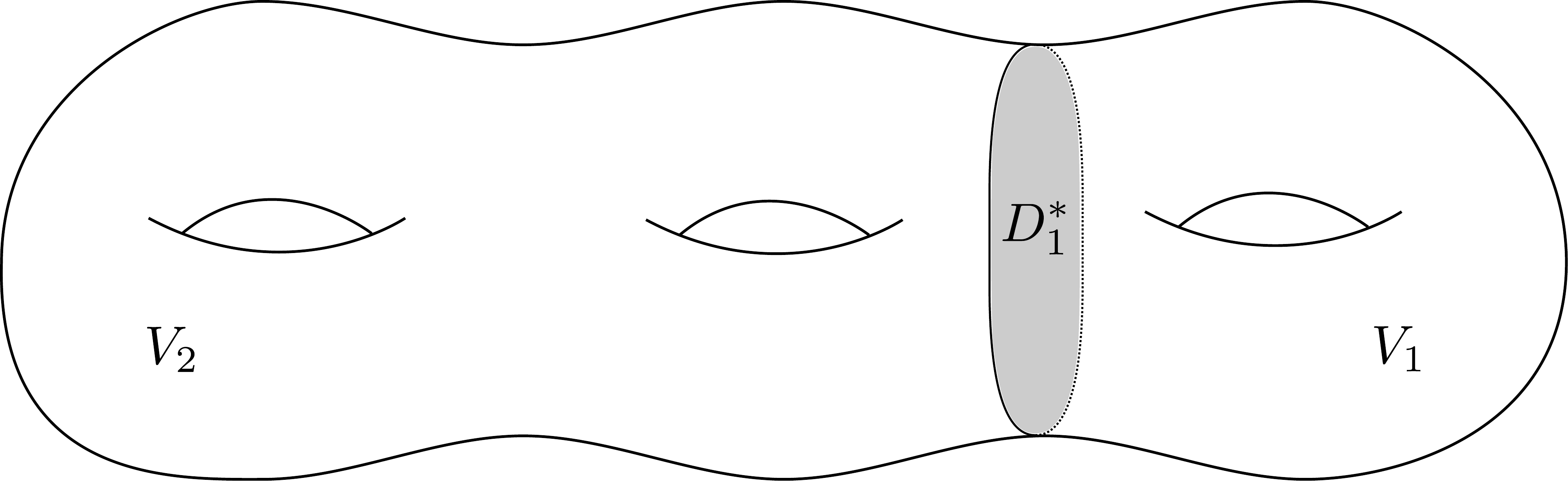}
\caption{}                 
\label{oneD.pdf}
\end{figure}
The $2$-sphere $S=D_1^*\cup P$ is separating, then $D_1^*$ is a separating disk in $V$. As the handlebody $V$ has genus three, the disk $D_1^*$ separates $V$ in a solid torus $V_1$ and a genus two handlebody $V_2$, as in Figure \ref{oneD.pdf}. Let $(B_1, T_1)$ denote the tangle containing $V_1$. The strings of this tangle lie in $V_1$, have end points in $D_1^*$ and, by Lemma \ref{strings parallel to boundary}, are simultaneously parallel to $\partial V_1$. Also, the complement of $V_1$ in $B_1$ is a torus. Hence, from Lemma \ref{mu-primitive characterization}, both strings of the tangle $(B_1, T_1)$ are $\mu$-primitive.
\end{proof}

\begin{prop}\label{n_1=2}
$n_1\neq2$.
\end{prop}
\begin{proof} Suppose $n_1=2$. We denote by $D_1^*$ and $D_2^*$ the components of $\mathcal{D}^*$. From Lemma \ref{property}(b), (g) $n_2>0$ and every outermost arc is a st-arc.\\
\begin{figure}[htbp]
\centering
\includegraphics[width=\textwidth]{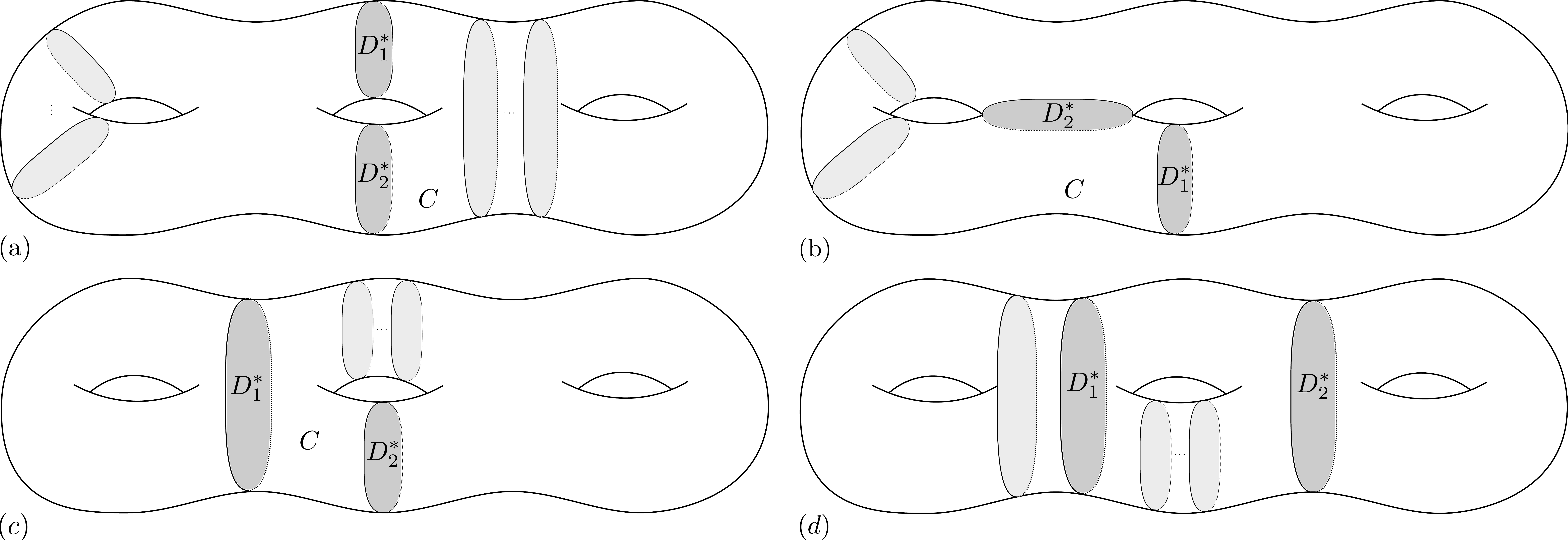}
\caption{}               
\label{twoDE.pdf}
\end{figure}

\textbf{Claim.} \textit{If $n_1=2$ there is no ball $C$ of $V-V\cap S$ containing strings of a tangle.}\\
\textit{Proof of Claim.}
Suppose that there is a ball component of $V-V\cap S$, $C$, containing strings of a tangle.\\
Suppose, the ball $C$ contains two strings. From Lemma \ref{strings parallel to boundary}, the strings are parallel to $\partial C$. Therefore, the tangle $(C, C\cap K)$ is trivial, which is a contradiction to Lemma \ref{inner ball}(c).\\
Otherwise, suppose that $C$ contains a single string. As $D_1^*\cup D_2^*$ intersects $K$ in four points only one of these disks can be in $\partial C$, and both ends of the string in $C$ are in this disk. Then, this string is trivial in $\partial C$. Furthermore, as this is the only string in $C$ it is also trivial in the respective tangle, which is a contradiction to the tangle decomposition defined by $S$ being essential.
\hspace*{\fill}$\triangle$\\

If $D_1^*$ and $D_2^*$ are parallel in $V$ then the ball component of $V-S\cap V$ cut by $D_1^*\cup D_2^*$ is in contradiction to Claim.\\
Suppose now that $D_1^*$ and $D_2^*$ are not parallel, as in the examples of Figure \ref{twoDE.pdf}. Then, the components of $V-D_1^*\cup D_2^*$ are solid tori. As $n_2>0$, the disks of $\mathcal{D}$ are in some of these solid tori. Then, some ball component of $V-V\cap S$ contains $D_1^*$, $D_2^*$, or both, which is a contradiction to Claim.
\end{proof}

\begin{prop}\label{n_1=3}
If $n_1=3$ then both strings of some tangle are $\mu$-primitive.
\end{prop}
\begin{proof}
Consider the components of $V-V\cap S$. From Remark \ref{remark no beta in ball} we can assume that some component of $V-S\cap V$ is not a ball.\\
If there is a genus two component of $V-V\cap S$ then, by Lemma \ref{genus 2 component n_1=3} , some string of some tangle is unknotted.
Otherwise, there is some solid torus component of $V-V\cap S$, and from Lemma \ref{solid torus component} two strings of some tangle are $\mu$-primitive.
\end{proof}

\begin{prop}\label{n_1=4}
If $n_1=4$ then both strings of some tangle are $\mu$-primitive.
\end{prop}
\begin{proof}
As in Proposition \ref{n_1=3}, we consider the components of $V-V\cap S$ and we assume that some component of $V-S\cap V$ is not a ball.\\
If $V-V\cap S$ has a genus two component then, by Lemma \ref{genus 2 component}, some string of some tangle is unknotted.\\
Now, assume that $V-V\cap S$ has no genus two component. This means at least one of its components is a solid torus, $T$. 
The collection of disks $\mathcal{D}^*\cap T$ is always even, because $\partial T$ is a separating torus in $S^3$. We consider several cases with respect to $\mathcal{D}^*\cap T$.\\
If $\mathcal{D}^*\subset T$, from Lemma \ref{4D*}, some string of some tangle is unknotted.\\
Suppose $\mathcal{D}^*\cap T$ is a collection of two disks. As the genus of $V$ is three, there are at most two disks of $\mathcal{D}$ in $\partial T$. Then we are under Lemmas \ref{2D*1DS}, \ref{2D*1DNS}, \ref{2D*2D} and \ref{2D*0D}, and we have that both strings of some tangle are $\mu$-primitive.\\
At last, suppose $\mathcal{D}^*\cap T=\emptyset$. From Lemma \ref{0D*}, we also have that both strings of some tangle are $\mu$-primitive. 
\end{proof}

We can now prove Theorem \ref{2-string tangle} and its Corollary \ref{t(K)=3}.

\begin{proof}[of Theorem \ref{2-string tangle}]
If $K$ has an inessential $2$-string free tangle decomposition then the tunnel number of $K$ is one. This is a contradiction with the assumption that the tunnel number of $K$ is two. Hence, the $2$-string free tangle decomposition of $K$ is essential.\\
We have that $0\leq n_1\leq 4$. If $n_1=0$ then, as $S\cap K\subset S\cap V$ we have $n_2=0$. Hence, $S\subset V$ which is a contradiction to Lemma \ref{strings parallel to boundary}(a). In case $n_1\neq 0$, from Propositions \ref{n_1=1}, \ref{n_1=2}, \ref{n_1=3} and \ref{n_1=4}, we have that two strings of some tangle are $\mu$-primitive.
\end{proof}

\begin{proof}[of Corollary \ref{t(K)=3}]
Let $K$ be a knot with a $2$-string free tangle decomposition where at least a string of each tangle is not $\mu$-primitive.\\
From Corollary 2.4 in \cite{Morimoto3} by Morimoto, if a knot $K$ has a $n$-string free tangle decomposition, then $t(K)\leq 2n-1$. Hence, in this case $t(K)\leq 3$.\\
On the other hand, as no tangle in the decomposition of $K$ has both strings being $\mu$-primitive, from Theorem \ref{2-string tangle} we have $t(K)\geq 3$.\\
Altogether, from the two inequalities, $t(K)=3$.
\end{proof}

%% file: tunnel_number_degeneration.tex
In this section, we construct an infinite class of knots with a $2$-string free tangle decomposition where no tangle has both strings being $\mu$-primitive. With these collection of knots, Theorem \ref{2-string tangle} and the work of Morimoto \cite{Morimoto3} we prove Theorem \ref{counter-example Moriah conjecture}.\\
A particular, simplified, version of Theorem 3.4 in \cite{Morimoto3} by Morimoto gives us the following proposition, which is relevant to the proof of Theorem \ref{counter-example Moriah conjecture}.

\begin{prop}[\cite{Morimoto3}, Morimoto] \label{Morimoto result}
Let $K_1$ be a knot which has a $2$-string free tangle decomposition and $K_2$ be a knot with a $3$-bridge decomposition. Then $t(K_1\#K_2)\leq 3$.
\end{prop}
\begin{figure}[htbp]
\centering\hspace{2cm}
\includegraphics[width=0.5\textwidth]{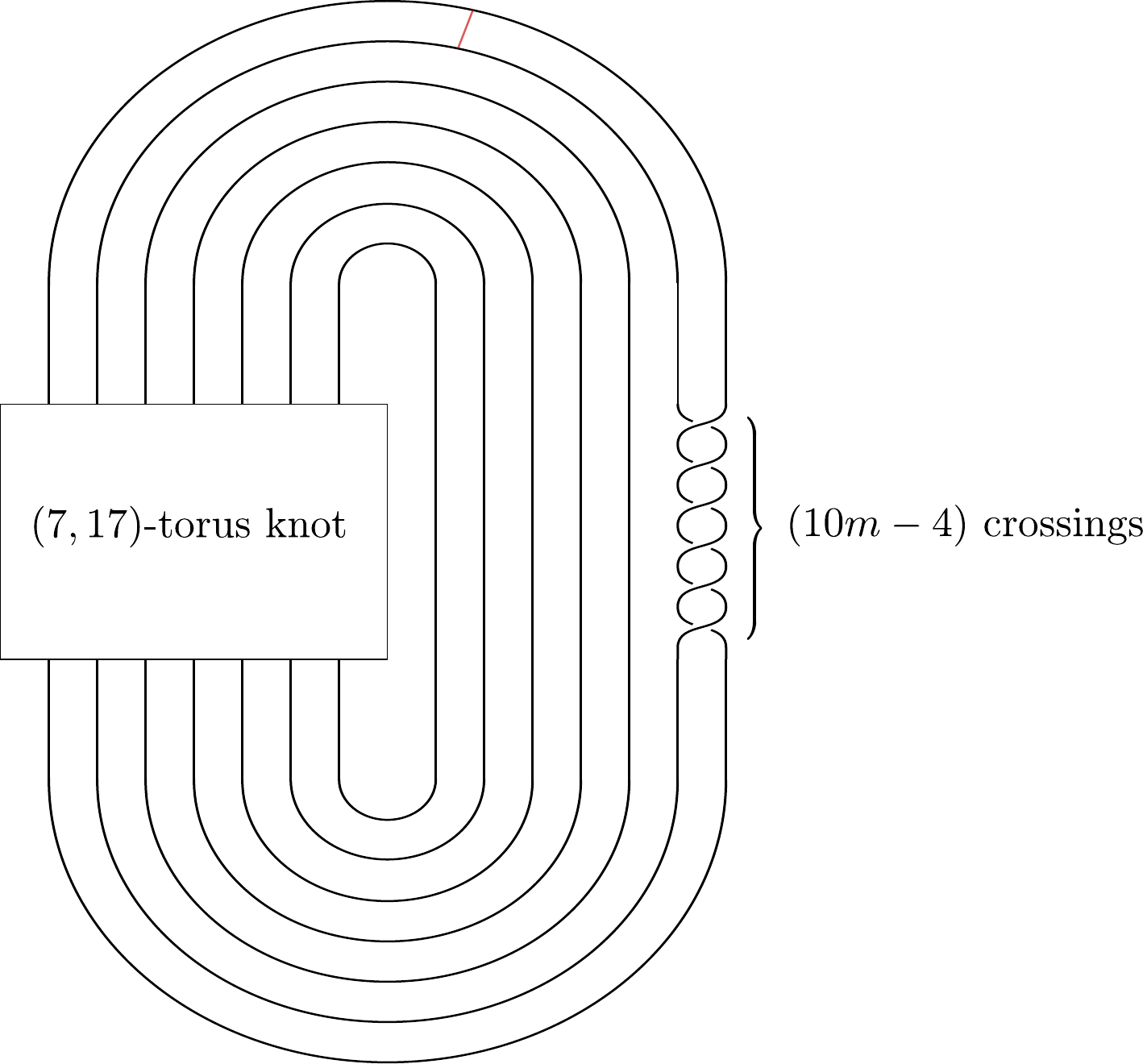}
\caption{: The knot $K(m)$ and one unknotting tunnel, with $m$ a natural number.}                 
\label{knotsK(m).pdf}
\end{figure}
For the construction of knots, $K_1$, as in Theorem \ref{counter-example Moriah conjecture} we consider $2$-string free tangle decompositions. Suppose there are two $2$-string free tangles $(B_1, \mathcal{T}_1)$ and $(B_2, \mathcal{T}_2)$ where one of the strings in each tangle is not $\mu$-primitive. Identify $(\partial B_1, \partial \mathcal{T}_1)$ to $(\partial B_2, \partial \mathcal{T}_2)$, such that no string of $\mathcal{T}_1$ has its end identified to the ends of the same string of $\mathcal{T}_2$. Then $(B_1\cup B_2, \mathcal{T}_1\cup \mathcal{T}_2)$ is a knot $(S^3, K_1)$ under the conditions of Proposition \ref{Morimoto result}. Furthermore, from Corollary \ref{t(K)=3}, $t(K_1)=3$. Hence, this procedure gives us a knot as in the statement of Theorem \ref{counter-example Moriah conjecture}.\\
So, we need $2$-string free tangles with one of the strings not $\mu$-primitive. As observed in Remark \ref{rem 1}, if a string $s$ properly embedded in a ball $B$ is $\mu$-primitive, then by capping $s$ along $\partial B$ we get a $\mu$-primitive knot. Then, for the construction of a $2$-string free tangle where at least one of the strings is not $\mu$-primitive we consider a tunnel number one knot $K$ that is not $\mu$-primitive, and one of its unknotting tunnels. For such a knot $K$, let $s$ be a string in a ball $B$, that when capped off along $\partial B$ we obtain $K$, together with one unknotting tunnel for $K$. If we slide the ends of the unknotting tunnel from $s$ to $\partial B$ we get an essential $2$-string free tangle where one of the strings is not $\mu$-primitive.\\
Then, we want tunnel number one knots that are not $\mu$-primitive. Existence results of such knots are known by work Johnson and Thompson in \cite{Johnson-Thompson} and also Moriah and Rubinstein in \cite{Moriah-Rubinstein}. On the other hand, explicit or constructive examples of knots with tunnel number one that are not $\mu$-primitive is given by work Eudave-Mu\~{n}oz in \cite{Munoz} and \cite{Munoz2}, Ram\'{i}rez-Losada and Valdez-S\'{a}nchez in \cite{Losada-Sanchez}, Minsky, Moriah and Schleimer in \cite{Minsky-Moriah-Schleimer} and also Morimoto, Sakuma and Yokota in \cite{Morimoto-Sakuma-Yokota}. With any of these examples it is possible to construct knots as in the statement of Theorem \ref{counter-example Moriah conjecture}. As an example of such construction we consider the class of knots $K(7, 17; 10m-4)$ from \cite{Morimoto-Sakuma-Yokota}, where $m$ is an integer, together with an unknotting tunnel. We denote these knots by $K(m)$, as in Figure \ref{knotsK(m).pdf}.\\
As previously described, from the knots $K(m)$ and an unknotting tunnel we construct tangles $T(m)$ where at least one of the strings is not $\mu$-primitive, as in Figure \ref{tangle(m)A.pdf}. 
\begin{figure}[htbp]
\centering
\includegraphics[width=0.75\textwidth]{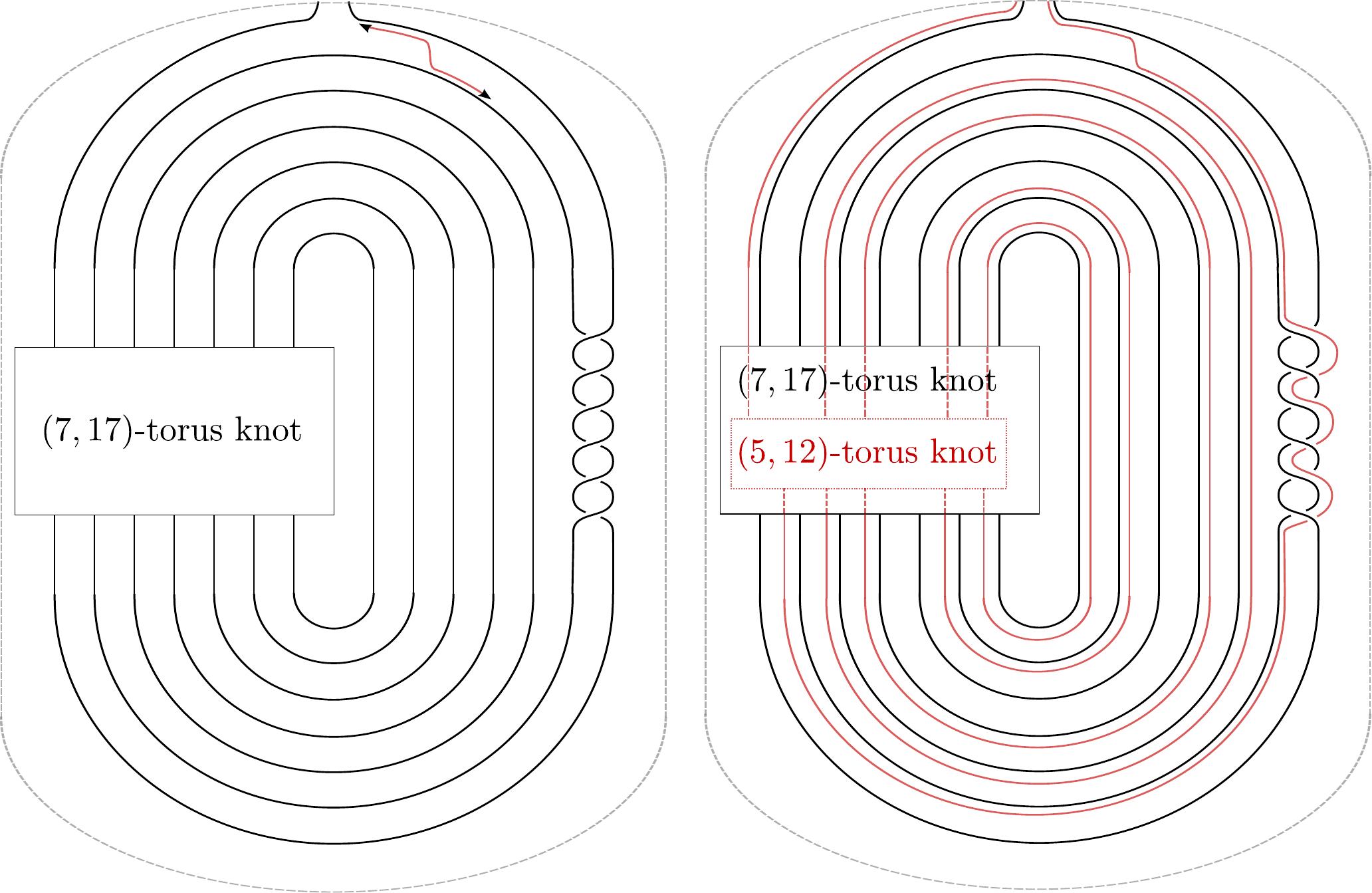}
\caption{: A possible construction of a tangle $T(m)$ from the knot $K(m)$ and one of its unknotting tunnels.}                 
\label{tangle(m)A.pdf}
\end{figure}
From the construction we have that the tangles $T(m)$ are free. With the tangles $T(m)$ and $T(m')$ we construct a knot $K_1(m, m')$, as explained before, that has a $2$-string free tangle decomposition where no tangle has both strings being $\mu$-primitive. With this construction we can now prove Theorem \ref{counter-example Moriah conjecture} and its Corollary \ref{tunnel number degeneration}.

\begin{proof}[of Theorem \ref{counter-example Moriah conjecture}]
Consider the collection of knots $\{K_1(m, m'): m, m'\in \mathbb{N}, m\leq m'\}$.  From Corollary \ref{t(K)=3}, we have that $t(K_1(m, m'))=3$. From Ozawa's unicity theorem, the knot $K_1(m, m')$ is prime. And, from Proposition \ref{Morimoto result}, for any $3$-bridge knot $K_2$, $t(K_1(m, m')\#K_2)\leq 3$.
\end{proof}

\begin{proof}[of Corollary \ref{tunnel number degeneration}]
Consider the collection of knots $\{K_1(m, m'): m, m'\in \mathbb{N}, m\leq m'\}$. Let $K_2$ be any $3$-bridge prime knot with tunnel number two. From Proposition \ref{Morimoto result}, $t(K_1(m, m')\#K_2)\leq 3$. From tunnel number one knots being prime and the main theorem in \cite{Morimoto2}, we also have that $t(K_1(m, m')\#K_2)\geq 3$. Then, $t(K_1(m, m')\#K_2)=3=t(K_1(m, m'))+t(K_2)-2$.
\end{proof}